\documentclass[12pt]{amsart}
\usepackage{a4wide}
\usepackage{amsmath,amssymb,amsfonts}
\usepackage{amsthm}
\usepackage{xcolor}
\usepackage{graphicx}
\usepackage{subfig}
\usepackage{algorithm}
\usepackage{algpseudocode}
\usepackage{array,tabularx}
\usepackage[hidelinks]{hyperref}
\usepackage{fix-cm}
\AddToHook{env/proof/end}{\nopagebreak[4]}
\newtheorem{theorem}{Theorem}
\newtheorem{lemma}{Lemma}
\newtheorem{proposition}{Proposition}

\newtheorem{assumption}{Assumption}[section]

\theoremstyle{definition}
\newtheorem{definition}{Definition}
\theoremstyle{remark}
\newtheorem{remark}{Remark}
\newtheorem{example}{Example}

\newcommand{\E}{\mathbb{E}}

\newcommand{\R}{\mathbb{R}}

\newcommand{\CE}{\mathcal{E}}

\newcommand{\CP}{\mathcal{P}}

\newcommand{\ga}{\gamma}

\newcommand{\con}{v_{\alpha}\left(\rho_{t}\right)}

\newcommand{\ener}{\mathcal{V}\left(\rho_{t}\right)}
\newcommand{\rt}{\rho_{t}}
\newcommand{\rti}{\rho_{\ti}}
\newcommand{\ti}{T_{\alpha,\epsilon}}
\newcommand{\dvt}{\,d\rho_{t}(v)}
\newcommand{\dvti}{\,d\rho_{\ti}(v)}
\newcommand{\wt}{\omega_{\alpha}}

\newcommand{\ec}{\mathcal{E}^{0}_{r}}
\newcommand{\er}{\mathcal{E}_{r}}

\newcommand{\dH}{\,dH_{d-1}(v)}
\newcommand{\argmin}{\operatornamewithlimits{argmin}}

\newcommand{\vit}{V^{i,N}_{\tau}}
\newcommand{\viti}{V^{i,N}_{t}}
\newcommand{\emconst}{v_{\alpha}\left(\hat{\rho}^{N}_{\tau}\right)}
\newcommand{\dempt}{\hat{\rho}_{\tau}^{N}}

\newcommand{\xs}{v^*}
\newcommand{\bxs}{v_{\alpha}\left(\hat{\rho}\right)}

\newcommand{\Xj}{X^j}

\newcommand{\nb}{\nabla}

\newcommand{\ds}{\displaystyle}

\newcommand{\qd}{\quad}

\newcommand{\h}{\hat}

\newcommand{\rl}{\right\rVert}

\def\mE{\mathcal{E}}

\def\mN{\mathcal{N}}

\def\l{\left}
\def\r{\right}

\def\h{\hat}
\def\t{\tilde}

\def\R{\mathbb{R}}
\def\E{\mathbb{E}}

\def\mN{\mathcal{N}}

\def\g{\gamma}
\def\lam{\lambda}

\def\t_i{{t_i}}

\def\g{\gamma}

\def\s{\sigma}

\def\l{\left}
\def\r{\right}

\def\ll{\left\lVert}
\def\rl{\right\rVert}

\def\Xj{V^j}

\def\hv{\h{v}}

\def\e{\epsilon}
\def\a{\alpha}
\def\mE{\mathcal{E}}
\def\t{\tilde}

\newcommand{\dist}{\operatorname{dist}}

\newcommand{\HZnew}[1]{{#1}}

\title[Consensus method for constrained optimization]{An interacting particle consensus method for constrained global optimization}
\author[J. A. Carrillo]{Jos\'e A. Carrillo\textsuperscript{1}}
\thanks{\textsuperscript{1} Mathematical Institute, University of Oxford, United Kingdom. (carrillo@maths.ox.ac.uk)}
\address{(JAC) Mathematical Institute, University of Oxford, United Kingdom.}
\email{carrillo@maths.ox.ac.uk}
\author[S. Jin]{Shi Jin\textsuperscript{2}}
\thanks{\textsuperscript{2} School of Mathematical Sciences, Institute of Natural Sciences, MOE-LSC, Shanghai Jiao Tong University, P. R. China. (shijin-m@sjtu.edu.cn)}
\address{(SJ) School of Mathematical Sciences, Institute of Natural Sciences, MOE-LSC, Shanghai Jiao Tong University, P. R. China.}
\email{shijin-m@sjtu.edu.cn}
\author[H. Zhang]{Haoyu Zhang\textsuperscript{3}}
\thanks{\textsuperscript{3} Department of Mathematics, University of California, San Diego, USA. (haz053@ucsd.edu)}
\address{(HZ) Department of Mathematics, University of California, San Diego, USA.}
\email{haz053@ucsd.edu}
\author[Y. Zhu]{Yuhua Zhu\textsuperscript{4}\textsuperscript{$\dagger$}}
\thanks{\textsuperscript{4} Department of Statistics and Data Science, University of California, Los Angeles, USA. (yuhuazhu@ucla.edu)}
\address{(YZ) Department of Statistics and Data Science, University of California, Los Angeles, USA.}
\email{yuhuazhu@ucla.edu}
\thanks{All authors contributed equally to this work and are listed alphabetically.}
\thanks{\textsuperscript{$\dagger$} Corresponding author.}
\date{}
\begin{document}
\begin{abstract}
This paper presents a particle-based optimization method designed for addressing minimization problems with equality constraints, particularly in cases where the loss function exhibits non-differentiability or non-convexity. The proposed method combines components from consensus-based optimization algorithm with a newly introduced forcing term directed at the constraint set. \HZnew{A rigorous mean-field limit is derived for the anisotropic particle system, while convergence to the constrained minimizer is established for the isotropic mean-field dynamics.} \HZnew{Additionally, we introduce a linearly implicit discretized algorithm, analyze its well-posedness, stability properties, and consistency with the continuous particle dynamics, and conduct various numerical experiments to demonstrate the performance of the proposed method.}
\end{abstract}
\maketitle

\section{Introduction}

In this paper, we are concerned with the following minimization problem with equality constraints,
\begin{gather}\label{opt1}
    \min_{v\in\R^d}\CE(v)
    \qquad\text{s.t.}\qquad
    g(v)=0,
\end{gather}
\HZnew{where $g:\R^d\rightarrow\R^m$ represents a collection of $m$ equality constraints.}
The above optimization problems have widespread application across various domains. For example, in supply chain optimization, equality constraints play a pivotal role in maintaining a balance between demand and supply \cite{MNS}; astronomers employ constrained optimization to calculate
spacecraft trajectories, adhering to the laws of physics and orbital equations \cite{CSTCX,YCLG}; in structural design, engineers optimize dimensions of beams, columns, or trusses while ensuring that the
structural equilibrium equations are satisfied as equality constraints \cite{GV}{}.  In this paper we deal with the cases when the objective function can be {\it non-convex} and {\it non-differentiable}. 

Traditional algorithms like the  Lagrange Multipliers \cite{RTy} and the Alternating Direction Method of Multipliers (ADMM) \cite{WO} lack guarantees of converging to the global
constrained minimizer when dealing with non-convex or non-differentiable loss functions $ \mathcal{E}(v) $. A new framework is required to effectively handle such cases, and recently, a class of gradient-free methods called consensus-based optimization (CBO) methods \cite{bailo2024cbx,CCTT,PTTM,TPBS} have emerged as promising approaches for handling non-convex and non-differentiable loss functions. Motivated by the well-known Laplace's principle \cite{BO,DZ,M}, they are decentralized and gradient-free algorithms that leverage the power of information sharing and cooperation among individual particles.  However, it is important to highlight that much of the existing work has focused on the unconstrained case, see \cite{CCTT,CJLZ,CJL,FKR,FRRS,GHV,GP,HJK2,HQ,HQR,RK,RKGF,TW}{}. We refer the readers to survey articles \cite{CTV,GHPQ,T} for a more detailed and complete summary. For the efficient implementation of CBO, we refer to the CBX packages developed for Python and Julia \cite{bailo2024cbx}.

Limited work has been done for the constrained case. The primary challenge lies in reconciling the
CBO model’s tendency to drive agents towards the global minimizer with the need for agents to
remain within the constraint set and converge to the constrained minimizer. Currently, there are
mainly three approaches. One involves projection onto the hypersurface \cite{albers2019ensemble,FHPS2,FHPS1,FHPS3}.  However, this method requires computing the {derivatives of} the distance function $\dist(\Gamma, v)=\inf\{\|v-u\|_2\mid u\in\Gamma\}$ with $\Gamma$ representing the constraint set. Extending this method to handle general multiple equality constraints is not straightforward. In cases where the constraint set $\Gamma$ is complicated, this computation {about} $\dist(\Gamma, v)$ becomes infeasible.  Another method introduces constraints as a penalization term in the objective function \cite{BHP,CTV}, transforming it into an unconstrained problem for CBO. However, the convergence is sensitive to the landscape of the objective function and the penalization constant, which makes it difficult to achieve high accuracy. \HZnew{Moreover, when the constraints are feasible, Problem~\eqref{opt1}
is equivalent to the bilevel problem of minimizing $\mathcal{E}$
over the minimizers of $\|g\|_2^2$. This connects the present problem to consensus-based bilevel optimization methods such as CB2O \cite{CB2O} and SCB2O \cite{chao2026convergence}. Both methods use only evaluations of the lower- and upper-level
objectives.}

In this paper, we introduce a new strategy for constrained CBO along with convergence analysis and numerical experiments. Instead of projecting the particles onto the constraint set or incorporating the constraint violation into the objective function, we augment the classical unconstrained CBO dynamics with a gradient forcing term generated by
\HZnew{$G(v)=\|g(v)\|_2^2=\sum_{i=1}^m g_i(v)^2,$}
which drives the particles toward the constraint set.
\HZnew{Importantly, differentiability of the target function $\mathcal{E}$ is not required. For the continuous-in-time model and its theoretical analysis, only first-order information of $G$ is needed, while the linearly implicit discretization introduced in Section~\ref{sec:numerical} additionally uses second-order information. The proposed approach naturally accommodates multiple equality constraints. Numerical comparisons with projected CBO, Penalized CBO, and CB2O demonstrate favorable performance in the constrained examples considered here.}

\subsection{Contributions}
Our main contributions are threefold. Firstly, we introduce a new CBO-based method for solving constrained optimization problems, with possibly non-convex and non-differentiable objective functions. This method can accommodate a wide range of equality constraints, including the ability to handle multiple constraints concurrently. Secondly, we provide rigorous theoretical guarantees for the continuous-in-time model of the proposed method.
\HZnew{Specifically, we establish the mean-field limit for the anisotropic dynamics and provide a detailed convergence analysis for the isotropic mean-field dynamics, using a new quantitative Laplace Principle that differs from the unconstrained version in \cite{FKR}.} This principle enables us to quantify the contributions of the consensus dynamics and gradient descent, providing a framework for handling constraint sets. See Remark \ref{innovation} for the details.  \HZnew{Thirdly, we present a linearly implicit discretized algorithm designed to approximate the dynamics of the continuous-in-time model efficiently. We provide sufficient conditions under which the resulting update is well defined and establish a one-step consistency estimate with the continuous particle dynamics. In particular, for affine equality constraints, we prove the implicit treatment of the stiff term of order $O(\epsilon^{-1})$ does not impose a stability restriction relating the time step $\gamma$ and $\epsilon$. We further provide numerical comparisons with existing constrained CBO methods, together with sensitivity, ablation, and computational cost studies.}

\subsection{Organization}The paper is structured as follows. Section \ref{sec:dynamics} provides an introduction to the continuous-in-time stochastic differential equations, which serves as the model for the proposed method. Following that, Section \ref{sec:meanfield} studies the well-posedness of the introduced SDEs and explores their mean-field limit. In Section \ref{longtime}, we analyze the convergence properties of the method by establishing the long-time behavior of the \HZnew{isotropic} mean-field limit. This includes demonstrating, under appropriate assumptions, the convergence of the mean-field limit model to the constrained minimizer. Section~\ref{sec:numerical} introduces the numerical scheme,
studies its numerical properties, and presents numerical comparisons on a
range of constrained optimization problems, followed by sensitivity,
ablation, and computational cost studies. Finally, Section \ref{sec:conclusion} offers a comprehensive summary of the findings presented in this paper.

\subsection{Notations}\label{sec:notation}
We use $\mathcal{C}^k_b(\mathbb{R}^d)$ and $\mathcal{C}^k_c(\mathbb{R}^d)$ to denote the space of $k$-times continuous differentiable functions defined on $\mathbb{R}^d$ that are bounded and compactly supported respectively. The space $\mathcal{C}_{*}^2$ is defined as 
\begin{align}\label{def:C2}
	\mathcal{C}^{2}_{*}(\mathbb{R}^{d}):=
	\left\{
	\begin{array}{c}
		\phi\in\mathcal{C}^{2}(\mathbb{R}^{d})|\text{ }|\partial_{x_k}\phi(x)|\leq C(1+|x_{k}|)\\ \text{ and }\\
		\sup_{x\in \mathbb{R}^{d}}|\partial_{x_kx_k}\phi(x)|<\infty
		\text{ for all $k=1,2,...,d$}
	\end{array}
	\right\},
\end{align}
\HZnew{and the space $\mathcal{C}_{*}^{1,1}$ is defined as
\begin{align}\label{def:C11}
	\mathcal{C}^{1,1}_{*}(\mathbb{R}^{d}):=
	\left\{
	\begin{array}{c}
		\phi\in\mathcal{C}^{1}(\mathbb{R}^{d})|\text{ }|\partial_{x_k}\phi(x)|\leq C(1+|x_{k}|)\text{ for all $k=1,2,...,d$}\\ \text{ and }\\
		\nabla \phi\text{ is globally Lipschitz}
	\end{array}
	\right\}.
\end{align}}
When $X$ and $Y$ are topological spaces, we use $\mathcal{C}(X,Y)$ to denote the space of continuous functions mapping from $X$ to $Y$. 
When $X$ is a topological space, $\mathcal{P}(X)$ denotes the space of all the Borel probability measure, which is equipped with the Levy-Prokhorov metric. Given $1\leq p< \infty$, $\mathcal{P}_p(\mathbb{R}^d)$ is the collection of all probability measures on $\mathbb{R}^d$ with finite $p$-th moment, which is equipped with the Wasserstein-$p$ distance, denoted by $W_p(\cdot,\cdot)$. If $\rho$ is a probability measure, $\rho^{\otimes 
	N}$ denotes the probability space obtained by coupling $\rho$ independently $N$ times.

$\|\cdot\|_p $ denotes the usual $ l^{p} $ vector norm in the Euclidean space, $ \|\cdot\|_{L^1({\rho})} $ denotes $ L^{1} $ norm of a function with respect to $\rho$, $ |\cdot| $ denotes the absolute value of a real number, and $\|\cdot\|_{\mathrm{op}}$ denotes the $\ell_2$-induced operator norm. $ B(x,r) $ denotes the closed $l^2$ ball centered at $ x $ with radius $ r .$ $\mathbb{I}_d$ denotes the $d\times d$ identity matrix. When $u$ is a vector, $\text{diag}(u)$ denotes the diagonal matrix with $u$ being the diagonal. When $\phi$ is a function and $\mu$ is a measure, $\langle \phi,\mu\rangle$ denotes the pairing between, i.e., $\int \phi \,d\mu$. When $u$ and $v$ are vectors, $\langle u,v\rangle $ denotes the inner product in the Euclidean space. We use $\Rightarrow$ to denote weak convergence of probability measures.

Throughout this paper, we use the symbols $C$ and $L$ to represent generic positive uniform constants. It is important to note that these constants may take on different values in different sections or parts of this paper.

\section{The dynamics of the constrained consensus-based optimization algorithm}\label{sec:dynamics}

In this section, we carry out the continuous-in-time dynamics of our method. The practical discretized algorithm will be introduced in Section \ref{sec:numerical}.

Consider the following constrained optimization problem,
\begin{gather}
    \min_{v\in\R^d}\CE(v)
    \qquad\text{s.t.}\qquad
    g(v)=0,
\end{gather}
\HZnew{where $
    g=(g_1,\ldots,g_m)^T:\R^d\rightarrow\R^m$ represents a collection of $m$ equality constraints. Define $
    G(v):=\|g(v)\|_2^2
    =\sum_{i=1}^{m}g_i(v)^2.$}
Then Problem~(\ref{opt1}) can be \HZnew{equivalently} reformulated as
\begin{gather}
    \min_{v\in\R^d}\CE(v)
    \qquad\text{s.t.}\qquad
    G(v)=0.
\end{gather}
Our method will be based on this equivalent formulation.

To start with, we take $ N $ particles $ V^{1,N},V^{2,N},...,V^{N,N} $, which are independently sampled from a common initial law $ \rho_{0} $ at initialization. \HZnew{We assume that all the random variables are defined on a filtered
probability space $(\Omega,\mathcal{F},(\mathcal{F}_t)_{t\geq0},\mathbb{P})$. The initial particles
$\{V_0^{i,N}\}_{i=1}^{N}$ are $\mathcal{F}_0$-measurable and
independent of the driving Wiener processes
$\{B_t^{i,N}\}_{i=1}^{N}$, which are independent $d$-dimensional
$(\mathcal{F}_t)$-Brownian motions.} Here we use $ V^{{i,N}}_{t}$ for the location of the $ i $-th particle at time $ t $ and $
		\hat{\rho}_{t}^{N}(v)=\tfrac{1}{N}\sum_{i=1}^{N}\delta_{V^{i,N}_{t}}(v)$ to denote the empirical measure. The goal of the dynamics is to encourage the measure $ \hat{\rho}_{t}^{N} $ to converge to the measure $ \delta_{v^{*}} $, which is the Dirac measure at the solution of the constrained optimization problem (\ref{opt1}). Now we propose the dynamics of the $ i $-th particle, which follows the below stochastic differential equation:
    \begin{equation}\label{constraincon1}
\begin{gathered}
	dV^{i,N}_{t} = -\lambda \left(V^{i,N}_{t} - v_{\alpha}\left(\hat{\rho}_{t}^N\right)\right)\,dt - \dfrac{1}{\epsilon} \nabla G\left(V^{i,N}_{t}\right)\,dt + \sigma D^{i,N}_{t}\,dB^{i,N}_{t},
\end{gathered}
\end{equation}
with  $V^{i,N}_{0} \sim \rho_{0}$, where 
\begin{equation}\label{avg}
	v_{\alpha}(\hat{\rho}^{N}_{t})=\int v \cdot \dfrac{\wt(v)}{\|\wt\|_{L^{1}(\hat{\rho}_{t}^{N})}}\,d\hat{\rho}_{t}^{N}.
\end{equation}

The dynamics are driven by three distinct terms. The first and third terms are inherited from classical consensus-based optimization methods, while the second term is crafted as a forcing term to enforce the constraint. We will now explain  each of them in sequence.

The first drift term $ -\lambda \big(V^{i,N}_{t}-v_{\alpha}(\hat{\rho}^N_{t})\big)\,dt $ is formulated to guide all particles toward the consensus point $ v_{\alpha}(\hat{\rho}_{t}^{N}) $. This consensus point is strategically chosen as a location where the function is likely to achieve a small value. It is defined through a Gibbs-type distribution \eqref{avg}
where the weight $ \wt $ is defined as $\wt(v)=e^{-\alpha \mathcal{E}(v)}$. Here $ \lambda $ controls the force magnitude driving the particles towards the consensus point $v_\alpha(\hat{\rho}_t^N)$.

The choice of the consensus point is inspired by the well-known Laplace's
principle~\cite{BO,DZ,M}. In particular, if $\mathcal{E}$ is
continuous and $\rho\in\mathcal{P}(\mathbb{R}^{d})$ is compactly supported,
then
\[
    \lim_{\alpha\to\infty}
    \left(
    -\dfrac{1}{\alpha}
    \log\int e^{-\alpha\mathcal{E}(v)}\,d\rho(v)
    \right)
    =
    \inf_{v\in\operatorname{supp}(\rho)}
    \mathcal{E}(v).
\]

It is expected that the consensus point
$v_{\alpha}(\hat{\rho}_{t}^{N})$ serves as a reasonable smooth
approximation of
$\argmin_{i=1,...,N}\mathcal{E}(V^{i,N}_{t})$
when $\alpha$ is sufficiently large. Consequently, the particles are gathered to a location where $ \mathcal{E}(v) $ attains a small value. 

The diffusion term $ \sigma D^{i,N}_{t}\,dB^{i,N}_{t}$ encourages particles to explore the landscape of $ \mathcal{E}(v) $,
where $ D^{i,N}_{t} $ is a $ d\times d $ matrix function that determines the way in which particles explore the landscape and $ \{B_{t}^{i,N} \}_{i=1,...,N}$ are independent Wiener processes. There are different choices for the matrix function $ D^{i,N}_{t} $. One option is isotropic exploration \cite{PTTM}, where $ D^{i,N}_{t} $ is defined as, \begin{gather}\label{iso} D^{i,N}_{t}:=\|V^{i,N}_{t} - v_{\alpha}\big(\hat{\rho}_{t}^{N}\big)\|_{2} \mathbb{I}_{d}. \end{gather} 
Another option is anisotropic exploration \cite{CJLZ}, where $ D^{i,N}_{t} $ is defined as, 
\begin{gather}\label{aniso}
	D^{i,N}_{t}:=\text{diag}\left(V^{i,N}_{t} - v_{\alpha}\left(\hat{\rho}_{t}^{N}\right)\right).
\end{gather}


The third term $- \tfrac{1}{\epsilon} \nabla G\big(V^{i,N}_{ t}\big)\,d t, $ addresses the constraint $ \{G=0\} $. Since $ 0 $ is the minimum of the non-negative function $ G(v) $, locating the constraint $ \{G(v)=0\} $ is the same as minimizing $ G(v) $. Therefore, we propose the third term as a gradient descent of $ G(v) $, allowing $ G(v) $ to be minimized during the algorithm's progression. Here $ \epsilon>0 $ is a parameter that controls the magnitude of this term. When $ \epsilon $ is small, this term will encourage particles to concentrate around $ \{G=0\} $. These ideas were used in kinetic equations for swarming including alignment terms of Cucker-Smale type in order to derive kinetic models on the sphere such as the Viczek-Fokker--Planck model, see \cite{BC1,BC2,BC3,ABCD} for instance. \HZnew{Inevitably, taking $\epsilon$ small can introduce numerical stiffness and instability, so this term must be handled carefully in the discretization. In Section~\ref{sec:numerical}, we introduce a linearly implicit discretization to deal with this stiff term and analyze its well-posedness, stability properties, and consistency with the continuous particle dynamics.}

\HZnew{Throughout the paper, we assume that
$\alpha,\lambda,\sigma,\epsilon>0$.}

\HZnew{Before we proceed to the theoretical analysis of the model, we first present a preliminary numerical result in Figures \ref{fig:cbo_obj} and \ref{fig:converge} to illustrate the performance of the proposed interacting particle system~\eqref{constraincon1} compared to the projected CBO system \cite{FHPS2}, the penalized CBO system \cite{BHP} (with fixed penalty parameter), and the CB2O method \cite{CB2O} on a two-dimensional Ackley function (shown in Figure \ref{fig:cbo_obj} (a)) with different constraints.}

\begin{figure}
	\centering
	\subfloat[Ackley function.]{
		\includegraphics[width=0.3\textwidth]{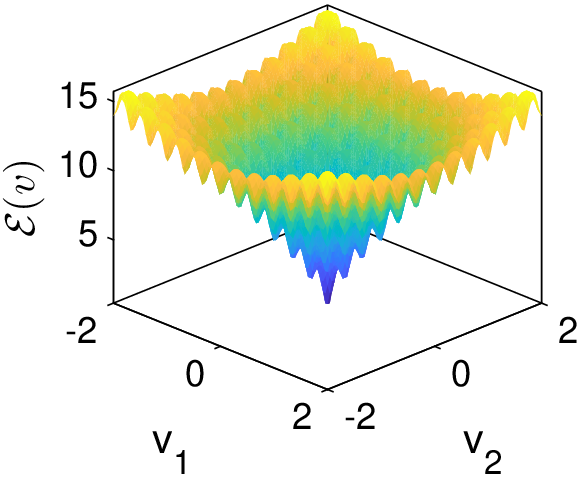}
	}

	\vspace{1ex}

	\subfloat[The success rate and averaged distance $D(v_\alpha,v^*)$ to the constrained minimizer.]{
		\centering
		{\color{black}
		\small
		\renewcommand{\arraystretch}{1.35}
		\setlength{\tabcolsep}{3pt}
\begin{tabular}{
    p{0.29\textwidth}
    p{0.13\textwidth}
    >{\centering\arraybackslash}p{0.12\textwidth}
    >{\centering\arraybackslash}p{0.12\textwidth}
    >{\centering\arraybackslash}p{0.12\textwidth}
    >{\centering\arraybackslash}p{0.12\textwidth}}
			\hline
			Case & Metric & Proposed & Projected CBO & Penalized CBO & CB2O \\
			\hline

			Sphere Constraint: constrained and unconstrained minimizers coincide
			& Success rate
			& \textbf{$100\%$}
			& \textbf{$100\%$}
			& $93\%$
			& $83\%$
			\\
			& $D(v_\alpha,v^*)$
			& $5.14\times10^{-4}$
			& \textbf{$4.30\times10^{-4}$}
			& $2.59\times10^{-3}$
			& $6.28\times10^{-3}$
			\\
			\hline

			Sphere Constraint: constrained and unconstrained minimizers differ
			& Success rate
			& \textbf{$100\%$}
			& $100\%$
			& $25\%$
			& $69\%$
			\\
			& $D(v_\alpha,v^*)$
			& $1.59\times10^{-3}$
			& \textbf{$1.18\times10^{-3}$}
			& $1.86\times10^{-2}$
			& $1.48\times10^{-2}$
			\\
			\hline

			Parabolic Curve Constraint: constrained and unconstrained minimizers differ
			& Success rate
			& \textbf{$100\%$}
			& N/A
			& $0\%$
			& $74\%$
			\\
			& $D(v_\alpha,v^*)$
			& \textbf{$1.44\times10^{-3}$}
			& N/A
			& $2.80\times10^{-2}$
			& $1.03\times10^{-2}$
			\\
			\hline
		\end{tabular}
		}
	}

	\caption{\HZnew{Ackley function landscape and preliminary numerical comparison of the constrained CBO methods.}}
	\label{fig:cbo_obj}
\end{figure}

\begin{figure}
	\centering
	{\fontsize{6}{7}\selectfont
	\textcolor[rgb]{0,0,1}{\rule{1.4em}{0.8pt}}\ Proposed\quad
	\textcolor[rgb]{1,0,0}{\rule{1.4em}{0.8pt}}\ Projected CBO\quad
	\textcolor[rgb]{1,0,1}{\rule{1.4em}{0.8pt}}\ Penalized CBO\quad
	\textcolor[rgb]{0,0.6,0}{\rule{1.4em}{0.8pt}}\ CB2O}\par\smallskip
	\subfloat[Sphere Constraint, where constrained and unconstrained minimizers coincide.]{\includegraphics[width=0.32\textwidth]{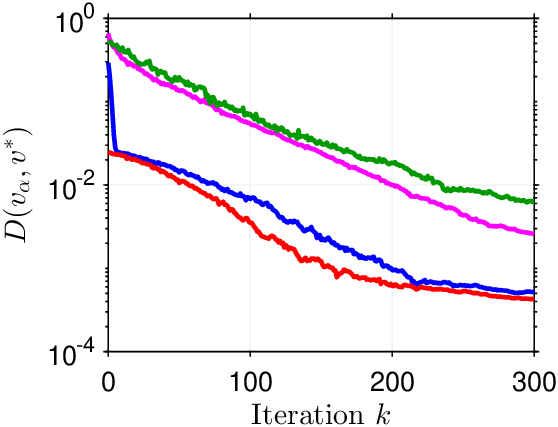}} \label{fig:cbo_const}\hfill
	\subfloat[Sphere Constraint, where constrained and unconstrained minimizers differ.]{\includegraphics[width=0.32\textwidth]{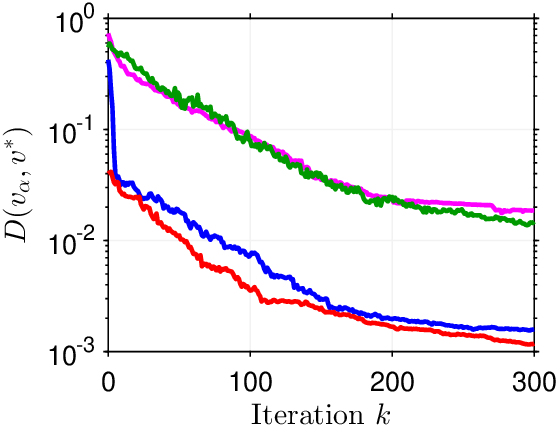}}\hfill
	\subfloat[Parabolic Curve Constraint, where constrained and unconstrained minimizers differ.]{\includegraphics[width=0.32\textwidth]{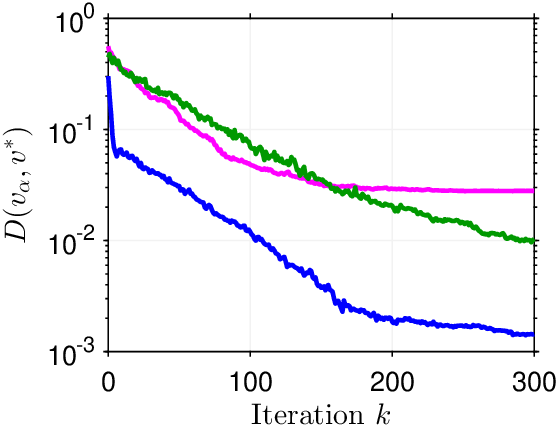}}
	\caption{\HZnew{The averaged distance $D(v_\alpha,v^*)$ to the constrained minimizer over $100$ simulations for the proposed method and the comparison methods.}}
	\label{fig:converge}
\end{figure}
\HZnew{We defer the algorithmic formulation to Section \ref{sec:numerical}, and detailed implementations of the preliminary experiments to Appendix \ref{appendix:fig0}. Our method achieves a $100\%$ success rate in finding the constrained minimizer in all three cases. The projected CBO performs similarly to our method when the constraint is a circle, but it is not directly applicable to parabolic curves}\footnote{{For parabolic curve $v_1^2=v_2$, each evaluation involves computing the closest-point projection by solving a cubic optimality condition and selecting the minimizing real root. This per-iteration overhead is considerably higher than for the other approaches, so we omit this experiment in the comparison.}}\HZnew{. Compared with the Penalized CBO and CB2O, the proposed method achieves higher success rates and smaller averaged distances to the constrained minimizer in all three cases.}

\HZnew{For the Penalized CBO, the global minimizer $v_p^*$ of the penalized objective function usually differs from the constrained minimizer $v^*$. Although it is possible to increase the penalty sufficiently to reduce the distance $\|v_p^* - v^*\|_2$, the landscape is then dominated by the penalization term, which makes it more difficult for CBO to locate the constrained minimizer. On the other hand, CB2O treats $G$ as a lower-level objective and enforces feasibility through derivative-free lower-level selection, while the proposed method directly drives the particles toward the constraint set through the gradient forcing term $-\nabla G/\epsilon$. The numerical results in Figures \ref{fig:cbo_obj} and \ref{fig:converge} illustrate the benefit of this additional constraint information in these examples.}

\section{Well-posedness and Mean-field limit}\label{sec:meanfield}

In this section we study some theoretical properties of   the particle system   described by Equation (\ref{constraincon1}). \HZnew{In this section, we consider the anisotropic diffusion \eqref{aniso} and study the corresponding particle system and its mean-field limit.\footnote{The results presented in this section can be extended to cover isotropic diffusion~\eqref{iso} following similar proof techniques.}}
Consequently, for $i=1,...,N$, the system defined by Equation
\eqref{constraincon1} transforms into the following form:
\begin{equation}\label{constraincon}
	\begin{aligned}
		dV^{i,N}_{t} &= -\lambda \left(V^{i,N}_{t} - v_{\alpha}\left(\hat{\rho}_{t}^N\right)\right)\,dt - \dfrac{1}{\epsilon} \nabla G\left(V^{i,N}_{t}\right)\,dt \\
		&\quad + \sigma \, \text{diag}\left(V^{i,N}_{t} - v_{\alpha}\left(\hat{\rho}_{t}^{N}\right)\right)\,dB^{i,N}_{t}, \quad V^{i,N}_{0} \sim \rho_{0},
	\end{aligned}
\end{equation}

When the number of particles $ N $ is large enough, one could study the mean-field limit as $ N \rightarrow \infty$. This limit yields an equation that characterizes the macroscopic behavior of the particles, specifically their density distribution. The investigation of the mean-field equation reveals the long-term dynamics of the particle system, which is related to the convergence of the particle system or the optimization method.  However, prior to this analysis in Section \ref{longtime}, it is necessary to establish the existence of the mean-field limit. In this section, we  establish the well-posedness of Equation (\ref{constraincon}), its mean-field limit, and the well-posedness of the resultant mean-field model. 

Throughout this section, we make the following assumptions.
\begin{assumption}\label{assump1}
    (1) The function $\CE$ is bounded with
    $\inf\CE=\underline{\CE}$ and $\sup\CE=\overline{\CE}$.

    (2) There exists $L>0$ such that $\forall u,v\in\mathbb{R}^d$, $
        |\CE(u)-\CE(v)|
        \leq L(1+\|u\|_2+\|v\|_2)\|u-v\|_2.$

    (3) $G\in\mathcal{C}^1(\mathbb{R}^d)$ and there exists $L>0$
    such that $\forall u,v\in\mathbb{R}^d$, $
        \|\nabla G(u)-\nabla G(v)\|_2
        \leq L\|u-v\|_2.$

\end{assumption}
Briefly speaking, in Assumption \ref{assump1} (1) and (2), we assume
the loss function $\mathcal{E}$ is bounded and locally Lipschitz.
In Assumption \ref{assump1} (3), we assume that $G$ is continuously
differentiable and its gradient is globally Lipschitz. \HZnew{Notice that this
also implies that $\nabla G$ has at most linear growth. Indeed, $
    \|\nabla G(u)\|_2
    \leq
    \|\nabla G(u)-\nabla G(0)\|_2+\|\nabla G(0)\|_2
    \leq
    C(1+\|u\|_2),
    $ for $ u\in\mathbb{R}^d.$}
Now we establish the well-posedness of the interacting particle system (\ref{constraincon}) in the following theorem.

\begin{theorem}\label{wellposedmicro}
	(Proof in Appendix \ref{appenmicro}) For any $N\in\mathbb{N}$,
	the stochastic differential equation (\ref{constraincon}) has a unique
	strong solution $\{V_{t}^{i,N}|t\geq0\}_{i=1}^{N}$ for any
	$\mathcal{F}_0$-measurable initial condition
	$\{V_{0}^{i,N}\}_{i=1}^{N}$ that is independent of the driving Wiener
	processes and satisfies $
		\mathbb{E}\big[\|V_{0}^{i,N}\|_{2}^{2}\big]<\infty,$ for $ i=1,\ldots,N.$
\end{theorem}

By letting the number of agents $ N\rightarrow\infty $ in the model (\ref{constraincon}), the mean-field limit of the model is formally given by the following SDE,
\begin{gather}\label{SDE}
	d\bar{V}_{t}=-\lambda\Big(\bar{V}_{t}-v_{\alpha}\big(\rho_{t}\big)\Big)\,dt-\dfrac{1}{\epsilon}\nabla G(\bar{V}_t)\,dt+\sigma\text{diag}\bigg(\bar{V}_{t} - v_{\alpha}\big(\rho_{t}\big)\bigg)\,dB_{t}.
\end{gather}
Then the corresponding Fokker--Planck equation is
\begin{gather}\label{FPK}
\partial_t\rho_t
=
\lambda\,\mathrm{div}\left(
\left(v-v_\alpha(\rho_t)\right)\rho_t
\right)
+
\frac{1}{\epsilon}\,
\mathrm{div}\left(
\nabla G(v)\rho_t
\right)
+
\frac{\sigma^2}{2}
\sum_{k=1}^d
\partial_{x_kx_k}
\left(
\left(v-v_\alpha(\rho_t)\right)_k^2\rho_t
\right).
\end{gather}
Next, we will prove the above equations (\ref{SDE}), (\ref{FPK}) are well-posed, and they model the mean-field limit. For the corresponding Fokker--Planck equation, we in particular study its weak solution, which is defined as follows.

\begin{definition}\label{weakdef}
	We say $ \rho_{t}\in\mathcal{C} \left([0,T],\mathcal{P}_{4}\left(\mathbb{R}^{d}\right)\right)$ is a weak solution to (\ref{FPK}) if\\
	(i) It admits continuity in time in $ \mathcal{C}^{'}_{b} $ topology, i.e., $\left\langle \phi, \rho_{t_{n}}\right\rangle \rightarrow \left\langle \phi,\rho_{t}\right\rangle$, $ \forall \phi\in\mathcal{C}_{b}\big(\mathbb{R}^{d}\big) $ and $ t_{n}\rightarrow t $.\\
(ii) For every $\phi\in\mathcal{C}_c^2(\mathbb{R}^d)$ and
$t\in[0,T]$, it holds that
\begin{align*}
\langle\phi,\rho_t\rangle-\langle\phi,\rho_0\rangle
={}&
-\lambda\int_0^t
\left\langle
\left(v-v_\alpha(\rho_s)\right)\cdot\nabla\phi,
\rho_s
\right\rangle\,ds-\frac{1}{\epsilon}\int_0^t
\left\langle
\nabla G(v)\cdot\nabla\phi,
\rho_s
\right\rangle\,ds\\
&+\frac{\sigma^2}{2}
\sum_{k=1}^d
\int_0^t
\left\langle
\left(v-v_\alpha(\rho_s)\right)_k^2
\partial_{kk}\phi,
\rho_s
\right\rangle\,ds.
\end{align*}
\end{definition}
Now we present the well-posedness result of (\ref{SDE}) and (\ref{FPK}).
\begin{theorem}\label{meanfieldwellposed}
(Proof in Appendix \ref{appenmean})
Suppose Assumption~\ref{assump1} holds and
$\rho_0\in\mathcal{P}_4(\mathbb{R}^d)$.
Then given $T>0$, there exists a unique nonlinear process
$\bar V$ with continuous paths satisfying \eqref{SDE} with initial
distribution $\bar V_0\sim\rho_0$ in the strong sense. Moreover, $
\rho_t=\operatorname{Law}(\bar V_t)
\in
\mathcal{C}\left([0,T],\mathcal{P}_4(\mathbb{R}^d)\right)$ satisfies the corresponding Fokker--Planck equation \eqref{FPK}
in the weak sense with $\lim_{t\rightarrow0}\rho_t=\rho_0$.
\end{theorem}
Then we are ready to present the result showing that (\ref{SDE}), (\ref{FPK}) indeed characterize the mean-field limit of the particle system.


\begin{theorem}\label{meanfieldthm}
(Proof in Appendix \ref{appenmeanfield})
Fix $T>0$. Suppose Assumption~\ref{assump1} holds and
$\rho_0\in\mathcal{P}_4(\mathbb{R}^d)$. For any $N\geq2$, let
$\{(V_t^{i,N})_{t\in[0,T]}\}_{i=1}^{N}$ be the unique solution to
\eqref{constraincon} with $\rho_0^{\otimes N}$-distributed initial
data $\{V_0^{i,N}\}_{i=1}^{N}$. Then, for every fixed
$t\in[0,T]$, $
\mathcal{L}(\hat{\rho}_t^N)
\Rightarrow
\delta_{\rho_t}$ in $
\mathcal{P}\big(\mathcal{P}(\mathbb{R}^d)\big)$ as $N\rightarrow\infty,$ where
$\rho_t=\operatorname{Law}(\bar V_t)$ and $\bar V$ is the unique
nonlinear process given by Theorem~\ref{meanfieldwellposed}.
Equivalently, $\hat{\rho}_t^N$ converges in probability to
$\rho_t$ in $\mathcal{P}(\mathbb{R}^d)$ equipped with the weak
topology. Moreover, $(\rho_t)_{t\in[0,T]}$ is the unique weak
solution to \eqref{FPK} with initial data $\rho_0$.
\end{theorem}
\section{Convergence to the constrained minimizer in the mean-field limit}\label{longtime}

\HZnew{
In this section, we establish the long-time convergence result for the mean-field dynamics. Throughout this section, we consider the isotropic version of the dynamics, where the diffusion matrix is chosen as in~\eqref{iso}. The corresponding microscopic particle system is given by
\begin{equation}\label{constraincon_iso}
	dV^{i,N}_{t}
	=-\lambda
	\left(
	V^{i,N}_{t}
	-v_{\alpha}\left(\hat{\rho}_{t}^{N}\right)
	\right)\,dt
	-\dfrac{1}{\epsilon}
	\nabla G\left(V^{i,N}_{t}\right)\,dt+\sigma
	\left\|
	V^{i,N}_{t}
	-v_{\alpha}\left(\hat{\rho}_{t}^{N}\right)
	\right\|_{2}\,dB^{i,N}_{t},
\end{equation}
with $V^{i,N}_{0}\sim\rho_{0}$ for $i=1,\ldots,N.$ The corresponding mean-field dynamics is
\begin{gather}\label{SDE_iso}
d\bar{V}_{t}
=
-\lambda\Big(\bar{V}_{t}-v_{\alpha}(\rho_{t})\Big)\,dt
-\dfrac{1}{\epsilon}\nabla G(\bar V_t)\,dt
+\sigma
\Big\|\bar{V}_{t}-v_{\alpha}(\rho_{t})\Big\|_{2}\,dB_{t}.
\end{gather}
And the corresponding Fokker--Planck equation is
\begin{gather}\label{FPK_iso}
\partial_t\rho_t
=
\lambda\,\operatorname{div}
\left(
\left(v-v_{\alpha}(\rho_t)\right)\rho_t
\right)
+
\dfrac{1}{\epsilon}
\operatorname{div}
\left(
\nabla G(v)\rho_t
\right)
+
\dfrac{\sigma^2}{2}
\Delta
\left(
\left\|v-v_{\alpha}(\rho_t)\right\|_2^2\rho_t
\right).
\end{gather}
}
\HZnew{The isotropic diffusion is adopted in this section since its radial structure allows us
to establish the lower-mass estimate, i.e., Lemma~\ref{ball}, needed in the
long-time convergence analysis. Under Assumption~\ref{assump1}, the
well-posedness and mean-field limit arguments established in
Section~\ref{sec:meanfield} for anisotropic dynamics extend directly to the isotropic dynamics
\eqref{constraincon_iso}, \eqref{SDE_iso} and~\eqref{FPK_iso}, as the isotropic diffusion coefficient satisfies the same
Lipschitz and linear growth properties required in the arguments of
Section~\ref{sec:meanfield}. The corresponding estimates follow in the same way, up to changes in the constants, and we omit the details. The corresponding weak solution
in the isotropic case is defined as follows.
}
\begin{definition}\label{weakdef_iso}
\HZnew{We say
$\rho_t\in\mathcal{C}\left([0,T],\mathcal{P}_4(\mathbb{R}^d)\right)$
is a weak solution to \eqref{FPK_iso} if}
\HZnew{(i) It admits continuity in time in the $\mathcal{C}_b'$ topology, i.e., $
\langle\phi,\rho_{t_n}\rangle
\rightarrow
\langle\phi,\rho_t\rangle,$ $
\forall\phi\in\mathcal{C}_b(\mathbb{R}^d)$ and $ t_n\rightarrow t.$}\\
\HZnew{(ii) For every $\phi\in\mathcal{C}_c^2(\mathbb{R}^d)$ and
$t\in[0,T]$, it holds that
\begin{align*}
\langle\phi,\rho_t\rangle-\langle\phi,\rho_0\rangle
=&
-\lambda\int_0^t
\left\langle
\left(v-v_\alpha(\rho_s)\right)\cdot\nabla\phi,
\rho_s
\right\rangle\,ds-\dfrac{1}{\epsilon}\int_0^t
\left\langle
\nabla G(v)\cdot\nabla\phi,
\rho_s
\right\rangle\,ds\\
&+\dfrac{\sigma^2}{2}\int_0^t
\left\langle
\left\|v-v_\alpha(\rho_s)\right\|_2^2
\Delta\phi,
\rho_s
\right\rangle\,ds.
\end{align*}}
\end{definition}
\HZnew{\begin{remark}\label{remark:extension} In Definition \ref{weakdef_iso} (ii), the test function space is $ \mathcal{C}^{2}_{c}(\mathbb{R}^{d}) $. We could extend $ \mathcal{C}^{2}_{c}(\mathbb{R}^{d}) $ to a larger space $ \mathcal{C}^{2}_{*}(\mathbb{R}^{d}) $ as explained in Appendix \ref{testextension}, which will be used in the proof later. $ \mathcal{C}^{2}_{*}(\mathbb{R}^{d}) $ is defined below. \begin{align*} \mathcal{C}^{2}_{*}(\mathbb{R}^{d}):=\big\{&\phi\in\mathcal{C}^{2}(\mathbb{R}^{d})\mid|\partial_{k}\phi(x)|\leq C(1+|x_{k}|)\\ &\text{ and }\sup_{x\in \mathbb{R}^{d}}|\partial_{kk}\phi(x)|<\infty \text{ for all $k=1,2,...,d$}\big\}. \end{align*} In other words, under our assumptions that $\mathcal{E}$ is bounded and $\nabla G$ has at most linear growth, if $ \rho_{t} \in\mathcal{C} \left([0,T],\mathcal{P}_{4}\left(\mathbb{R}^{d}\right)\right) $ solves equation (\ref{FPK_iso}) in the weak sense as in Definition \ref{weakdef_iso}, the equality in Definition \ref{weakdef_iso} will hold for any test function $ \phi\in \mathcal{C}_{*}^{2}\big(\mathbb{R}^{d}\big)$. 
\end{remark}}

\textcolor{black}{Throughout this section, we assume Problem (\ref{opt1}) admits a unique solution $v^*$, which is a common assumption in both unconstrained and constrained CBO-related literature, e.g., see \cite{BHP,CJLZ,FHPS1,FKR}{}\footnote{In \cite{BHP}{}, the uniqueness assumption is slightly weaker: it requires the penalized loss function to admit a unique global minimizer, which coincides with the constrained minimizer.}. Our primary goal is to establish a key result: under suitable assumptions and the selection of appropriate parameters, the particles will concentrate around $v^*$ with arbitrary closeness, confirming the effectiveness of the method in the mean-field limit.} 

For simplicity and without loss of generality, we assume
$ \mathcal{E}(v^{*})=0 $. {Throughout this section, we use $ \rho_{t} $ to represent a weak solution of Equation~(\ref{FPK_iso}) in the sense of
Definition~\ref{weakdef_iso}, and assume that $\rho_t$ admits a
density with respect to the Lebesgue measure for every $t$ in its
interval of definition. With a slight abuse of notation, we also use
$\rho_t$ to refer to its density function.}

\subsection{Main Results}
To study the convergence of $ \rho_{t} $ to $ v^{*} $, we define the following energy functional
\begin{equation}\label{energyfunc}
	\mathcal{V}\big(\rho_{t}\big):=\dfrac{1}{2}\int\|v-v^{*}\|_{2}^{2}\,d\rho_{t}(v) .
\end{equation}
The above defined quantity $ \mathcal{V}\big(\rho_{t}\big) $ provides a measure of the distance between the distribution of the particles $ \rho_{t} $ and the Dirac measure at $ v^{*} $, denoted as $ \delta_{v^{*}} $. Specifically, we have the relationship $2\mathcal{V}\big(\rho_{t}\big)=W^{2}_{2}\big(\rho_{t},\delta_{v^{*}}\big)$, where $ W_{2}\big(\rho_{t},\delta_{v^{*}}\big) $ denotes the Wasserstein-2 distance between $ \rho_{t} $ and $ \delta_{v^{*}} $. The diminishing behavior of $ \mathcal{V}\big(\rho_{t}\big) $ indicates that $ \rho_{t} $ is approaching $ \delta_{v^{*}} $, implying that particles are concentrating around $v^*$. In this paper, we establish the following main theorem concerning the decay of $ \mathcal{V}\big(\rho_{t}\big) $.

\begin{theorem}\label{main}
Suppose $G$ and $\mathcal{E}$ satisfy Assumption~\ref{wellbehave}
(well-behaved). Fix any $\tau\in(0,1)$ and parameters
$\lambda,\sigma>0$ with
\HZnew{$2\lambda>d\sigma^{2}$.}
Let $\rho_{0}\in\mathcal{P}_{4}(\mathbb{R}^{d})$ satisfy $
\rho_{0}\big(B(v^{*},r)\big)>0,\ \forall r>0.$ \HZnew{For any error tolerance
$\delta\in\big(0,\mathcal{V}(\rho_{0})\big)$, there exist
$\alpha_{\delta}>0$ and $I_{\delta}>0$ such that, for $\alpha=\alpha_{\delta}$, if $
\int G(v)\,d\rho_{0}(v)\leq I_{\delta},$ then one can choose $\epsilon>0$ sufficiently small so that any weak
solution $(\rho_t)_{t\in[0,T^*]}$ to \eqref{FPK_iso} in the sense of
Definition~\ref{weakdef_iso}, with initial data $\rho_0$, and
admitting a density with respect to the Lebesgue measure for every
$t\in[0,T^*]$, satisfies}
\begin{gather}\label{result}
\min_{t\in[0,T^{*}]}\mathcal{V}\left(\rho_{t}\right)
\leq \delta,
\qquad
\HZnew{
T^{*}
=
\dfrac{1}
{(1-\tau)(2\lambda-d\sigma^{2})}
\log\left(
\dfrac{\mathcal{V}(\rho_{0})}{\delta}
\right).
}
\end{gather}
Furthermore, until $\ener$ reaches the prescribed accuracy $\delta$,
the following exponential decay holds:
\begin{gather}\label{decayrate}
\HZnew{
\ener
\leq
\mathcal{V}(\rho_{0})
\exp\left(
-(1-\tau)(2\lambda-d\sigma^{2})t
\right).
}
\end{gather}
\HZnew{Here, $\epsilon$ is the constraint-forcing parameter in
\eqref{FPK_iso}.}
\end{theorem}
\begin{remark}
	\HZnew{In the above theorem, when $\delta$ is fixed, we first choose
$\alpha=\alpha(\delta)$ according to \eqref{alpha}. Based on this
choice of $\alpha$, we then choose the admissible initial constraint
energy $I(\delta)$ according to \eqref{initial}. Finally, provided
$\int G\,d\rho_0\leq I(\delta)$, we choose $\epsilon$ sufficiently
small according to \eqref{epsilon}.
Notice also that, due to the use of isotropic diffusion in this section,
the condition $2\lambda>d\sigma^2$ depends explicitly on the dimension
$d$.} \textcolor{black}{Intuitively, the initial condition $ \int G\,d\rho_{0}(v) \leq I(\delta) $ requires the initial particles to be near the constraint set $\{G = 0\}$. This is not an unusual assumption in constrained CBO literature; e.g., in \cite{FHPS1}{}, the convergence result assumes that the initial distribution $\rho_0$ is fully supported on the constraint set.  While this condition is primarily technical, arising from the proof techniques, it is not essential in practice for the algorithm to converge, as demonstrated by the numerical examples in Section \ref{sec:numerical}. Notably, our assumption does not require the particles to be close to the constrained minimizer $v^*$.}

\end{remark}

\subsection{Assumption}\label{subsec:wellbehave}
In this subsection, we define clearly what it means by being \lq\textbf{well-behaved}\rq. $ G $ and $ \mathcal{E} $ are \textbf{well-behaved} if the following Assumption \ref{wellbehave} is satisfied. It is worth noting that Assumption \ref{wellbehave} in this section is independent of Assumption \ref{assump1}. In other words, for the proofs in this section, Assumption \ref{assump1} is not required.

\begin{assumption}\label{wellbehave}
	\textbf{A. Assumptions on $ \mathcal{E} $:}
	\begin{enumerate}
		\item[(A1)] \HZnew{$ \mathcal{E}\in\mathcal{C}(\mathbb{R}^{d}) $} is bounded: $\underline{\mathcal{E}}\leq\mathcal{E}\leq \overline{\mathcal{E}}. $
		\item[(A2)] $ \mathcal{E} $ is locally H\"{o}lder continuous around $ v^{*} $, i.e. there exists $ r_{0}>0 $ such that $ \forall v_{1}, v_{2}\in B(v^{*},r_{0}) $, $
			|\mathcal{E}(v_{1})-\mathcal{E}(v_{2})|\leq C\|v_{1}-v_{2}\|_{\HZnew{2}}^{\beta}
		$ for some $ C\geq 0 $ and $ \beta >0 $.
	\end{enumerate}
	
	\textbf{B. Assumptions on $ G $:}
	\begin{enumerate}
		\item[(B1)] $ \left\langle \nabla G(v),v-v^{*}\right\rangle\geq0 $ holds for any $ v\in\mathbb{R}^{d} $.
		\item[(B2)]  \HZnew{$ G(v)\in \mathcal{C}^{1,1}_{*}\big(\mathbb{R}^{d}\big) $} and there exists $ C>0 $ such that $G(v)\leq C  \| \nabla G(v) \|_{2}^{2} $, $ \forall v\in \mathbb{R}^{d} $\HZnew{, where $\mathcal{C}^{1,1}_*(\mathbb{R}^d)$ is defined in \eqref{def:C11}.}
		\item[(B3)]  $ \nabla G(v)\neq 0 $, $ \forall v\in \big\{G(v)\in(0,u_{0})\big\} $ and $ \int_{G(v)\in(0,u_{0})}\frac{1}{\|\nabla G(v)\|_{2}}\,dv<\infty $ for some $ u_{0}>0 $ small enough.
		
	\end{enumerate}
	
	\begin{remark}\label{rmk:ass G}
\HZnew{Assumption~\ref{wellbehave} (B1) can be viewed as a one-point convexity condition on $G$ with respect to the constrained minimizer $v^*$. To see its role, consider the gradient flow generated only by
the constraint force, $
    \dot v=-\nabla G(v).$
Then $
    \frac{1}{2}\frac{d}{dt}\|v-v^*\|_2^2
    =
    -\left\langle\nabla G(v),v-v^*\right\rangle.$ Therefore, Assumption~\ref{wellbehave} (B1) guarantees that the constraint force does not increase the distance to the particular constrained minimizer $v^*$. This sign property is used in the
convergence analysis when estimating the decay of $\mathcal{V}(\rho_t)$, the mass around $v^*$, and the evolution of the constraint energy $\int G(v)\,d\rho_t(v)$.}

\HZnew{Assumption~\ref{wellbehave} (B1) is weaker than the usual convexity condition. Thus convexity is sufficient for (B1), while (B1) only requires a favorable geometry relative to the single point $v^*$. On the other
hand, for a general nonconvex $G$, (B1) does not necessarily hold. Notice that the gradient flow still satisfies $
    \frac{d}{dt}G(v)
    =
    -\|\nabla G(v)\|_2^2\leq0,$ so the constraint force always decreases the constraint energy itself, but it does not necessarily decrease the distance to $v^*$. Therefore, (B1) is a sufficient condition for the present convergence
argument rather than a requirement for defining the dynamics or the numerical algorithm.}

\HZnew{Assumption~\ref{wellbehave} (B2) contains both a regularity
condition and the gradient-dominance condition $
    G(v)\leq C\|\nabla G(v)\|_2^2.$ The latter implies that $\nabla G(v)=0$ only if $G(v)=0$, and hence rules out stationary points of $G$ away from the feasible set. Without such a condition, similar to other gradient descent algorithms, some particles may get trapped around local minimizers $\hat v$ of $G$ satisfying $\nabla G(\hat v)=0$ but $G(\hat v)>0$. Nevertheless, if the objective values at such points are not significantly smaller than $\mathcal{E}(v^*)$, this behavior need not prevent the consensus point from approaching the constrained minimizer $v^*$, as illustrated by the experiments in Section~\ref{sec:eq1}, Figure~\ref{fig: evolution}. One possible way to avoid the extreme case where such an infeasible stationary point has a very small objective value is to add a positive scalar multiple of $G$ to $\mathcal{E}$ without changing the constrained minimizer $v^*$. This differs from the penalization method in \cite{BHP}, since the convergence toward the constraint is still enforced through the dynamics and the added penalty parameter does not need to approach infinity.}

\HZnew{Finally, Assumption~\ref{wellbehave} (B3) is a nondegeneracy
condition near the feasible set. It guarantees that $\nabla G$ does
not vanish too rapidly around $\{G=0\}$ and provides the integrability
needed in the coarea estimates used in the quantitative Laplace principle (Lemma~\ref{lp}).}
        
        {Although Assumption \ref{wellbehave} (B) is formulated for an abstract function $G$, it is in fact satisfied by a rather large class of functions
arising from geometric constraints. Lemma~\ref{lem:assumpB-distance} introduced below shows that
Assumption \ref{wellbehave} (B) holds for the case where $G$ is the squared
distance to a compact convex constraint set $K$ of small intrinsic
dimension. \HZnew{This squared-distance representation also allows the framework
to accommodate inequality constraints, possibly combined with equality
constraints, when their feasible set $K$ satisfies the conditions of
Lemma~\ref{lem:assumpB-distance} and projections onto $K$ are computable.}}
{\begin{lemma}
\label{lem:assumpB-distance}(Proof in Appendix \ref{appen:proofassumpB})
Let $K\subset\mathbb{R}^{d}$ be a nonempty compact convex set whose affine
hull has dimension at most $d-2$. Define $
g(v) := \dist(v,K),
$ and $
G(v) := g(v)^{2},
$ for $ v\in\mathbb{R}^{d}.
$
Then $G$ satisfies Assumption \ref{wellbehave} (B).
\end{lemma}}
{A practical example is portfolio optimization in finance, where $v\in\mathbb{R}^d$ denotes portfolio weights (with $v_j$ the position in asset $j$), see \cite{portfolio_book}. Typical constraints include a budget constraint $\sum_{i=1}^d v_i + c = 0$, market neutrality $\beta\cdot v = 0$ (with $\beta$ the vector of asset betas), diversification $\|v\|_2^2\leq D$, and leverage $\|v\|_1\leq 2$. Intersecting such constraints yields a compact convex feasible set whose intrinsic dimension can be at most $d-2$. Common optimization objectives in this application include, for example, variance \cite{Markowitz}, $\beta$-Value-at-Risk \cite{Kast,mausser1999beyond}, and Conditional Value-at-Risk \cite{Rockafellar}.}
	\end{remark}

	\textbf{C. Assumptions on the coupling of $\mathcal{E}$ and $G$:}
There exist a non-negative increasing function $\tau_1(x)$ from
$\mathbb{R}$ to $\mathbb{R}$ with
$\lim_{x\rightarrow0}\tau_1(x)=0$, $\eta>0$, $\mu>0$,
$R_0>0$ and $\mathcal{E}_\infty>0$, such that the following holds
$\forall u\in[0,u_0]$, where $u_0>0$ is a small constant.

\begin{enumerate}
\item[(C1)]
There exists \HZnew{a Borel measurable map}
$[0,u_0]\ni u\mapsto v_u\in\mathbb{R}^d$ and
$\underline{\mathcal{E}_u}\in\mathbb{R}$ such that $
v_u=\mathop{\arg\min}_{v\in\{G=u\}}\mathcal{E}(v),$ $
\underline{\mathcal{E}_u}=\mathcal{E}(v_u).$
Moreover,
\HZnew{$
\|v_u-v^*\|_2\leq\tau_1(u),$ and $
\partial B(v_u,r)\cap\{v\mid G(v)=u\}\neq\emptyset,$ $\forall\text{ $r$ small}$.
}

\item[(C2)]\label{subsec:wellbehave:C2}
\HZnew{
It holds that $
\|v-v_u\|_2
\leq
\frac1\eta
\left(
\mathcal{E}(v)-\underline{\mathcal{E}_u}
\right)^\mu$
for all $
v\in B(v_u,R_0)\cap\{G(v)=u\},$
and $
\mathcal{E}_\infty
<
\mathcal{E}(v)-\underline{\mathcal{E}_u}$
for all $
v\in B(v_u,R_0)^c\cap\{G(v)=u\}.$
}
\end{enumerate}
\end{assumption}
\begin{remark}
Assumption~\ref{wellbehave} (C1) controls the behavior of the minimizers of $\mathcal{E}$ on nearby level sets $\{G=u\}$ for sufficiently small $u$. Together with the regularity of $\mathcal{E}$, it ensures that the minimizers $v_u$ and the corresponding minimum values $\mathcal{E}(v_u)$ are close to $v^*$ and $\mathcal{E}(v^*)$, respectively. Without this condition, a minimizer $v_u$ on a nearby level set may lie far from $v^*$, as illustrated in Figure~\ref{counter}(a), where $v^*$ and $v_u$ are represented by a solid green pentagon and a solid orange circle, respectively. In an extreme case, $\mathcal{E}(v_u)$ may also be substantially smaller than $\mathcal{E}(v^*)$ even when $u>0$ is small. Since the gradient forcing does not enforce feasibility exactly, particles may occupy nearby positive level sets. If many particles lie near $v_u$, their low objective values can cause the weighted consensus point to favor this region. The consensus attraction may then reinforce concentration near $v_u$, potentially preventing convergence to $v^*$, as illustrated in Figure~\ref{counter}(b). This possible failure mechanism motivates Assumption~\ref{wellbehave} (C1). Together with the other assumptions, it yields the estimates in Lemma~\ref{C1}, which quantify the variation of local objective values across nearby level sets.  

Assumption~\ref{subsec:wellbehave:C2} (C2) ensures that the minimizer on each nearby level set is distinguishable from other points. In particular, on each such level set $\{G=u\}$, the minimizer $v_u$ is unique, and the objective gap $\mathcal{E}(v)-\mathcal{E}(v_u)$ is bounded below by a positive constant multiple of \HZnew{$\|v-v_u\|_2^{1/\mu}$} near $v_u$. \textcolor{black}{This type of local coercivity condition is common in the CBO literature; see, e.g., \cite{BHP,FKR}.}

\HZnew{Further, Assumption~\ref{wellbehave} (A1), (B2), and
(C1) imply that
$u\mapsto\underline{\mathcal{E}_u}$ is Borel measurable on
$[0,u_0]$, and that $
v\mapsto v_{G(v)}$
is Borel measurable on
$\{v\in\mathbb{R}^d\mid G(v)\in[0,u_0]\}$. As a result, all the
integrals involving $v_{G(v)}$ are well-defined.}

\begin{figure}
\centering
\subfloat[]{\includegraphics[width=0.4\textwidth]{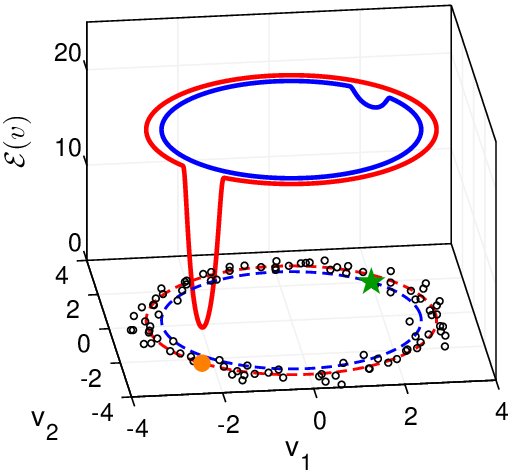}}
\subfloat[]{\includegraphics[width=0.4\textwidth]{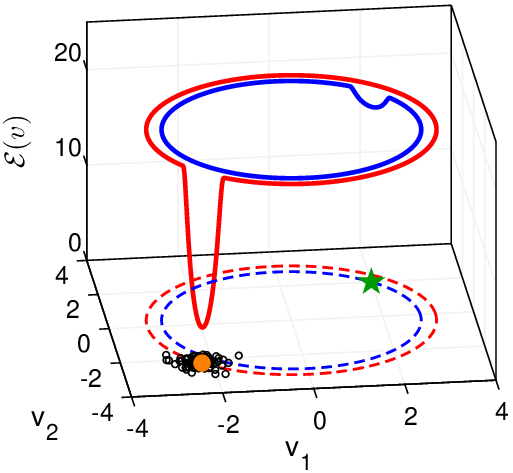}}	
\caption{The blue curves represent function values on the constraint set $\{G=0\}$ and the red curves on the level set $\{G=u\}$. Dashed lines represent corresponding constraint sets. The green pentagon denotes the constrained minimizer $v^*$, and the orange circle represents the minimizer $v_u$ on the nearby level set ${G=u}$. Empty circles represent particles.}
\label{counter}
\end{figure}
\end{remark}

\subsection{Sketch of the Proof}\label{subsec:proofsketch}
In this subsection, we lay out the strategy of the proof and main Lemmas, 
and the complete proof of Theorem  \ref{main}.

We first plug $  \mathcal{V}\big(\rho_{t}\big)=\tfrac{1}{2}\|v-v^{*}\|_{2}^{2}$ into Definition \ref{weakdef_iso}, which yields the following differential inequality that describes the dynamics of the energy functional $ \mathcal{V}\big(\rho_{t}\big) $.
\begin{lemma}\label{energy}
    (Proof in Appendix \ref{appen:energy}) Let
    $\mathcal{V}\big(\rho_{t}\big)$ be the energy functional defined in
    (\ref{energyfunc}). Under Assumption \ref{wellbehave},
    \HZnew{for a.e. $t\in(0,T)$,}
    \vspace{-0.2ex}
    \begin{equation}\label{eq:dyn}
    \begin{aligned}
    \dfrac{d}{dt}\mathcal{V}\big(\rho_{t}\big)
    \leq&
    -\HZnew{(2\lambda-d\sigma^{2})}
    \mathcal{V}\big(\rho_{t}\big)
    +
    \HZnew{\sqrt{2}(\lambda+d\sigma^{2})}
    \sqrt{\mathcal{V}\big(\rho_{t}\big)}
    \|v_{\alpha}\big(\rho_{t}\big)-v^{*}\|_2\\
    &+
    \HZnew{\dfrac{d\sigma^{2}}{2}}
    \|v_{\alpha}\big(\rho_{t}\big)-v^{*}\|_2^{2}
    -
    \dfrac{1}{\epsilon}
    \int
    \Bigl\langle \nabla G,v-v^{*}\Bigr\rangle
    \,d\rho_{t}(v).
    \end{aligned}
    \end{equation}
\end{lemma}
It is noteworthy that if
$\|v_{\alpha}\big(\rho_{t}\big)-v^{*}\|_{2}$
can be bounded by a suitable scalar multiple of
$\sqrt{\mathcal{V}\big(\rho_{t}\big)}$,
and the last term is negative, we would then obtain the inequality:
\vspace{-2ex}
\begin{gather}\label{gronwall}
    \dfrac{d}{dt}\mathcal{V}\big(\rho_{t}\big)
    \leq
    \HZnew{
    -(1-\tau)(2\lambda-d\sigma^{2})
    \mathcal{V}\big(\rho_{t}\big),
    }
    \vspace{-1ex}
\end{gather}
to which Gronwall's inequality can be applied, ensuring exponential decay. To this end, we first present the following auxiliary lemma, which draws out additional consequences from Assumption \ref{wellbehave} (C).
\begin{lemma}\label{C1}
    (Proof in Appendix \ref{C1details})
    There exist non-negative increasing functions
    $\tau_2(x),\tau_3(x)$ and $\tau_4(x)$ mapping from
    $\mathbb{R}$ to $\mathbb{R}$ with
    $\lim_{x\rightarrow0}\tau_i(x)=0$ for $i=2,3,4$
    so that the following hold for all $u,r\geq0$ small enough:
    \begin{gather*}
        |\underline{\mathcal{E}_u}|
        =|\mathcal{E}(v_u)|
        \leq\tau_2(u),\\
        |\mathcal{E}_r^u-\mathcal{E}_r|
        \leq\tau_3(\max\{u,r\}),\\
        |\mathcal{E}_r^u-\mathcal{E}_r^0|
        \leq\tau_4(\max\{u,r\}),
    \end{gather*}
    where
    \HZnew{$
    \mathcal{E}_r^u
    =
    \max_{v\in B(v_u,r)\cap\{G=u\}}
    \mathcal{E}(v),
    $ and $
    \mathcal{E}_r
    =
    \max_{v\in B(v^*,r)}
    \mathcal{E}(v).
    $
    }
\end{lemma}
Using the functions given in Lemma \ref{C1}, we establish the following inequality to control $\|v_{\alpha}\big(\rho_{t}\big)-v^{*}\|_{2}$.
\begin{lemma}[A Quantitative Laplace Principle]\label{lp}
	(Proof in Appendix \ref{appendix:laplace_proof}) Fix $r>0$ and $u>0$
	small enough. Let $q>0$ be such that $
		q+\mathcal{E}_{r}^{\tilde{u}}
		-\underline{\mathcal{E}_{\tilde{u}}}
		<\mathcal{E}_{\infty},
		\ \forall \tilde{u}\in[0,u).$ \HZnew{Assume in addition that $\rho_t\big(B(v^*,r)\big)>0$.} Then
    \begin{equation}
        \begin{aligned}\label{eq_11}
		\|\con-v^{*}\|_2\leq&
		\HZnew{
		2\dfrac{\big(q+\ec+\tau_{2}(u)
		+\tau_{4}(\max\{u,r\})\big)^{\mu}}{\eta}
		}\\
		&+\HZnew{
		\dfrac{e^{-\alpha\big(q-\tau_{3}(\max\{u,r\})\big)}}
		{\rt\big(B(v^{*},r)\big)}
		\int_{\{G\in(0,u)\}}\|v-v_{G(v)}\|_2\dvt
		}\\
		&+\HZnew{
		\dfrac{e^{-\alpha\big(q-\tau_{3}(r)\big)}}
		{\rt\big(B(v^{*},r)\big)}
		\int_{\{G=0\}}\|v-v^{*}\|_2\dvt
		}\\
		&+\HZnew{\tau_{1}(u)}
		+\int_{\{G(v)\geq u\}}
		\dfrac{\|v-v^{*}\|_2}{\|\wt\|_{L^1(\rt)}}
		e^{-\alpha\mathcal{E}(v)}\dvt.
	\end{aligned}
    \end{equation}
\end{lemma}


It is observed that with an appropriate choice of $ q,u,r $, provided $ \rt(B(v^{*},r)) $ is suitably  bounded from below, as proven in Lemma \ref{ball} in Subsection \ref{subsec:lowerbound}, letting $ \alpha $ be sufficiently large will make the first four terms above small enough. Concerning the last term of \eqref{eq:dyn}, it is related to $ \int G\,d\rho_{t}(v) $, which can be controlled by Lemma \ref{G} in Subsection \ref{subsec:G}. Consequently, we can control $	\|\con-v^{*}\|_2  $ in such a way that (\ref{gronwall}) holds, thereby ensuring exponential decay. 

\color{black}
\begin{remark}\label{innovation}
The framework of the proof strategy was introduced in \cite{FKR}{}. However, in our case, the Laplace principle shown in \cite{FKR} is not directly applicable. There are two challenges in our proof. First, a global valley-like structure around $v^*$ is required to use the Laplace principle. In the constrained case, such a structure does not hold. But one can notice that, on a nearby level set $\{G=u\}$, if a comparable constrained valley-like structure exists for the corresponding constrained minimizer $v_u$, a constrained Laplace principle on $\{G=u\}$ can be established for small $u$. Integrating over all $u$ yields reasonable estimates on $\|\con-v^*\|_2$. Secondly, there is always mass around the unconstrained global minimum. Consequently, increasing $\alpha$ would attract all particles towards the unconstrained global minimum rather than the constrained one. Therefore, one needs to find the balance between the gradient force and the tendency towards the consensus point. 

Lemma \ref{lp} is one of the key contributions of this paper. This lemma characterizes the distribution of the particles. The particles on the constraint set $\{G = 0\}$ are characterized by the first and third terms. The second and fourth terms account for the particles near the constraint, while the final term describes the particles far from the constraint. Notably, Lemma \ref{lp} is independent of the dynamics and is a property of the landscape. 
By integrating the new quantitative Laplace principle \eqref{eq_11} with the drift towards the constrained set (as indicated by the last term in Lemma \ref{energy}), we are able to quantify the tendency towards the constrained global minimum.

\end{remark}

\color{black}

\subsection{Proof of Theorem \ref{main}}\label{subsec:wellprepared}
In this subsection, we present the complete proof of Theorem \ref{main}. For simplicity, in the following, we assume
$\tau_{1}(u)=\tau_{2}(u)=\tau_{3}(u)=\tau_{4}(u)=u$,
where $\tau_1,\tau_2,\tau_3,\tau_4$ are defined in
Assumption~\ref{wellbehave} (C) and Lemma~\ref{C1}.
The proof for general non-negative increasing functions $\tau_i$ satisfying $\lim_{u\rightarrow0}\tau_i(u)=0$ is identical after choosing $u$ and $r$ in Lemma~\ref{lp} sufficiently small. To be more precise, the proof only requires $\tau_1(u)$, $\tau_2(u)+\tau_4(\max\{u,r\})$, $\tau_3(\max\{u,r\})$, and $\tau_3(r)$ to be sufficiently small. Such a simultaneous choice is possible since all $\tau_i$ vanish at zero, and the remaining argument is unchanged.

Before the proof, we need two more lemmas on the lower bound of $ \rt\big(B(v^{*},r)\big) $ and the dynamics of $ \int G \dvt $. 

{\bf Lower bound for $ \rt\big(B(v^{*},r)\big) $}\label{subsec:lowerbound}
We establish a lower bound for $ \rt\big(B(v^{*},r)\big) $, a crucial element for our subsequent application of the Laplace principle. We first define the mollifier $\phi_r(v)$ as follows
\begin{equation}\label{molli}
\HZnew{
\phi_r(v)
=
\begin{cases}
\displaystyle
\exp\left(
1-\dfrac{r^2}
{r^2-\|v-v^*\|_2^2}
\right),
& \text{if }\|v-v^*\|_2<r,\\[1ex]
0,
& \text{else}.
\end{cases}
}
\end{equation}
Then we can prove the following lemma.
\begin{lemma}\label{ball}
	(Proof in Appendix \ref{balldetails})
	\HZnew{
	Let $
	B=\sup_{t\in[0,T]}
	\|v_\alpha(\rho_t)-v^*\|_2<\infty.$ Then, for all $t\in[0,T]$, $
	\rho_t\left(B(v^*,r)\right)
	\geq
	\left(
	\int\phi_r(v)\,d\rho_0(v)
	\right)e^{-at},$
	where
	\[
	a
	=
	\max\left\{
	\dfrac{2\lambda(\sqrt{c}r+B)\sqrt{c}}
	{(1-c)^2r}
	+
	\dfrac{2\sigma^2(cr^2+B^2)(2c+d)}
	{(1-c)^4r^2},
	\dfrac{4\lambda^2}
	{(2c-1)\sigma^2}
	\right\},
	\]
	and $c\in(1/2,1)$ is any constant satisfying $
	(2c-1)c\geq d(1-c)^2.$
	}
\end{lemma}

{\bf Dynamics of $ \int G \dvt $}\label{subsec:G}
In Lemma \ref{energy}, we have gained control over $ \|\con-v^{*}\|_2 $, yet the last term in the dynamics \eqref{eq:dyn} remains to be studied, which we do now. 
\begin{lemma}\label{G}
	(Proof in Appendix \ref{Gdetails})
	\HZnew{
	Given $B,\tilde B,\xi>0$, there exists $\epsilon_0>0$,
	depending only on $B,\tilde B,\xi$ and the fixed model parameter,
	such that, for any $0<\epsilon\leq\epsilon_0$ and any weak
	solution $(\rho_t)_{t\in[0,T]}$ to \eqref{FPK_iso} satisfying $
		\sup_{t\in[0,T]}\|\con-v^{*}\|_{2}\leq B
		$ and $
		\sup_{t\in[0,T]}\ener\leq\tilde B,$
	one has $
		\int G\,d\rho_t(v)
		\leq
		\int G\,d\rho_0(v)+\xi,$ $\forall t\in[0,T].$
	}
\end{lemma}
Now we are ready to prove Theorem \ref{main}. 
\begin{proof}[Proof of Theorem \ref{main}]
	First we use Lemma \ref{energy} to derive the dynamics of $\ener$:
	\HZnew{
	\begin{align*}
		\dfrac{d}{dt}\mathcal{V}\big(\rho_{t}\big)
		\leq&
		-\big(2\lambda-d\sigma^{2}\big)\mathcal{V}\big(\rho_{t}\big)+\sqrt{2}\big(\lambda+d\sigma^{2}\big)
		\sqrt{\mathcal{V}\big(\rho_{t}\big)}
		\|v_{\alpha}\big(\rho_{t}\big)-v^{*}\|_{2}\\
		&+\dfrac{d\sigma^{2}}{2}
		\|v_{\alpha}\big(\rho_{t}\big)-v^{*}\|_{2}^{2},
	\end{align*}
	}
	where the last term on the right-hand-side of Lemma \ref{energy}
	is omitted because of its non-positivity due to
	Assumption \ref{wellbehave} (B1).

	\HZnew{Now we define $C(t)$ as
	\begin{gather}\label{Cdef}
		C(t)
		=
		\min\left\{
		\dfrac{\tau}{2}
		\dfrac{2\lambda-d\sigma^{2}}
		{\sqrt{2}\big(\lambda+d\sigma^{2}\big)},
		\sqrt{
		\tau\dfrac{2\lambda-d\sigma^{2}}
		{d\sigma^{2}}
		}
		\right\}
		\sqrt{\ener}.
	\end{gather}
	}
	As long as
	$\|v_{\alpha}(\rho_t)-v^{*}\|_2\leq C(t)$ is true,
	\HZnew{one has}
	\begin{align*}
		\dfrac{d}{dt}\ener
		&\leq
		-\big(2\lambda-d\sigma^{2}\big)\ener
		+\dfrac{\tau}{2}
		\big(2\lambda-d\sigma^{2}\big)\ener
		+\dfrac{\tau}{2}
		\big(2\lambda-d\sigma^{2}\big)\ener\\
		&=
		-(1-\tau)\big(2\lambda-d\sigma^{2}\big)\ener.
	\end{align*}
	Thus by Gronwall's inequality, as long as the above condition holds,
	\begin{gather}\label{proof_decay}
		\ener
		\leq
		\mathcal{V}(\rho_{0})
		\exp\left(
		-(1-\tau)\big(2\lambda-d\sigma^{2}\big)t
		\right).
	\end{gather}
	\HZnew{We now fix
	$\delta\in\big(0,\mathcal{V}(\rho_0)\big)$ and define
	\begin{gather}\label{Cdelta}
		C_{\delta}
		:=
		\min\left\{
		\dfrac{\tau}{2}
		\dfrac{2\lambda-d\sigma^{2}}
		{\sqrt{2}\big(\lambda+d\sigma^{2}\big)},
		\sqrt{
		\tau\dfrac{2\lambda-d\sigma^{2}}
		{d\sigma^{2}}
		}
		\right\}
		\sqrt{\delta}.
	\end{gather}
	We pick
	\begin{gather}\label{qdelta}
		q(\delta)
		=
		\min\left\{
		\dfrac{1}{4}
		\left(
		\eta\dfrac{C_{\delta}}{8}
		\right)^{1/\mu},
		\dfrac{1}{2}\mathcal{E}_{\infty}
		\right\}.
	\end{gather}
	Since $\mathcal{E}_{r}^{0}\rightarrow0$ as $r\rightarrow0$,
	we can choose $r(\delta)>0$ sufficiently small such that
	\begin{gather}\label{rdelta}
		r(\delta)
		\leq
		\min\left\{
		R_0,\dfrac{q(\delta)}{4}
		\right\},
		\qquad
		\mathcal{E}_{r(\delta)}^{0}
		\leq
		\dfrac{q(\delta)}{4}.
	\end{gather}
	We further set
	\begin{gather}\label{udelta}
		u(\delta)
		=
		\min\left\{
		\dfrac{u_0}{2},
		\dfrac{1}{4}
		\left(
		\eta\dfrac{C_{\delta}}{8}
		\right)^{1/\mu},
		\dfrac{q(\delta)}{4},
		\dfrac{C_{\delta}}{4}
		\right\}.
	\end{gather}
	Notice that for any
	$\tilde u\in[0,u(\delta))$, Lemma \ref{C1} gives
	\begin{align*}
		q(\delta)
		+\mathcal{E}_{r(\delta)}^{\tilde u}
		-\underline{\mathcal{E}_{\tilde u}}
		&\leq
		q(\delta)
		+\mathcal{E}_{r(\delta)}^{0}
		+\max\{\tilde u,r(\delta)\}
		+\tilde u\\
		&\leq
		q(\delta)
		+\dfrac{q(\delta)}{4}
		+\dfrac{q(\delta)}{4}
		+\dfrac{q(\delta)}{4}\\
		&=
		\dfrac{7}{4}q(\delta)
		\leq
		\dfrac{7}{8}\mathcal{E}_{\infty}
		<
		\mathcal{E}_{\infty}.
	\end{align*}
	Thus the assumptions of Lemma \ref{lp} are satisfied with
	$q=q(\delta)$, $r=r(\delta)$, and $u=u(\delta)$.}
    \HZnew{We next prepare the constants used in the lower bound of the mass.
	Since
	$\rho_0(B(v^*,r))>0$ for every $r>0$ and
	$\phi_{r(\delta)}>0$ on $B(v^*,r(\delta)/2)$, one has $
		\int\phi_{r(\delta)}(v)\,d\rho_0(v)>0.$
	Fix $c\in(1/2,1)$ satisfying
	$(2c-1)c\geq d(1-c)^2$ and define
	\begin{align}\label{adelta}
	a(\delta)
	=
	\max\Bigg\{&
	\dfrac{
	2\lambda\big(\sqrt{c}r(\delta)+C(0)\big)\sqrt{c}
	}{
	(1-c)^2r(\delta)
	}+
	\dfrac{
	2\sigma^2
	\big(cr(\delta)^2+C(0)^2\big)(2c+d)
	}{
	(1-c)^4r(\delta)^2
	},
	\dfrac{4\lambda^2}{(2c-1)\sigma^2}
	\Bigg\}.
	\end{align}
	We now fix $\alpha=\alpha(\delta)$. Also let $
		M(\delta)
		:=
		2\sqrt{2\mathcal{V}(\rho_0)}
		+u(\delta).$
	We choose
	\begin{gather}\label{alpha}
		\alpha(\delta)
		=
		\dfrac{2}{q(\delta)}
		\log\left(
		1+
		\dfrac{
		4M(\delta)e^{a(\delta)T^*}
		}{
		C_\delta
		\int\phi_{r(\delta)}\,d\rho_0
		}
		\right).
	\end{gather}
	With this choice,
	\begin{gather}\label{alpha_bound}
		\dfrac{
		M(\delta)e^{a(\delta)T^*}
		}{
		\int\phi_{r(\delta)}\,d\rho_0
		}
		e^{-\alpha(\delta)q(\delta)/2}
		\leq
		\dfrac{C_\delta}{4}.
	\end{gather}
	}
	\HZnew{With this choice of $\alpha$, we define
	\begin{gather}\label{initial}
		I(\delta)
		:=
		\dfrac{
		C_\delta^2u(\delta)
		}{
		256\mathcal{V}(\rho_0)
		}
		\exp\left(
		-2\alpha(\delta)
		\big(
		\overline{\mathcal E}
		-\underline{\mathcal E}
		\big)
		\right).
	\end{gather}
	By assumption, $
		\int G\,d\rho_0(v)\leq I(\delta).$ We now fix $\alpha=\alpha(\delta)$ and choose $\epsilon>0$
	sufficiently small (see~\eqref{epsilon} for a quantification) so that Lemma \ref{G}, with
	$B=C(0)$, $\tilde B=\mathcal{V}(\rho_0)$ and
	$\xi=I(\delta)$, is applicable.}

	\HZnew{We first verify that
	$\|v_\alpha(\rho_0)-v^*\|_2<C(0)$.
	Applying Lemma \ref{lp} with
	$q=q(\delta)$, $r=r(\delta)$, and $u=u(\delta)$,
	the first term on the right-hand side of \eqref{eq_11} satisfies
	\begin{align*}
	2\dfrac{
	\left(
	q(\delta)
	+\mathcal{E}_{r(\delta)}^0
	+u(\delta)
	+\max\{u(\delta),r(\delta)\}
	\right)^\mu
	}{\eta}\leq
	2\dfrac{
	\left(
	\left(
	\eta\dfrac{C_\delta}{8}
	\right)^{1/\mu}
	\right)^\mu
	}{\eta}
	=
	\dfrac{C_\delta}{4}.
	\end{align*}
	Moreover, $
		q(\delta)-\max\{u(\delta),r(\delta)\}
		\geq\dfrac{3}{4}q(\delta)
		>\dfrac{q(\delta)}{2},$
	and $
		q(\delta)-r(\delta)
		\geq\dfrac{3}{4}q(\delta)
		>\dfrac{q(\delta)}{2}.$ Using
	$\rho_0(B(v^*,r(\delta)))
	\geq\int\phi_{r(\delta)}\,d\rho_0$,
	we can bound the sum of the second and third terms by
	\begin{align*}
	&\dfrac{
	e^{-\alpha(\delta)q(\delta)/2}
	}{
	\int\phi_{r(\delta)}\,d\rho_0
	}
	\left(
	\int_{\{G\in(0,u(\delta))\}}
	\|v-v_{G(v)}\|_2\,d\rho_0(v)
	+
	\int_{\{G=0\}}
	\|v-v^*\|_2\,d\rho_0(v)
	\right)\\
	&\leq
	\dfrac{
	M(\delta)
	}{
	\int\phi_{r(\delta)}\,d\rho_0
	}
	e^{-\alpha(\delta)q(\delta)/2}\leq
	\dfrac{
	M(\delta)e^{a(\delta)T^*}
	}{
	\int\phi_{r(\delta)}\,d\rho_0
	}
	e^{-\alpha(\delta)q(\delta)/2}
	\leq
	\dfrac{C_\delta}{4},
	\end{align*}
	where we used
	$\|v_{G(v)}-v^*\|_2\leq G(v)\leq u(\delta)$.
	The fourth term satisfies $
		u(\delta)\leq\dfrac{C_\delta}{4}.$	Finally, we have
	\begin{align*}
	&\int_{\{G(v)\geq u(\delta)\}}
	\dfrac{\|v-v^*\|_2}
	{\|\wt\|_{L^1(\rho_0)}}
	e^{-\alpha(\delta)\mathcal E(v)}
	\,d\rho_0(v)\\
	&\leq
	e^{\alpha(\delta)
	(\overline{\mathcal E}-\underline{\mathcal E})}
	\sqrt{2\mathcal{V}(\rho_0)}
	\sqrt{\rho_0(\{G\geq u(\delta)\})}\\
	&\leq
	e^{\alpha(\delta)
	(\overline{\mathcal E}-\underline{\mathcal E})}
	\sqrt{\dfrac{2\mathcal{V}(\rho_0)}{u(\delta)}}
	\sqrt{\int G\,d\rho_0(v)}\\
	&\leq
	e^{\alpha(\delta)
	(\overline{\mathcal E}-\underline{\mathcal E})}
	\sqrt{\dfrac{2\mathcal{V}(\rho_0)}{u(\delta)}}
	\sqrt{I(\delta)}
	\leq
	\dfrac{C_\delta}{8}.
	\end{align*}
	Therefore,
	\begin{gather}\label{initial_bootstrap}
		\|v_\alpha(\rho_0)-v^*\|_2
		\leq
		\dfrac{7}{8}C_\delta
		<C(0),
	\end{gather}
	where the last inequality follows from
	$\mathcal{V}(\rho_0)>\delta$.}

    \HZnew{Notice that, by Assumption~\ref{wellbehave} (A1) and 
    $\rho_t\in \mathcal{C}([0,T],\mathcal{P}_4(\mathbb{R}^d))$, both
    $t\mapsto \mathcal{V}(\rho_t)$ and $t\mapsto v_\alpha(\rho_t)$ are continuous.
    As a result, $C(t)$ is also continuous.} 
    
	\HZnew{We now define}
	\begin{gather*}
		T_{\alpha,\epsilon}
		=
		\sup\Big\{
		t\in[0,T^*]\ \big|\ 
		\mathcal{V}(\rho_{t'})>\delta,\
		\|v_{\alpha}(\rho_{t'})-v^{*}\|_2<C(t')
		\text{ for all }t'\in[0,t]
		\Big\}.
	\end{gather*}
	\HZnew{By \eqref{initial_bootstrap},
	$T_{\alpha,\epsilon}>0$.}
	As long as
	$\|v_{\alpha}(\rho_{t'})-v^{*}\|_2\leq C(t')$ is true,
	\eqref{proof_decay} holds.

	\HZnew{By the above continuity, if
    $T_{\alpha,\epsilon}<T^*$, then
    \[
    \mathcal{V}(\rho_{T_{\alpha,\epsilon}})\geq\delta,
    \qquad
    \|v_\alpha(\rho_{T_{\alpha,\epsilon}})-v^*\|_2
    \leq
    C(T_{\alpha,\epsilon}).
    \]
    If both inequalities in the definition of $T_{\alpha,\epsilon}$ are strict at
    $T_{\alpha,\epsilon}$, then by continuity, they will remain strict
    on a larger time interval, which contradicts the definition of
    $T_{\alpha,\epsilon}$. Thus, if
    $T_{\alpha,\epsilon}<T^*$, either
    \[
    \mathcal{V}(\rho_{T_{\alpha,\epsilon}})=\delta,
    \]
    or
    \[
    \mathcal{V}(\rho_{T_{\alpha,\epsilon}})>\delta
    \qquad\text{and}\qquad
    \|v_\alpha(\rho_{T_{\alpha,\epsilon}})-v^*\|_2
    =
    C(T_{\alpha,\epsilon}).
    \]
    So there are only 3 cases as described below.}\\
    
	{\it Case 1} $\left(T_{\alpha,\epsilon}=T^{*}\right)$.

	\HZnew{Taking $t\uparrow T^*$ in \eqref{proof_decay}, one obtains
	\[
		\mathcal{V}(\rho_{T^{*}})
		\leq
		\mathcal{V}(\rho_{0})
		\exp\left(
		-(1-\tau)
		(2\lambda-d\sigma^{2})T^{*}
		\right)
		=
		\delta.
	\]
	}
	So we have
	$\min_{t\in[0,T^{*}]}\mathcal{V}(\rho_t)\leq\delta$.
	This completes the proof.\\

	{\it Case 2}
	$\left(
	T_{\alpha,\epsilon}<T^{*}
	\text{ and }
	\mathcal{V}(\rho_{T_{\alpha,\epsilon}})=\delta
	\right)$.

	In this case, it is clear that $
		\min_{t\in[0,T^{*}]}\mathcal{V}(\rho_t)
		\leq
		\mathcal{V}(\rho_{T_{\alpha,\epsilon}})
		=
		\delta,$
	which completes the proof.\\

	{\it Case 3}
	$\left(
	T_{\alpha,\epsilon}<T^{*},\
	\mathcal{V}(\rho_{T_{\alpha,\epsilon}})>\delta
	\text{ and }
	\|v_{\alpha}(\rho_{T_{\alpha,\epsilon}})-v^{*}\|_2
	=
	C(T_{\alpha,\epsilon})
	\right)$.

	Case 3 is the only non-trivial case.
	We now show that the choices of $\alpha$ (see~\eqref{alpha_bound}) and $\epsilon$ (see~\eqref{epsilon}) above
	make Case 3 impossible.

	\HZnew{Since $\mathcal{V}(\rho_{\ti})>\delta$, one has
	\begin{gather}\label{Cdelta_case3}
		C(\ti)>C_\delta.
	\end{gather}
	Moreover, for $t\in[0,\ti]$, \eqref{proof_decay} implies
	$\mathcal{V}(\rho_t)\leq\mathcal{V}(\rho_0)$, and hence $
		\sup_{t\in[0,\ti]}
		\|v_\alpha(\rho_t)-v^*\|_2
		\leq C(0).$
	Therefore, Lemma \ref{ball}, with $r=r(\delta)$ and $B=C(0)$,
	gives
	\begin{gather}\label{mass_case3}
		\rti\big(B(v^*,r(\delta))\big)
		\geq
		\left(
		\int\phi_{r(\delta)}\,d\rho_0
		\right)
		e^{-a(\delta)T^{*}}.
	\end{gather}
	}

	\HZnew{We now apply Lemma \ref{lp} to $\rho_{\ti}$ with
	$q=q(\delta)$, $r=r(\delta)$ and $u=u(\delta)$.
	For the first term, one has
	\begin{equation}\label{pf1}
	\begin{aligned}
	&2\dfrac{
	\left(
	q(\delta)
	+\mathcal{E}_{r(\delta)}^0
	+u(\delta)
	+\max\{u(\delta),r(\delta)\}
	\right)^\mu
	}{\eta}\\
	&\leq
	2\dfrac{
	\left(
	\left(
	\eta\dfrac{C_\delta}{8}
	\right)^{1/\mu}
	\right)^\mu
	}{\eta}
	=
	\dfrac{C_\delta}{4}
	<
	\dfrac{C(\ti)}{4}.
	\end{aligned}
	\end{equation}
	}

	\HZnew{For the second and third terms, by
	\eqref{mass_case3},
	$q(\delta)-\max\{u(\delta),r(\delta)\}>q(\delta)/2$,
	and $q(\delta)-r(\delta)>q(\delta)/2$, one has
	\begin{align}
	&\dfrac{
	e^{-\alpha(\delta)
	(q(\delta)-\max\{u(\delta),r(\delta)\})}
	}{
	\rti(B(v^*,r(\delta)))
	}
	\int_{\{G\in(0,u(\delta))\}}
	\|v-v_{G(v)}\|_2\,\dvti
	\notag\\
	&\quad+
	\dfrac{
	e^{-\alpha(\delta)(q(\delta)-r(\delta))}
	}{
	\rti(B(v^*,r(\delta)))
	}
	\int_{\{G=0\}}
	\|v-v^*\|_2\,\dvti
	\notag\\
	&\leq
	\dfrac{
	e^{a(\delta)T^*}
	e^{-\alpha(\delta)q(\delta)/2}
	}{
	\int\phi_{r(\delta)}\,d\rho_0
	}
	\left(
	2\sqrt{2\mathcal{V}(\rho_0)}
	+u(\delta)
	\right)
	\notag\\
	&\leq
	\dfrac{C_\delta}{4}
	<
	\dfrac{C(\ti)}{4}.
	\label{pf23}
	\end{align}
	Here we used $
	\int\|v-v^*\|_2\,\dvti
	\leq
	\sqrt{2\mathcal{V}(\rho_{\ti})}
	\leq
	\sqrt{2\mathcal{V}(\rho_0)},$ $\|v-v_{G(v)}\|_2\leq \|v-v^*\|_2 + \|v^* - v_{G(v)}\|_2$, and, by Assumption \ref{wellbehave} (C1), $
	\|v_{G(v)}-v^*\|_2
	\leq
	G(v)
	\leq
	u(\delta)
	\quad\text{on }\{G\in(0,u(\delta))\}.$
	}

	\HZnew{For the fourth term, one has
	\begin{gather}\label{pf4}
		u(\delta)
		\leq
		\dfrac{C_\delta}{4}
		<
		\dfrac{C(\ti)}{4}.
	\end{gather}
	}

	\HZnew{It remains to estimate the last term in Lemma \ref{lp}.
	Since
	$\underline{\mathcal{E}}\leq\mathcal{E}\leq
	\overline{\mathcal{E}}$, one has
	\begin{align}
	&\int_{\{G(v)\geq u(\delta)\}}
	\dfrac{\|v-v^{*}\|_2}
	{\|\wt\|_{L^1(\rti)}}
	e^{-\alpha(\delta)\mathcal{E}(v)}
	\dvti
	\notag\\
	&\leq
	e^{\alpha(\delta)
	(\overline{\mathcal{E}}-\underline{\mathcal{E}})}
	\int_{\{G\geq u(\delta)\}}
	\|v-v^*\|_2\,\dvti
	\notag\\
	&\leq
	e^{\alpha(\delta)
	(\overline{\mathcal{E}}-\underline{\mathcal{E}})}
	\sqrt{2\mathcal{V}(\rho_0)}
	\sqrt{\rti(\{G\geq u(\delta)\})}
	\notag\\
	&\leq
	e^{\alpha(\delta)
	(\overline{\mathcal{E}}-\underline{\mathcal{E}})}
	\sqrt{
	\dfrac{2\mathcal{V}(\rho_0)}
	{u(\delta)}
	}
	\sqrt{\int G\,\dvti}.
	\label{lastterm}
	\end{align}
	}

	\HZnew{By applying Lemma \ref{G} on $[0,\ti]$ with
	$B=C(0)$, $\tilde B=\mathcal{V}(\rho_0)$ and
	$\xi=I(\delta)$, our choice of $\epsilon$ gives
	\begin{gather}\label{Glemma}
		\int G\,\dvti
		\leq
		\int G\,d\rho_0(v)+I(\delta)
		\leq
		2I(\delta).
	\end{gather}
	Combining \eqref{lastterm}, \eqref{Glemma}, and
	\eqref{initial}, we obtain
	\begin{align}
	&\int_{\{G(v)\geq u(\delta)\}}
	\dfrac{\|v-v^{*}\|_2}
	{\|\wt\|_{L^1(\rti)}}
	e^{-\alpha(\delta)\mathcal{E}(v)}
	\dvti
	\notag\\
	&\leq
	e^{\alpha(\delta)
	(\overline{\mathcal{E}}-\underline{\mathcal{E}})}
	\sqrt{
	\dfrac{2\mathcal{V}(\rho_0)}
	{u(\delta)}
	}
	\sqrt{2I(\delta)}
	=
	\dfrac{C_\delta}{8}
	<
	\dfrac{C(\ti)}{8}.
	\label{lasttermfinal}
	\end{align}
	}

	\HZnew{Combining \eqref{pf1}, \eqref{pf23}, \eqref{pf4},
	and \eqref{lasttermfinal}, we get
	\begin{align*}
		\|v_\alpha(\rho_{\ti})-v^*\|_2
		&\leq
		\dfrac{C(\ti)}{4}
		+\dfrac{C(\ti)}{4}
		+\dfrac{C(\ti)}{4}
		+\dfrac{C(\ti)}{8}\\
		&=
		\dfrac{7}{8}C(\ti)
		<
		C(\ti),
	\end{align*}
	which contradicts the assumption
	$\|v_\alpha(\rho_{\ti})-v^*\|_2=C(\ti)$ in Case 3.
	Therefore, Case 3 cannot happen.}

	Thus we have proved that if all the conditions in
	Theorem \ref{main} are satisfied, the desired decay can be achieved
	with the specified choices of $\alpha$ and $\epsilon$.
\end{proof}

\section{Numerical \HZnew{Scheme and }Experiments}\label{sec:numerical}

In this section, we \HZnew{introduce a linearly implicit discretization} of the continuous model (\ref{constraincon1}), \HZnew{discuss its well-posedness, stability properties, and consistency with the continuous particle dynamics, and then present numerical experiments.} Throughout this section, we use the anisotropic version (\ref{aniso}). 

\subsection{\HZnew{Discretization and numerical properties}}\label{subsec:algorithm}
First, one notices that in the time-continuous model (\ref{constraincon1}), the forcing term $\frac{1}{\epsilon}\nabla G$   needs to be relatively large for the particles to remain near the constraint set. However, a straightforward explicit scheme of the dynamics requires the time step $\gamma$ to be of the same order as $\epsilon$. This implies that as $\epsilon$ approaches zero, the algorithm becomes expensive. On the other hand, directly making the stiff term $\frac{1}{\epsilon} \nabla G(V^{i}_{k+1})$ implicit enhances numerical  stability, but it becomes computationally challenging for complex constraints, \HZnew{as it requires solving a nonlinear equation at every time step for general nonlinear constraints.} To address this, \HZnew{we consider a linearly implicit approximation.}

The key idea is to employ Taylor expansion to approximate the term $\nabla G(V^{i}_{k+1})$ in the implicit algorithm with its first-order approximation. 
\begin{equation*}
	\begin{aligned}
		\Xj_{k+1} =& \Xj_k - \lambda \gamma(\Xj_k -  v_{\alpha}(\hat{\rho}_{k}))\\
        &- \l(\frac{\gamma}{\e}\nb G(V^j_k) + \frac{\g}{\e}\nb^2 G(V^j_k) (V_{k+1}^j - V^j_k)\r) - \sigma\sqrt{\gamma} (\Xj_k - v_{\alpha}(\hat{\rho}_{k})) \odot z_k^j,
	\end{aligned}
\end{equation*}
which leads to the following constrained CBO update,
\begin{equation}\label{eq:implicit1}
	\begin{aligned}
		\Xj_{k+1} = \Xj_k 
        - \l[I+\frac{\gamma}{\e}\nb^2G(\Xj_k)\r]^{-1}\Bigg(&\lambda \gamma(\Xj_k -  v_{\alpha}(\hat{\rho}_{k}))\\
        &+ \frac{\gamma}{\e}\nb G(\Xj_k) + \sigma\sqrt{\gamma} (\Xj_k - v_{\alpha}(\hat{\rho}_{k})) \odot z_k^j\Bigg).
	\end{aligned}
\end{equation}
\HZnew{Recalling that $G(v)=\sum_{i=1}^{m}g_i(v)^2$, one has $
    \nb G(v)=2\sum_{i=1}^{m}g_i(v)\nb g_i(v),$ and $
    \nb^2G(v)
    =
    2\sum_{i=1}^{m}
    \left[
    \nb g_i(v)\nb g_i(v)^{\top}
    +
    g_i(v)\nb^2g_i(v)
    \right].$}
Here $\gamma$ is the time step, and $ \odot $ is a point-wise multiplication, i.e., the i-th component of $x\odot y$ is $x_iy_i$. Here $V^j_{k}$ approximates the space location of the $j$-th particle at time $t = k\g$, and $z_k^j$ is a $d$-dimensional random variable following a standard normal distribution $\mN(0,\mathbb{I}_d)$. During different steps and for different particles, $z_k^j$ is sampled independently. \HZnew{We summarize the complete implementation in Algorithm~\ref{algo}.}
\begin{algorithm}
	\caption{Constrained CBO Algorithm}
	\label{algo}
	\textbf{Initialization:} Choose hyperparameters $\epsilon$, $\alpha$, time step $\gamma$, stopping threshold $\epsilon_{\text{stop}}$, and sample size $N$. Sample $N$ particles $\Xj$ from distribution $\rho_0(v)$.
	
	\begin{algorithmic}[1]
		\While{$\frac{1}{dN}\sum_{j=1}^N \|V^j - v_{\alpha}(\hat{\rho})\|^2 > \epsilon_{\text{stop}}$}
		\State Calculate $v_{\alpha}(\hat{\rho})$:
		\[
		v_{\alpha}(\hat{\rho})= \frac{1}{Z}\sum_{j=1}^N \mu_jV^{j}, \quad \text{with} \quad Z =  \sum_{j=1}^Ne^{-\alpha\mE(V^{j})}, \quad \mu_j = e^{-\alpha \mE(V^{j})}
		\]
		
		\State Update each particle's position $\{\Xj\}_{j = 1}^N$:\label{step2}
		\begin{align*}
		\ds \Xj \leftarrow \Xj - \l[I+\frac{\g}{\epsilon}\sum_{i=1}^m\nb^2\l[g^2_i(V^j)\r]\r]^{-1}\Bigg(&\lam \g (\Xj - v_{\alpha}(\hat{\rho}))  
        + \frac{\g}{\epsilon}\sum_{i=1}^m\nb\l[g_i^2(\Xj)\r] \\
        &+ \sigma\sqrt{\gamma} (\Xj - v_{\alpha}(\hat{\rho})) \odot z^j\Bigg),
		\end{align*}
		where $z^j\sim \mathcal{N}(0,\mathbb{I}_d)$.
		\EndWhile
		\State Output $\bxs, \mE(\bxs)$
	\end{algorithmic}
\end{algorithm}

\HZnew{Before we analyze the resulting update, we first clarify the different regularity requirements for the continuous model, the convergence analysis, and the numerical implementation introduced here. The continuous particle dynamics and the mean-field analysis in Section~\ref{sec:meanfield} only involve $\nb G$ and are studied under Assumption~\ref{assump1}. The convergence analysis in Section~\ref{longtime} additionally requires the conditions in Assumption~\ref{wellbehave}, but still only requires first-order regularity of $G$, namely $G\in\mathcal{C}^{1,1}_{*}(\R^d)$. In contrast, the linearly implicit update~\eqref{eq:implicit1} involves the Hessian $\nb^2G$ and therefore requires additional second-order information for its classical implementation. This also introduces additional computational cost, since each particle update requires evaluating the Hessian and solving a linear system. }

\HZnew{Now we are ready to present the analysis for the proposed linear implicit scheme. For a general nonconvex $G$, the matrix $
I+\frac{\gamma}{\epsilon}\nb^2G(V_k^j) $ need not be invertible for arbitrary $\gamma$ and $\epsilon$. We therefore first establish sufficient conditions under which the update is well defined, and then discuss the stability properties and consistency of the proposed discretization.}

\HZnew{The following proposition provides a sufficient condition for the
linear system/inversion in \eqref{eq:implicit1} to be well defined.}
\HZnew{\begin{proposition}\label{prop:implicit_inverse}
(Proof in Appendix~\ref{appendix:invert}) Let $\mathcal{U}\subseteq\mathbb{R}^{d}$ be an open set and suppose
$g_i\in\mathcal{C}^{2}(\mathcal{U})$, $i=1,\ldots,m$. Assume there exist
constants $M_i>0$ such that
\begin{gather}\label{eq:hessian_g_bound}
    \sup_{v\in\mathcal{U}}\|\nb^{2}g_i(v)\|_{\mathrm{op}}
    \leq M_i,\qquad i=1,\ldots,m.
\end{gather} Let $
    M:=\left(\sum_{i=1}^{m}M_i^{2}\right)^{1/2},$ and $
    g(v):=(g_1(v),\ldots,g_m(v))^{\top}.$ Then, for every $v\in\mathcal{U}$,
\begin{gather}\label{eq:hessian_lower}
    \nb^{2}G(v)\succeq -2M\|g(v)\|_2 I.
\end{gather}
Consequently, if
\begin{gather}\label{eq:inverse_condition}
    \frac{2\g}{\e}M\|g(v)\|_2<1,
\end{gather}
then the matrix $
    I+\frac{\g}{\e}\nb^{2}G(v)$ is positive definite and hence invertible. Moreover,
\begin{gather}\label{eq:inverse_bound}
    \left\|
    \left[I+\frac{\g}{\e}\nb^{2}G(v)\right]^{-1}
    \right\|_{\mathrm{op}}
    \leq
    \frac{1}{
    1-\frac{2\g}{\e}M\|g(v)\|_2}.
\end{gather}
In particular, if $\|g(v)\|_2\leq c\e$ for some $c>0$, then the
condition $
    2\g Mc<1$ guarantees that the matrix in \eqref{eq:implicit1} is invertible, with
\begin{gather}\label{eq:inverse_tube_bound}
    \left\|
    \left[I+\frac{\g}{\e}\nb^{2}G(v)\right]^{-1}
    \right\|_{\mathrm{op}}
    \leq
    \frac{1}{1-2\g Mc},
\end{gather}
where the bound is independent of $\e$.
\end{proposition}}
\HZnew{\begin{remark}\label{remark:convex_stability}
If $G$ is convex, then $\nb^{2}G(v)\succeq 0$ for every $v$. Therefore,
for arbitrary $\g,\e>0$, $
    I+\frac{\g}{\e}\nb^{2}G(v)\succeq I,$ and consequently
\begin{gather}\label{eq:convex_inverse_bound}
    \left\|
    \left[I+\frac{\g}{\e}\nb^{2}G(v)\right]^{-1}
    \right\|_{\mathrm{op}}
    \leq 1.
\end{gather}
Thus, in the convex case, the linear system/inversion in \eqref{eq:implicit1} is
well defined without any restriction on $\g$ and $\e$. We emphasize,
however, that for a general non-quadratic convex function $G$,
\eqref{eq:convex_inverse_bound} only controls the linearized implicit solve
and does not by itself imply unconditional stability of the update.
\end{remark}}
\HZnew{\begin{remark}\label{remark:nonconvex_counterexample}
For a general nonconvex constraint function, the condition in
Proposition~\ref{prop:implicit_inverse} cannot in general be removed.
For example, consider the one-dimensional constraint $
    g(v)=v^{2}-1,
    $ with $
    G(v)=(v^{2}-1)^{2}.$
Then $
    G''(v)=12v^{2}-4$,
and in particular $G''(0)=-4$. Therefore, at $v=0$, the matrix appearing
in the linearly implicit update reduces to $
    1+\frac{\g}{\e}G''(0)
    =
    1-\frac{4\g}{\e},$ which is singular when ${\g}/{\e}=1/4.$
Thus, unconditional invertibility cannot be guaranteed for general
nonconvex constraints. Consequently,
additional conditions such as those in
Proposition~\ref{prop:implicit_inverse} are necessary to guarantee that
the linearly implicit update is well defined.
\end{remark}}

\HZnew{Although unconditional stability cannot be expected for general nonlinear
constraints, the advantage of the linearly implicit scheme can be
characterized explicitly for affine equality constraints as shown in the below example. In this case, the
Taylor approximation used in deriving~\eqref{eq:implicit1} is exact.}
\begin{example}[Affine equality constraints]\label{example:linear_stability}
\HZnew{Consider the affine constraint
$g(v)=Bv-b$, where $B\in\R^{m\times d}$ and $b\in\R^{m}$, thus
$G(v)=\|Bv-b\|_2^2$. One has
$\nb G(v)=2B^{\top}(Bv-b)$ and $\nb^2G(v)=2B^{\top}B$.
Let $A:=2B^{\top}B\succeq0$ and choose any feasible point $v^{*}$ satisfying
$Bv^{*}=b$. It then follows that
$\nb G(v)=A(v-v^{*})$. Notice that in this case $\nb G$ is affine, and hence
the Taylor expansion used to derive \eqref{eq:implicit1} is exact.}

\HZnew{For simplicity, denote the vanilla CBO terms in
\eqref{eq:implicit1} by
\[
    F_k^j
    :=
    \lambda\g\big(\Xj_k-v_{\alpha}(\hat{\rho}_k)\big)
    +
    \sigma\sqrt{\g}\big(\Xj_k-v_{\alpha}(\hat{\rho}_{k})\big)\odot z^j_k.
\]
Then \eqref{eq:implicit1} can be rewritten as
\begin{align*}
    \Xj_{k+1}-v^{*}
    &=
    \Xj_k-v^{*}
    -
    \left(I+\frac{\g}{\e}A\right)^{-1}
    \left[
        F_k^j+\frac{\g}{\e}A(\Xj_k-v^{*})
    \right]\\
    &=
    \left(I+\frac{\g}{\e}A\right)^{-1}
    \left(\Xj_k-v^{*}-F_k^j\right),
\end{align*}
where the last equality follows from
$I-\frac{\g}{\e}(I+\frac{\g}{\e}A)^{-1}A
=(I+\frac{\g}{\e}A)^{-1}$.
Since $A\succeq0$, one has $
    \left\|
    \left(I+\frac{\g}{\e}A\right)^{-1}
    \right\|_{\mathrm{op}}
    \leq 1$ for any $\g,\e>0$. Therefore, the implicit treatment of the stiff constraint
force introduces no stability restriction relating $\g$ and $\e$ in this
case.}

\HZnew{For comparison, we consider the Euler--Maruyama discretization of
the same dynamics. Starting from the same particle configuration at step
$k$, all the coefficients are evaluated at $\Xj_k$. In particular, using
the same notation $F_k^j$ for the vanilla CBO terms, the
Euler--Maruyama update is
\begin{align*}
    \widetilde{V}_{k+1}^{j}
    =
    \Xj_k-F_k^j-\frac{\g}{\e}\nb G(\Xj_k).
\end{align*}
Here $\widetilde{V}_{k+1}^{j}$ denotes the Euler--Maruyama update from
$\Xj_k$. Since $\nb G(v)=A(v-v^{*})$, subtracting $v^{*}$ from both sides
gives
\begin{align*}
    \widetilde{V}_{k+1}^{j}-v^{*}
    &=
    \Xj_k-v^{*}
    -F_k^j
    -\frac{\g}{\e}A(\Xj_k-v^{*})=
    \left(I-\frac{\g}{\e}A\right)(\Xj_k-v^{*})-F_k^j.
\end{align*}
Therefore, $
    \|\widetilde{V}_{k+1}^{j}-v^{*}\|_2
    \leq
    \left\|I-\frac{\g}{\e}A\right\|_{\mathrm{op}}
    \|\Xj_k-v^{*}\|_2+\|F_k^j\|_2.$ To obtain the same non-amplifying bound as the proposed scheme, one needs
$\|I-\frac{\g}{\e}A\|_{\mathrm{op}}\leq1$. Since $A$ is symmetric and
positive semidefinite, this requires
$\frac{\g}{\e}\lambda_{\max}(A)\leq2$. Thus, the explicit Euler--Maruyama discretization requires $\g=O(\e)$, while the proposed linearly
implicit treatment of the constraint force does not impose such a
restriction for affine equality constraints.}
\end{example}

\HZnew{We next study the consistency of the proposed linearly implicit scheme
with the original continuous particle dynamics. In particular, we establish
a one-step consistency result for fixed $\epsilon>0$.}

\begin{proposition}\label{prop:consistency}
\HZnew{Let Assumption~\ref{assump1} hold and additionally assume
$G\in\mathcal{C}^{2}(\R^d)$. Fix $N\in\mathbb{N}$ and $\e>0$.
For any fixed initialization $V_0\in\R^{Nd}$, let
\[
    V_t
    =
    \left(
    (V_t^{1,N})^T,\ldots,(V_t^{N,N})^T
    \right)^T
\]
be the solution to the finite-particle dynamics \eqref{constraincon}
starting from $V_0$, and let $V_{\g}^{\mathrm{num}}\in\R^{Nd}$ be
obtained by applying one step of \eqref{eq:implicit1} to the same
configuration and using the same Brownian increments. Then, for
$\g>0$ sufficiently small,
\begin{gather}\label{eq:consistency}
    \left(
    \E\left[
    \|V_{\g}^{\mathrm{num}}-V_{\g}\|_2^2
    \right]
    \right)^{1/2}
    \leq C_{\e,N,V_0}\cdot \g,
\end{gather}
where $C_{\e,N,V_0}>0$ is independent of $\g$ but may depend on $\epsilon$.}
\end{proposition}
\begin{proof}
    \HZnew{Please see Appendix~\ref{appendix:consist}}
\end{proof}

\HZnew{The preliminary numerical results in Figure~\ref{fig:converge} are obtained using the above linearly implicit scheme \eqref{eq:implicit1}. The detailed numerical parameters are provided in Appendix~\ref{appendix:fig0}.}
\begin{algorithm}
	\caption{Constrained CBO Algorithm with Independent Noise}
	\label{algo2}
\textbf{Initialization:} Choose suitable hyper-parameters $\e,\alpha$, time step $\gamma$, concentration threshold $\e_{\mathrm{indep}}$, improvement threshold $\e_{\min}$, and independent noise $\s_{\mathrm{indep}}$. Sample $N$ particles $\Xj$ following distribution $\rho_0(v)$ and set $\mathcal{E}^\star$ to be a large constant.
	\small
	\begin{algorithmic}[1]
\While{$|\mathcal{E}(\bxs) - \mathcal{E}^\star| \geq \epsilon_{\min}$}
	\While{$\frac{1}{dN}\sum_{j=1}^N \|V^j - v_{\alpha}(\hat{\rho})\|^2 > \epsilon_{\mathrm{indep}}$}
		\State Calculate $v_{\alpha}(\hat{\rho})$:
		\[
		v_{\alpha}(\hat{\rho})= \frac{1}{Z}\sum_{j=1}^N \mu_jV^{j}, \quad \text{with} \quad Z =  \sum_{j=1}^Ne^{-\alpha\mE(V^{j})}, \quad \mu_j = e^{-\alpha \mE(V^{j})}
		\]
		
		\State Update each particle's position $\{\Xj\}_{j = 1}^N$:
		\begin{align*}
		\ds \Xj \leftarrow \Xj - \l[I+\frac{\g}{\epsilon}\sum_{i=1}^m\nb^2\l[g^2_i(V^j)\r]\r]^{-1}\Bigg(&\lam \g (\Xj - v_{\alpha}(\hat{\rho}))  + \frac{\g}{\epsilon}\sum_{i=1}^m\nb\l[g_i^2(\Xj)\r] \\
        &+ \sigma\sqrt{\gamma} (\Xj - v_{\alpha}(\hat{\rho})) \odot z^j\Bigg),
		\end{align*}
		where $z^j\sim \mathcal{N}(0,\mathbb{I}_d)$.
		\EndWhile
		
		\If{$\mathcal{E}(\bxs) < \mathcal{E}^\star$}
		\State \[\mathcal{E}^\star = \mathcal{E}(\bxs), \qd \bxs^\star = \bxs.\]
		
		\EndIf
\State Each particle does an independent move:
\[
\ds \Xj \leftarrow \Xj + \s_{\text{indep}}\sqrt{\g}\tilde{z}^j, \quad \text{for } 1\leq j\leq N,
\]
where $\{\tilde{z}^j\}_{j=1}^N$ are independently sampled from $\mN(0,\mathbb{I}_d)$.
		\EndWhile
		\State Output $\bxs^\star, \mathcal{E}^\star$.
	\end{algorithmic}
\end{algorithm}

\HZnew{For high-dimensional problems, we also consider a practical variant of Algorithm~\ref{algo}, which is presented in Algorithm~\ref{algo2}. In Algorithm~\ref{algo2}, an additional independent noise is injected after the particle ensemble has sufficiently concentrated. After each concentration, we record the best consensus point obtained so far and add noise to the particles independently before continuing the iteration. The anisotropic diffusion vanishes as particles concentrate at consensus,
so the additional independent noise is used to encourage further exploration
and avoid early concentration. Related concerns about degenerating diffusion also motivate the rescaled CBO
dynamics with nondegenerate diffusion studied in \cite{HKS}, while
Algorithm~\ref{algo2} uses occasional independent noises. The continuous-time results study the
dynamics without these additional noises. An ablation study on this additional independent noise is presented in Section~\ref{subsec:indep_ablation}.}

\subsection{Numerical examples}
\label{numerics}
\subsubsection{A simple example}\label{sec:eq1}

We first test the proposed algorithm on a simple objective function,
\begin{equation}
	\label{eq: numerics 1}
	\begin{aligned}
		\min_{v\in\mathbb{R}^d} &\|v\|_2^2.&
	\end{aligned}
\end{equation}
We consider three types of constraints. The first is a line-segment constraint in $\mathbb{R}^3$ ($d=3$), with
\begin{equation}\label{line_segment}
	g(v) = \mathrm{dist}(v, L),
\end{equation}
where $\mathrm{dist}(\cdot,L)$ denotes the Euclidean distance to the line segment $L$ with endpoints $(1.6,0.2,0.4)$ and $(-0.3,-0.7,0.5)$. The second case is an ellipse, with
\begin{equation}\label{ellipse}
	g(v) = \frac{(v_1+1)^2}{2} + v_2^2 -1 = 0.
\end{equation}
The third case is a line, with
\begin{equation}\label{line}
	g(v) = v_1 + v_2 - 3 = 0.
\end{equation}
The constrained minimizers of those three cases are
\begin{gather*}
\text{Line-segment: } v^{*}=(0.2361,-0.4460,0.4718);\\
\text{Ellipse: }v^* =(\sqrt{2}-1,0);
\qquad
\text{Line: }v^*=(3/2,3/2).
\end{gather*}

\HZnew{For the line-segment constraint, the numerical implementation directly uses $G(v)=\mathrm{dist}(v,L)^2$. Let $a\neq b$ denote the endpoints of $L$, and set $u=b-a$, $e=u/\|u\|_2$, and $
t(v)={u^\top(v-a)}/{\|u\|_2^2}.$
The Euclidean projection of $v$ onto $L$ is given by $
\Pi_L(v)=a+\min\{1,\max\{0,t(v)\}\}\,u,$ and hence $\nb G(v)=2\big(v-\Pi_L(v)\big)$. When $0<t(v)<1$, the projection lies in the interior of the segment, and we use $\nb^2G(v)=2(I-ee^\top)$. In the endpoint regions $t(v)<0$ and $t(v)>1$, we use $\nb^2G(v)=2I$. On the switching hyperplanes $t(v)=0$ and $t(v)=1$, where $G$ is not classically twice differentiable, we select the endpoint-side value $2I$.}

We use Algorithm \ref{algo} with
\[
N=50,\quad
\alpha=50,\quad
\epsilon=0.01,\quad
\lambda=1,\quad
\sigma=1,\quad
\gamma=0.1,\quad
\epsilon_{\mathrm{stop}}=10^{-14},
\]
and the particles are initially set to follow a uniform distribution in $[-3,3]^d$.

\HZnew{For comparison, we use the same Penalized CBO and CB2O implementations described in Appendix~\ref{appendix:fig0}, with the same common parameters and initial distribution whenever applicable. For the Penalized CBO, we use the objective $\mE(v)+\epsilon^{-1}G(v)$, while for CB2O we use $\beta=1/20$.}

We consider the search for the constrained minimizer successful if, when the algorithm finishes, $
\|\bxs-\xs\|_\infty\leq 0.1.$ The distance to $v^*$ is measured using
\begin{gather}\label{def of norm}
	D(v,v^*) = \dfrac{1}{\sqrt{d}}\ll v-\xs\rl
	=
	\left(\dfrac1d\sum_{i=1}^d(v-\xs)_i^2\right)^{1/2}.
\end{gather}
\begin{table}[h]
	\centering
	\caption{\HZnew{Comparison of the proposed method, Penalized CBO, and CB2O for \eqref{eq: numerics 1}. The averaged distance is computed over all $100$ simulations.}}
	\label{table}
	\vspace{-1ex}
	\small
	\setlength{\tabcolsep}{3.5pt}
	\begin{tabular}{|c|cc|cc|cc|}
		\hline
		& \multicolumn{2}{c|}{Proposed}
		& \multicolumn{2}{c|}{Penalized CBO}
		& \multicolumn{2}{c|}{CB2O} \\
		\cline{2-7}
		& Success & Mean $D$
		& Success & Mean $D$
		& Success & Mean $D$ \\
		\hline
		Line segment
		& $100\%$ & $8.31\times10^{-3}$
		& $29\%$ & $1.38\times10^{-1}$
		& $48\%$ & $9.29\times10^{-2}$ \\
		\hline
		Ellipse
		& $99\%$ & $1.71\times10^{-2}$
		& $63\%$ & $7.84\times10^{-2}$
		& $60\%$ & $9.18\times10^{-2}$ \\
		\hline
		Line
		& $100\%$ & $1.10\times10^{-2}$
		& $65\%$ & $9.54\times10^{-2}$
		& $61\%$ & $8.99\times10^{-2}$ \\
		\hline
	\end{tabular}
\end{table}

\HZnew{The success rates and the averaged distances over all $100$ simulations are summarized in Table~\ref{table}. The proposed method achieves success rates of $99\%$--$100\%$ in all three examples, compared with $29\%$--$65\%$ for the Penalized CBO and $48\%$--$61\%$ for CB2O. It also gives a much smaller averaged distance to the constrained minimizer in all three cases.}

\HZnew{In Figure~\ref{fig: eq1_er}, we show the evolution of the objective function value $\mathcal{E}(\bxs)$, the violation of feasibility $\sqrt{G(\bxs)}$, and the distance $D(\bxs,\xs)$. Notice that $\sqrt{G(v)}=\mathrm{dist}(v,L)$ for the line-segment constraint, and $\sqrt{G(v)}=|g(v)|$ for the ellipse and line constraints. The proposed method rapidly drives the consensus point close to both the feasible set and the constrained minimizer in all three examples.}

\begin{figure}
    \centering

    \begin{minipage}[c]{0.18\textwidth}
        \raggedright
        \fontsize{7}{8}\selectfont
        \begin{tabular}{@{}l@{}}
            \textcolor[rgb]{0,0,1}{\rule{1em}{0.8pt}}\
            $\mathcal{E}(v_\alpha)$ \\[0.8ex]
            \textcolor[rgb]{1,0,0}{\rule{1em}{0.8pt}}\
            $\sqrt{G(v_\alpha)}$
        \end{tabular}
    \end{minipage}\hspace{0.02\textwidth}%
    \begin{minipage}[c]{0.80\textwidth}
        \centering
        \subfloat[Line segment]{%
            \includegraphics[width=0.32\linewidth]{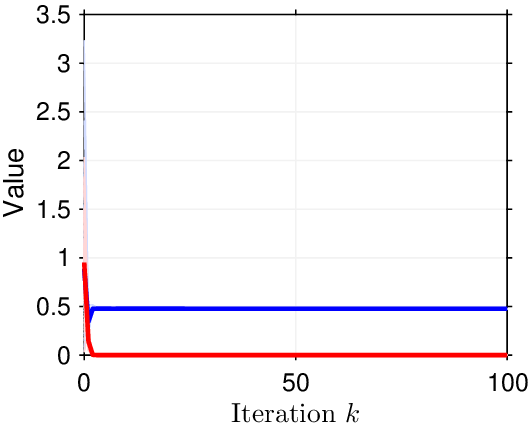}%
        }\hfill%
        \subfloat[Ellipse]{%
            \includegraphics[width=0.32\linewidth]{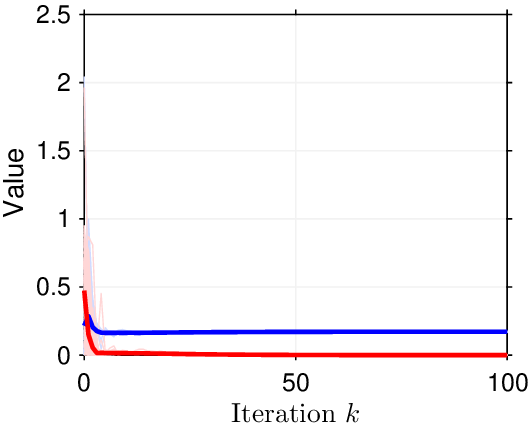}%
        }\hfill%
        \subfloat[Line]{%
            \includegraphics[width=0.32\linewidth]{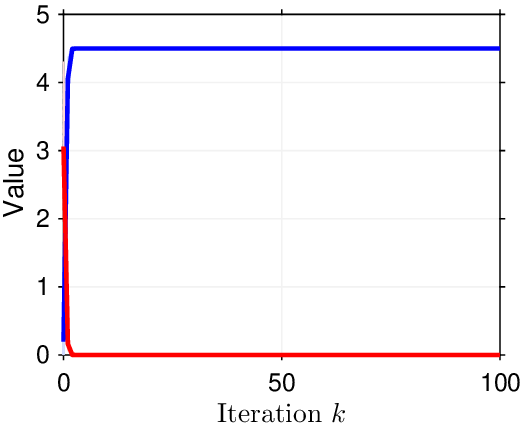}%
        }
    \end{minipage}

    \par\medskip

    \begin{minipage}[c]{0.18\textwidth}
        \raggedright
        \fontsize{7}{8}\selectfont
        \begin{tabular}{@{}l@{}}
            \textcolor[rgb]{0,0,1}{\rule{1em}{0.8pt}}\
            Proposed \\[0.8ex]
            \textcolor[rgb]{1,0,1}{\rule{1em}{0.8pt}}\
            Penalized CBO \\[0.8ex]
            \textcolor[rgb]{0,0.6,0}{\rule{1em}{0.8pt}}\
            CB2O
        \end{tabular}
    \end{minipage}\hspace{0.02\textwidth}%
    \begin{minipage}[c]{0.80\textwidth}
        \centering
        \subfloat[Line segment]{%
            \includegraphics[width=0.32\linewidth]{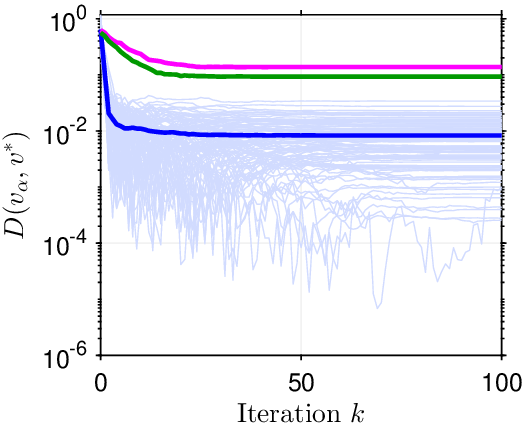}%
        }\hfill%
        \subfloat[Ellipse]{%
            \includegraphics[width=0.32\linewidth]{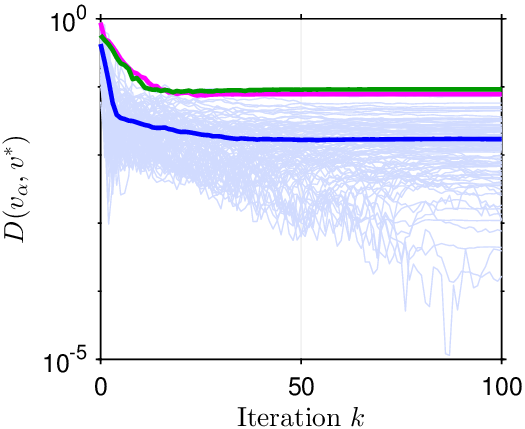}%
        }\hfill%
        \subfloat[Line]{%
            \includegraphics[width=0.32\linewidth]{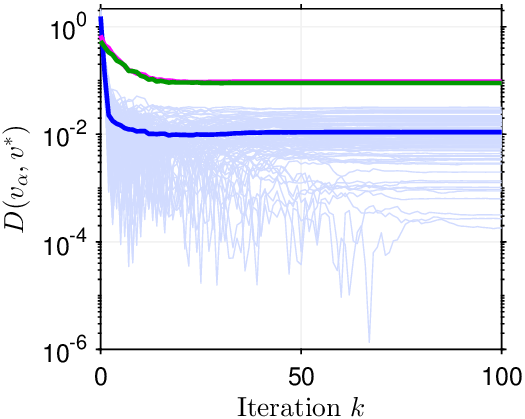}%
        }
    \end{minipage}

    \caption{\HZnew{
    The first, second, and third columns correspond to the
    line-segment constraint \eqref{line_segment}, the ellipse
    constraint \eqref{ellipse}, and the line constraint
    \eqref{line}, respectively.
    The first row shows the evolution of the objective function
    value $\mE(\bxs)$ and the violation of feasibility
    $\sqrt{G(\bxs)}$.
    The second row shows the distance $D(\bxs,\xs)$ defined
    in \eqref{def of norm}.
    In the first row, the light curves show the results of
    $100$ simulations of the proposed method, and the dark
    curves show their averages.
    In the second row, the light blue curves show the results
    of $100$ simulations of the proposed method, while the
    dark curves show the averaged distances for the proposed
    method, Penalized CBO, and CB2O.
    }}
    \label{fig: eq1_er}
\end{figure}

Figure~\ref{fig: evolution} shows the positions of all particles and the consensus point at time steps $k=0,5,50,100$. In all cases, after $5$ steps, most particles have been driven toward the constraint set by the strong constraint force $-\frac1\e\nb G(v)$ and remain close to it thereafter. \HZnew{
The line-segment example satisfies Assumption~\ref{wellbehave} (B), as established by Lemma~\ref{lem:assumpB-distance}. In the ellipse example, however, not all particles concentrate near the consensus point. Instead, some remain close to the infeasible stationary point $\tilde{v}=(-1,0)$ of $G$, where $\nabla g(\tilde{v})=0$ and hence $\nabla G(\tilde{v})=0$, while $G(\tilde{v})=1$. This violates the gradient-dominance condition in Assumption~\ref{wellbehave} (B2). Assumption~\ref{wellbehave} (B1) also fails globally for this example. At $\tilde{v}$, the constraint force $-\epsilon^{-1}\nabla G$ vanishes, although the consensus and diffusion terms need not vanish. In this numerical example, some particles remain near this point, while the consensus point still approaches the constrained minimizer.
}

\begin{figure}
	\centering
    \subfloat[k=0]{\includegraphics[width=0.2\textwidth]{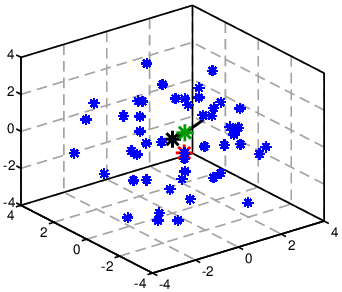}} 
	\subfloat[k=5]{\includegraphics[width=0.2\textwidth]{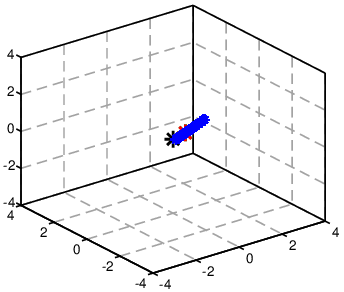}}
	\subfloat[k=50]{\includegraphics[width=0.2\textwidth]{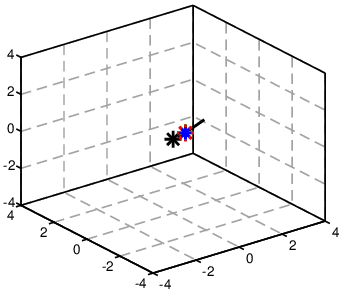}} 
	\subfloat[k=100]{\includegraphics[width=0.2\textwidth]{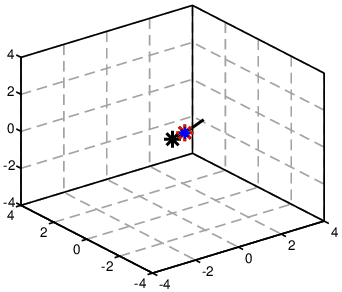}}\\[1ex]
	\subfloat[k=0]{\includegraphics[width=0.2\textwidth]{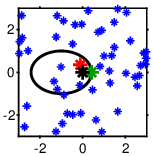}} 
	\subfloat[k=5]{\includegraphics[width=0.2\textwidth]{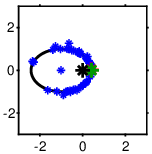}}
	\subfloat[k=50]{\includegraphics[width=0.2\textwidth]{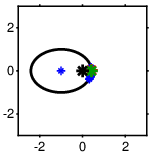}} 
	\subfloat[k=100]{\includegraphics[width=0.2\textwidth]{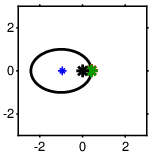}}\\[1ex]
	\subfloat[k=0]{\includegraphics[width=0.2\textwidth]{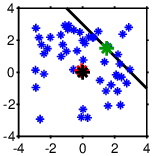}} 
	\subfloat[k=5]{\includegraphics[width=0.2\textwidth]{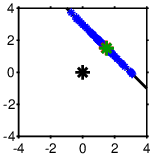}}
	\subfloat[k=50]{\includegraphics[width=0.2\textwidth]{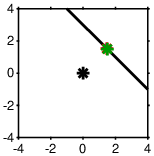}} 
	\subfloat[k=100]{\includegraphics[width=0.2\textwidth]{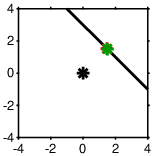}}
	\caption{\HZnew{
Evolution of the particles and the consensus point for \eqref{eq: numerics 1}. The first row corresponds to the line-segment constraint \eqref{line_segment}, the second row to the ellipse constraint \eqref{ellipse}, and the third row to the line constraint \eqref{line}. The snapshots are shown at $k=0,5,50,100$. The particles are shown in blue, the consensus point in red, the constraint set in black, the unconstrained minimizer in black, and the constrained minimizer $v^*$ in green.
}}
	\label{fig: evolution}
\end{figure}

\subsubsection{Ackley function}\label{ackley}
We now test the proposed algorithms on a highly non-convex objective function. Consider the following Ackley function,

\begin{equation}\label{eq:ackley}
	\begin{aligned}
		&\min_v \qd -A\exp\l(-a\sqrt{\frac{b^2}d\ll v - \hv\rl_2^2}\r)- \exp\l(\frac1d\sum_{i=1}^d\cos(2\pi b(v_i - \hv_i))\r)+e^1+A,
	\end{aligned}
\end{equation}
where $b = 1, A = 20, a = 0.1$, and $\hv$ is the global minimizer of the unconstrained problem. The above function in two dimensions is shown in Figure~\ref{fig:cbo_obj}. Here we consider five different constraints,
\begin{align}
&{\text{Case 1.}\qd \mathrm{dist}(v,L)=0, \text{ where $L\subset \mathbb{R}^3$ is the line segment with endpoints}} \notag\\
&\qquad {\text{ $(0.2,0.5,0.7)$ and $(0.5,0.2,0.5)$}}\label{case4}\\
&{\text{Case 2.}\qd \mathrm{dist}(v,S)=0,\text{ where $S$ is the $18$-dimensional ball in $\mathbb{R}^{20}$:}} \notag\\
&\qquad {\text{ $\{v \mid v_1^2+\cdots+v_{18}^2\le 0.25,\ v_{19}=v_{20}=0\}$}}\label{case5}\\
&\text{Case 3.}\qd \|v\|_2^2-1=0\ \text{in $\mathbb{R}^3$ and in $\mathbb{R}^{20}$}\label{case1}\\
&\text{Case 4.}\qd \sum_{i=1}^{d-1} v_i^2 - v_d = 0\ \text{in $\mathbb{R}^3$ and in $\mathbb{R}^{20}$}\label{case2}\\
&\text{Case 5.}\qd 
g(v)=\left(\sum_{i=1}^{d} v_i - 1,\ 
2\sum_{i=1}^{d-1} v_i-\frac12 v_d-\frac12\right)^T=0
\ \text{in $\mathbb{R}^3$}.
\label{case3}
\end{align}
\HZnew{For Case~2, write $v=(x,y)\in\R^{18}\times\R^2$
and set $r=\|x\|_2$. The numerical implementation directly uses
$G(v)=\mathrm{dist}(v,S)^2=\max\{r-1/2,0\}^2+\|y\|_2^2$.
When $r\leq1/2$, we use
$\nb G(v)=(0,2y)$ and
$\nb^2G(v)=\operatorname{diag}(0_{18\times18},2I_2)$.
When $r>1/2$, we use
$\nb G(v)=\big(2(1-\frac{1}{2r})x,2y\big)$ and
$\nb^2G(v)=\operatorname{diag}\big(
2(1-\frac{1}{2r})I_{18}+\frac{xx^\top}{r^3},2I_2\big)$.
At $r=1/2$, where $G$ is not classically twice differentiable,
we use the interior-side Hessian. By Lemma~\ref{lem:assumpB-distance}, Cases~1 and 2 satisfy Assumption~\ref{wellbehave} (B).}
Here, we set $\hv = (0.4,\cdots,0.4)$ so that the unconstrained minimizer is different from the constrained minimizer.

\HZnew{The reference constrained minimizers used in the numerical experiments for the three-dimensional cases are}
\begin{gather*}
\text{Case 1. }v^* = (0.4003,0.2997,0.5665); \qquad
\text{Case 3. }v^* = \frac1{\sqrt{3}}(1,1,1);\\
{\text{Case 4. }v^* = (0.4283,0.4283,0.3669);}
\qquad
\text{Case 5. }v^* = (0.2,0.2,0.6).
\end{gather*}
\HZnew{For the $20$-dimensional cases, we use}
\begin{align*}
&\text{Case 2. }v_i^* = 0.1179,\quad 1\le i\le 18,\qquad
v_{19}^*=v_{20}^*=0;\\
&\text{Case 3. }v_i^* = \frac1{\sqrt{20}},\quad 1\le i\le 20;\\
&{\text{Case 4. }v_i^* = 0.3542,\quad 1\le i\le 19,\qquad
v_{20}^* = 2.3839.}
\end{align*}

\HZnew{For the $20$-dimensional Case~2, the reference solution is computed by first projecting the unconstrained global minimizer onto $S$ and then refining the solution using the MATLAB Optimization Toolbox. Since an accurate reference solution for the $20$-dimensional Case~5 is difficult to obtain, we consider Case~5 only in $\mathbb{R}^3$.}

For the $3$-dimensional Ackley function, we use Algorithm~\ref{algo} with
\[
N = 100,\quad
\a = 50,\quad
\epsilon = 0.01,\quad
\lam = 1,\quad
\s = 1,\quad
\g = 0.1,\quad
\epsilon_{\text{stop}} = 10^{-14}.
\]

For the $20$-dimensional Ackley function, we use Algorithm~\ref{algo2} with
\[
\begin{aligned}
&N = 100,\quad
\a = 50,\quad
\epsilon = 0.01,\quad
\lam = 1,\quad
\s = 1,\quad
\g = 0.1,\quad
\epsilon_{\mathrm{indep}} = 10^{-5},\\
&\text{Case 2.}\quad \epsilon_{\min} = 0.05,\quad \s_{\mathrm{indep}} = 0.3;\\
&\text{Case 3.}\quad \epsilon_{\min} = 0.01,\quad \s_{\mathrm{indep}} = 0.3;\\
&\text{Case 4.}\quad \epsilon_{\min} = 0.001,\quad \s_{\mathrm{indep}} = 1.
\end{aligned}
\]
All particles initially follow $\Xj\sim\mathrm{Unif}[-3,3]^d$.

\begin{table}[h]
	\centering
	\caption{\HZnew{Comparison of the proposed method, Penalized CBO, and CB2O for the Ackley function \eqref{eq:ackley}. The mean distance is computed over all $100$ simulations.}}
	\label{table2}
	\vspace{-1ex}
	\small
	\setlength{\tabcolsep}{3.2pt}
	\begin{tabular}{|c|cc|cc|cc|}
		\hline
		& \multicolumn{2}{c|}{Proposed}
		& \multicolumn{2}{c|}{Penalized CBO}
		& \multicolumn{2}{c|}{CB2O}\\
		\cline{2-7}
		& Success & Mean $D$
		& Success & Mean $D$
		& Success & Mean $D$\\
		\hline
		Case 1, $d=3$
		& $100\%$ & $2.75\times10^{-3}$
		& $99\%$ & $2.41\times10^{-2}$
		& $100\%$ & $1.46\times10^{-2}$\\
		\hline
		Case 3, $d=3$
		& $100\%$ & $8.27\times10^{-3}$
		& $59\%$ & $3.62\times10^{-1}$
		& $57\%$ & $3.18\times10^{-1}$\\
		\hline
		Case 4, $d=3$
		& $100\%$ & $4.77\times10^{-3}$
		& $99\%$ & $1.97\times10^{-2}$
		& $87\%$ & $6.64\times10^{-2}$\\
		\hline
		Case 5, $d=3$
		& $100\%$ & $2.83\times10^{-3}$
		& $38\%$ & $2.73\times10^{-1}$
		& $85\%$ & $4.71\times10^{-2}$\\
		\hline
		Case 2, $d=20$
		& $96\%$ & $2.88\times10^{-2}$
		& $0\%$ & $1.69\times10^{-1}$
		& $0\%$ & $1.92\times10^{-1}$\\
		\hline
		Case 3, $d=20$
		& $100\%$ & $1.56\times10^{-2}$
		& $0\%$ & $2.56\times10^{-1}$
		& $0\%$ & $2.08\times10^{-1}$\\
		\hline
		Case 4, $d=20$
		& $94\%$ & $4.34\times10^{-2}$
		& $0\%$ & $4.38\times10^{-1}$
		& $0\%$ & $4.27\times10^{-1}$\\
		\hline
	\end{tabular}
\end{table}

\begin{figure}
	\centering
	{\fontsize{7.5}{8.5}\selectfont
	\textcolor[rgb]{0,0,1}{\rule{1.4em}{0.8pt}}\ Proposed\qquad
	\textcolor[rgb]{1,0,1}{\rule{1.4em}{0.8pt}}\ Penalized CBO\qquad
	\textcolor[rgb]{0,0.6,0}{\rule{1.4em}{0.8pt}}\ CB2O}\par\smallskip
    \subfloat[$d=3$ case 1]{\includegraphics[width=0.32\textwidth]{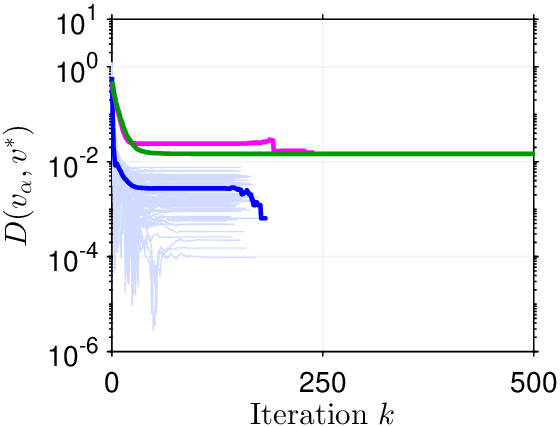}}
	\subfloat[$d=3$ case 3]{\includegraphics[width=0.32\textwidth]{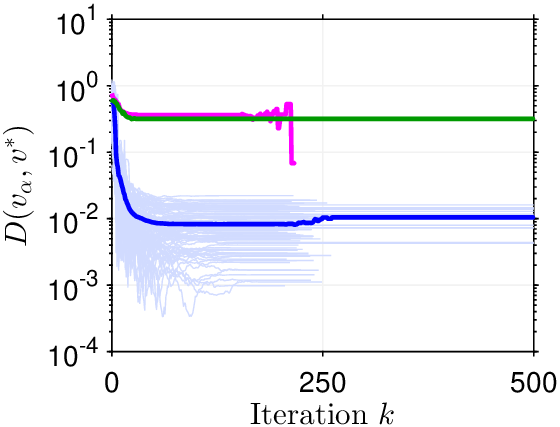}} 
	\subfloat[$d=3$ case 4]{\includegraphics[width=0.32\textwidth]{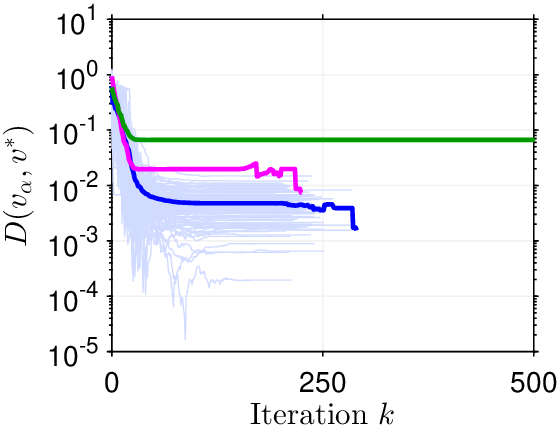}}\\
	\subfloat[$d=3$ case 5]{\includegraphics[width=0.32\textwidth]{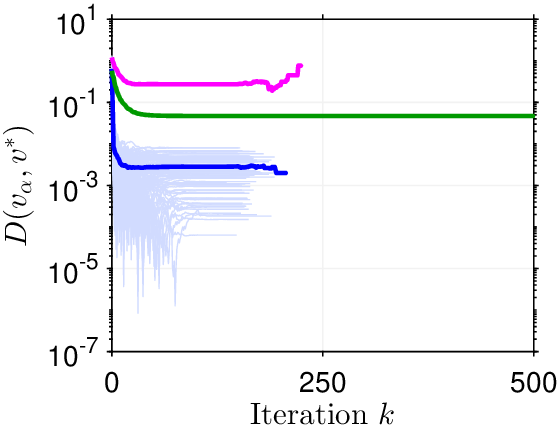}}
    \subfloat[$d=20$ case 2]{\includegraphics[width=0.32\textwidth]{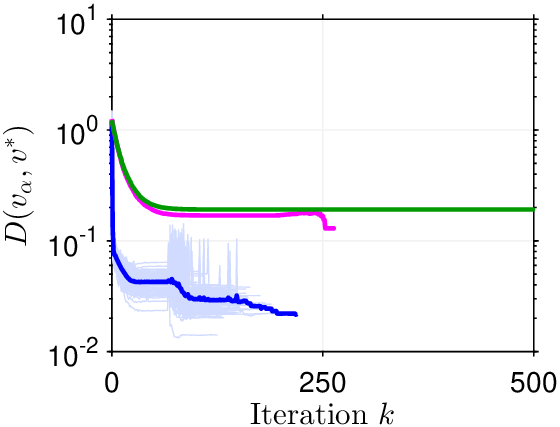}}
	\subfloat[$d=20$ case 3]{\includegraphics[width=0.32\textwidth]{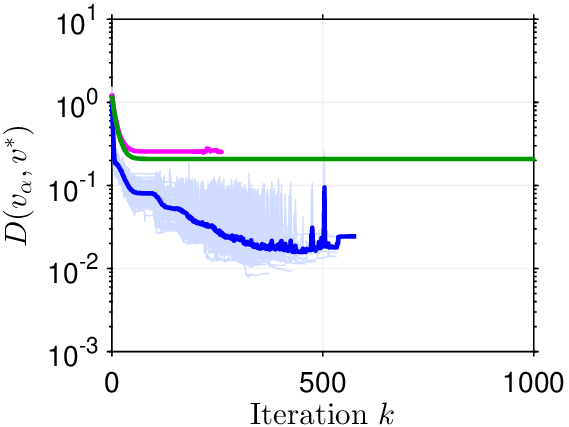}} \\ 
	\subfloat[$d=20$ case 4]{\includegraphics[width=0.32\textwidth]{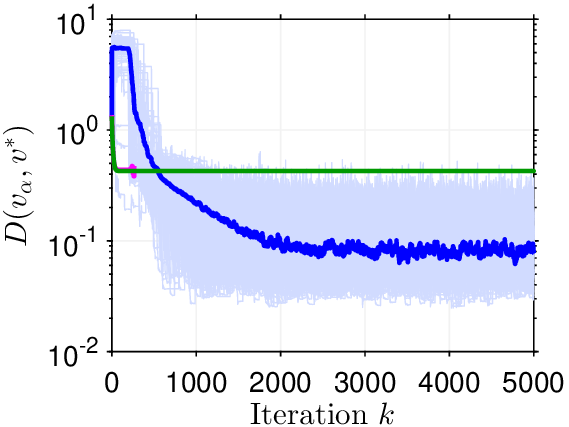}} 
	\caption{The evolution of the distance $D(\bxs,\xs)$ between the consensus point and the reference constrained minimizer. The objective function is the Ackley function \eqref{eq:ackley} and the constraints are \eqref{case4} - \eqref{case3}. The light lines are the results from $100$ simulations, while the dark lines are the average values. The horizontal range is chosen separately for each case to better show the convergence behavior, while all simulations are run with the original maximum number of iterations. 
    }
	\label{fig: eq2_er}
\end{figure}
\HZnew{For comparison, we use the Penalized CBO and CB2O methods described in Appendix~\ref{appendix:fig0}. In particular, the Penalized CBO is applied to $
\mE^\epsilon(v)=\mE(v)+\frac1\epsilon G(v),$ while for CB2O we take $G$ as the lower-level objective and $\mE$ as the upper-level objective with $\beta=1/20$. The CB2O comparison uses anisotropic diffusion.}

\HZnew{For all three methods, we consider a simulation successful if the output satisfies $
\|\bxs-v^*\|_\infty\leq 0.1.$ The distance to the reference constrained minimizer is measured by $D(\bxs,v^*)$ defined in \eqref{def of norm}. The success rate and the mean distance reported below are computed over $100$ simulations, and the mean distance is averaged over all $100$ simulations.}

\HZnew{The results in Table~\ref{table2} show that the proposed method achieves a success rate between $94\%$ and $100\%$ in all seven cases and gives the smallest mean distance to the reference constrained minimizer in every case. For all four $3$-dimensional cases, the proposed method achieves a $100\%$ success rate. The difference becomes larger in higher dimensions: for the three $20$-dimensional cases, the proposed method achieves success rates of $96\%$, $100\%$, and $94\%$, respectively, but neither Penalized CBO nor CB2O produces a successful run under the same success criterion. These examples indicate that the constraint forcing term remains effective for the more challenging high-dimensional constrained problems considered here.}

\HZnew{Figure~\ref{fig: eq2_er} further shows the evolution of $D(\bxs,v^*)$. The individual trajectories are shown for the proposed method, together with the mean trajectories for the proposed method, Penalized CBO, and CB2O. In several cases, the comparison methods stabilize away from the reference minimizer, which is consistent with the lower success rates reported in Table~\ref{table2}.}

\subsubsection{Thomson's Problem}\label{thompson}
The Thomson problem involves determining the positions for $k$ electrons on a sphere in a way that minimizes the electrostatic interaction energy between each pair of electrons with equal charges. The associated constrained optimization problem is formulated as follows,
\begin{equation*}
	\begin{aligned}
		&\min \qd \mathcal{E}(v_1, ..., v_k) = \frac{1}{k}\sum_{i < j}\frac{1}{\ll v_i - v_j \rl_2} \\
		&\text{s.t.}\qquad 
g(v)=\big(\|v_1\|_2^2-1,\ldots,\|v_k\|_2^2-1\big)^T=0.
	\end{aligned}
\end{equation*}

\HZnew{Since the problem is rotationally invariant, in the numerical implementation we fix the first electron at $
v_1=(1,0,0),$ and optimize only over the remaining $k-1$ electrons. Hence, for a system of $k$ physical electrons, the optimization dimension is $3(k-1)$, while the physical coordinate dimension is $3k$. The numerical problem has $k-1$ active equality constraints, $
g(v)=\big(\|v_2\|_2^2-1,\ldots,\|v_k\|_2^2-1\big)^T.$}

\HZnew{It is worth noting that the Thomson problem does not satisfy Assumption~\ref{wellbehave} (B1) globally. Therefore, Theorem~\ref{main} does not directly apply to this example. As discussed in Remark~\ref{rmk:ass G}, Assumption~\ref{wellbehave} (B1) is a sufficient condition used in the convergence analysis rather than a requirement for defining the dynamics or the numerical algorithm. The numerical results below therefore illustrate the performance of the proposed method beyond the setting covered by the current convergence theory.}

We use Algorithm~\ref{algo2} with
\[
\begin{aligned}
&N=50,\quad \a=50,\quad \epsilon=0.01,\quad
\lam=\s=1,\quad \g=0.1,\\
&\epsilon_{\mathrm{indep}}=10^{-14},\quad
\epsilon_{\min}=0.01,\quad
\s_{\mathrm{indep}}=0.3,
\end{aligned}
\]
with a maximum of $2000$ iterations.

\HZnew{For each free electron, the initial position is sampled on the unit sphere using $
\theta\sim\mathrm{Unif}(0,\pi),$ $
\phi\sim\mathrm{Unif}(0,2\pi),$
and $
v=(\sin\theta\cos\phi,\sin\theta\sin\phi,\cos\theta).$
The first electron remains fixed at $(1,0,0)$.} 

\HZnew{For the Thomson problem, the linearly implicit update has a
blockwise implementation due to the separable structure of the sphere
constraints. In particular, for $k=470$, the $1407$-dimensional implicit
system for each particle can be decomposed into $469$ independent $3\times3$
linear systems, and no dense $1407\times1407$ matrix needs to be formed or
inverted. The corresponding implementation, linear system condition numbers, and computational cost are provided in Appendix~\ref{appendix:thomson_diagnostics}.}

\HZnew{We consider $
k=2,3,8,15,56,470,$
corresponding to optimization dimensions $
3,\ 6,\ 21,\ 42,\ 165,\ 1407,$ respectively.}

We define the relative objective error by
\begin{equation}\label{def of rel er}
	\mathrm{RelErr}(v)
	=
	\frac{|\mathcal{E}(v)-\mathcal{E}(v^*)|}
	{\mathcal{E}(v^*)}.
\end{equation}

\HZnew{When we fix the first electron, the constraint is $
g(v)
=
\big(\|v_2\|_2^2-1,\ldots,\|v_k\|_2^2-1\big)^T,$
and therefore $
\sqrt{G(v)}=\|g(v)\|_2.$
We consider a simulation successful if the output consensus point $\bxs$ satisfies $
\mathrm{RelErr}(\bxs)\leq0.05,$ and $
\|g(\bxs)\|_1\leq10^{-3}.$}


\HZnew{The results are summarized in Table~\ref{table3}. Except for one simulation in the $k=2$ case, all simulations satisfy the prescribed success criterion. In particular, the proposed method achieves a $100\%$ success rate for $k=56$ and $k=470$, corresponding to optimization dimensions $165$ and $1407$, respectively. For all six cases, the mean relative objective error remains below $2.1\times10^{-2}$ and the mean constraint violation $\sqrt{G}=\|g\|_2$ is of order $10^{-9}$ or smaller. These results indicate that the proposed method remains effective for
the high-dimensional Thomson problems considered here.}

\HZnew{Figure~\ref{fig: eq3} shows the evolution of the relative objective error over the $100$ simulations. Since Algorithm~\ref{algo2} may terminate different simulations at different iterations, the mean at iteration $n$ is computed only over the simulations that remain active at that iteration.}

\begin{table}[h]
	\centering
	\caption{\HZnew{Results of Algorithm~\ref{algo2} for the Thomson problem. The mean relative error and mean $\sqrt{G}=\|g\|_2$ are computed over successful simulations, while the mean number of iterations is computed over all $100$ simulations.}}
	\label{table3}
	\small
	\setlength{\tabcolsep}{4pt}
	\begin{tabular}{|c|c|c|c|c|c|}
		\hline
		$k$
		& Optimization Dimension
		& Success Rate
		& Mean RelErr
		& Mean $\sqrt{G}$
		& Mean Iterations\\
		\hline
		$2$
		& $3$
		& $99\%$
		& $4.57\times10^{-3}$
		& $4.02\times10^{-11}$
		& $386.95$\\
		\hline
		$3$
		& $6$
		& $100\%$
		& $9.23\times10^{-3}$
		& $1.73\times10^{-10}$
		& $400.07$\\
		\hline
		$8$
		& $21$
		& $100\%$
		& $1.78\times10^{-2}$
		& $2.49\times10^{-10}$
		& $548.11$\\
		\hline
		$15$
		& $42$
		& $100\%$
		& $1.58\times10^{-2}$
		& $3.13\times10^{-10}$
		& $942.69$\\
		\hline
		$56$
		& $165$
		& $100\%$
		& $1.46\times10^{-2}$
		& $5.49\times10^{-10}$
		& $1690.04$\\
		\hline
		$470$
		& $1407$
		& $100\%$
		& $2.00\times10^{-2}$
		& $1.73\times10^{-9}$
		& $1916.65$\\
		\hline
	\end{tabular}
\end{table}

\begin{figure}
	\centering
	\subfloat[k=2]{\includegraphics[width=0.32\textwidth]{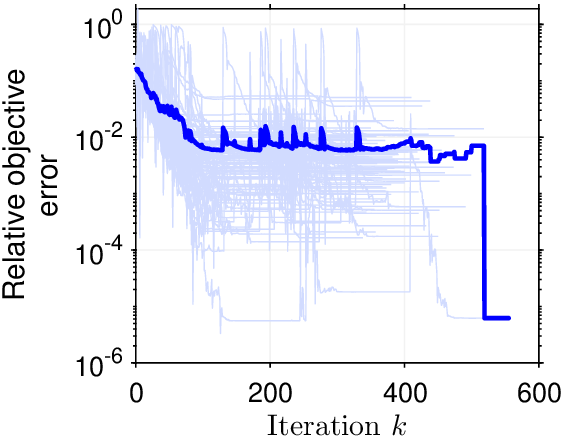}} 
	\subfloat[k=3]{\includegraphics[width=0.32\textwidth]{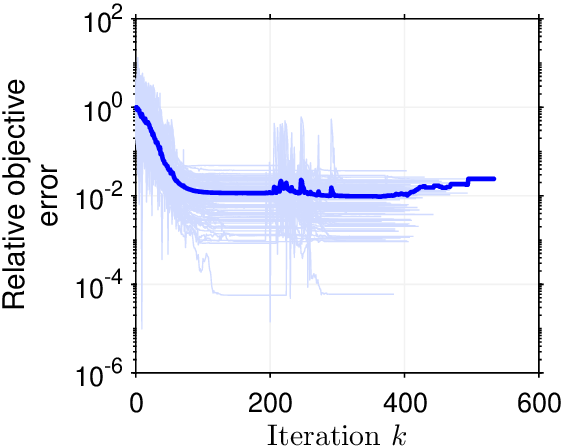}} 
	\subfloat[k=8]{\includegraphics[width=0.32\textwidth]{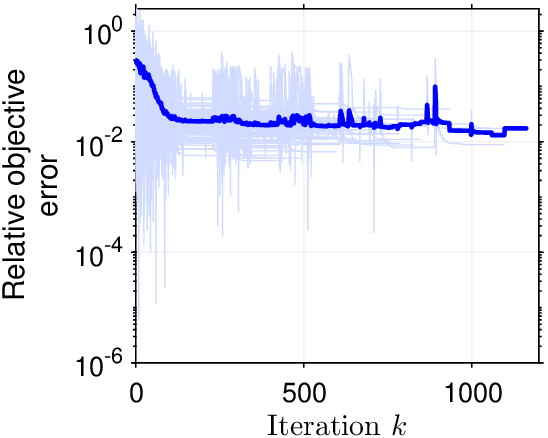}} \\
	\subfloat[k=15]{\includegraphics[width=0.32\textwidth]{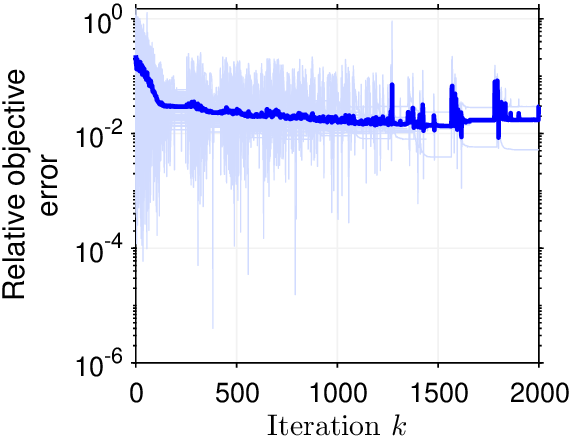}}
	\subfloat[k=56]{\includegraphics[width=0.32\textwidth]{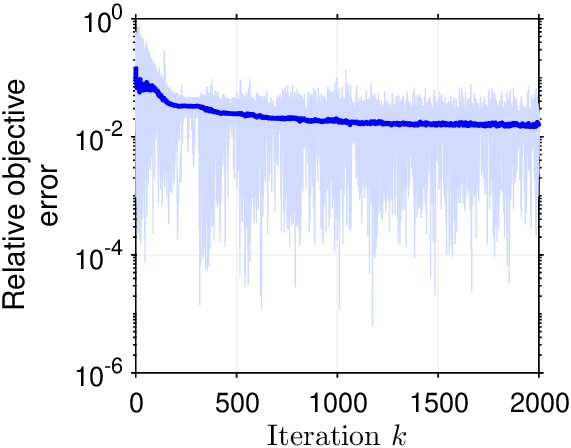}}
	\subfloat[k=470]{\includegraphics[width=0.32\textwidth]{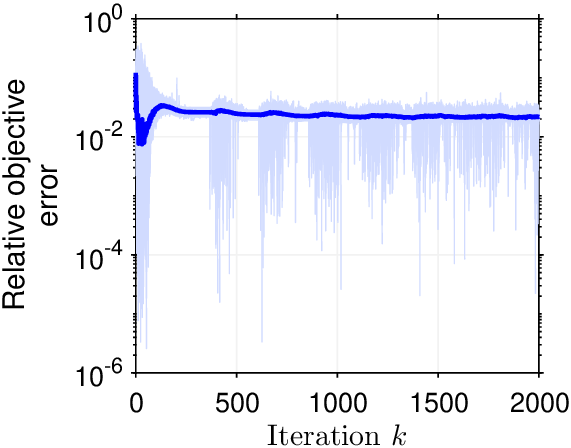}}
	\caption{\HZnew{
Thomson problem: evolution of the relative objective error \eqref{def of rel er} over $100$ simulations for $k=2,3,8,15,56,$ and $470$. The light curves correspond to individual simulations, while the dark curve denotes the mean over the simulations that remain active at each iteration.
}}
	\label{fig: eq3}
\end{figure}
\subsection{\HZnew{Sensitivity, Ablation, and Runtime Analysis}}
\label{subsec:diagnostics}

\HZnew{We conclude the numerical section reporting the sensitivity of the proposed method to the parameters $\alpha$ and $\epsilon$, the role of the independent noise introduced in Algorithm~\ref{algo2}, and the computational cost of the proposed
linearly implicit scheme compared to other methods.}

\subsubsection{\HZnew{Sensitivity to $\alpha$ and $\epsilon$}}
\label{subsec:sensitivity}

\HZnew{The convergence result in Theorem~\ref{main} requires specific
choices of the parameters $\alpha$ and $\epsilon$. More precisely, for a
fixed target accuracy, $\alpha$ is first chosen sufficiently large, and
$\epsilon$ is then chosen sufficiently small, depending on this choice of
$\alpha$, see \eqref{alpha} and
\eqref{epsilon}. This raises the practical question of how sensitive the
finite-particle algorithm is to the joint choice of $\alpha$ and $\epsilon$.}

\HZnew{To examine this dependence, we consider the $20$-dimensional Ackley
function with the spherical constraint in Case~3 described in Section~\ref{ackley} and vary $
\alpha\in\{10,20,30,50,80\},$ and $
\epsilon\in\{0.005,0.01,0.02,0.05,0.1\}.$ All other parameters are kept fixed as in Section~\ref{ackley}, and we perform
$100$ simulations for each pair $(\alpha,\epsilon)$.}
\begin{table}[h]
    \centering
    \caption{\HZnew{Success rates for different choices of $\alpha$ and
    $\epsilon$ in the $20$-dimensional Ackley problem with the spherical
    constraint. Each entry is computed over $100$ simulations.}}
    \label{table:sensitivity}
    \small
    \setlength{\tabcolsep}{5pt}
    \begin{tabular}{|c|ccccc|}
        \hline
        $\alpha\backslash\epsilon$
        & $0.005$ & $0.01$ & $0.02$ & $0.05$ & $0.1$\\
        \hline
        $10$ & $70\%$ & $80\%$ & $75\%$ & $77\%$ & $79\%$\\
        $20$ & $99\%$ & $99\%$ & $100\%$ & $100\%$ & $99\%$\\
        $30$ & $100\%$ & $99\%$ & $99\%$ & $100\%$ & $100\%$\\
        $50$ & $100\%$ & $100\%$ & $100\%$ & $100\%$ & $100\%$\\
        $80$ & $100\%$ & $100\%$ & $100\%$ & $100\%$ & $100\%$\\
        \hline
    \end{tabular}
\end{table}

\HZnew{The results suggest that the sequential dependence appearing in
Theorem~\ref{main} is less restrictive in the numerical experiments over the parameter range considered here. Once $\alpha$ is sufficiently large, the success rate is nearly insensitive to $\epsilon$ over the tested range.}

\subsubsection{\HZnew{Ablation Study of the Independent Noise}}
\label{subsec:indep_ablation}

\HZnew{We next examine the independent noise introduced in
Algorithm~\ref{algo2}. We consider the
$k=56$ Thomson problem and compare $\sigma_{\mathrm{indep}}=0.3$ with $\sigma_{\mathrm{indep}}=0$. The two experiments use the same $100$ initial conditions and the same random seeds, and all remaining parameters are identical as in Section~\ref{thompson}. The results are shown in Table~\ref{table:indep_ablation}.}

\begin{table}[h]
    \centering
    \caption{\HZnew{Ablation of the independent perturbation for the
    $k=56$ Thomson problem. The results are computed over $100$ paired
    simulations.}}
    \label{table:indep_ablation}
    \small
    \begin{tabular}{|c|cc|}
        \hline
        & $\sigma_{\mathrm{indep}}=0.3$
        & $\sigma_{\mathrm{indep}}=0$\\
        \hline
        Success rate
        & $100\%$ & $100\%$\\
        Mean relative error
        & $1.46\times10^{-2}$ & $3.31\times10^{-2}$\\
        Mean $\sqrt{G}=\|g\|_2$
        & $5.49\times10^{-10}$ & $5.71\times10^{-10}$\\
        Mean $\|g\|_1$
        & $1.00\times10^{-9}$ & $8.30\times10^{-10}$\\
        Mean iterations
        & $1690.04$ & $359.95$\\
        Mean runtime (s)
        & $20.11$ & $5.15$\\
        \hline
    \end{tabular}
\end{table}

\HZnew{Both variants satisfy the prescribed success criterion in all
$100$ simulations. However, the independent noise reduces the mean
relative objective error by approximately $55.7\%$, from
$3.31\times10^{-2}$ to $1.46\times10^{-2}$. Moreover, the relative
objective error is smaller with the independent perturbation in all
$100$ paired simulations. The mean number of iterations increases from approximately $360$ to $1690$.}

\subsubsection{\HZnew{Runtime Analysis}}
\label{subsec:runtime}

\HZnew{Finally, we compare the computational cost of the proposed method
with Penalized CBO and CB2O. Since the three methods may terminate at
different iterations, we use the runtime per iteration,
\[
T_{\mathrm{iter}}
=
\frac{\text{total runtime}}{\text{number of iterations}}.
\]
All experiments were performed in MATLAB R2025b on a
64-bit Windows machine with an Intel i5-11400 2.6GHz processor
and 16 GB of RAM.}

\HZnew{Across the examples in Sections~\ref{sec:eq1}--~\ref{ackley}, the median
paired per-run overhead of the proposed method is approximately
$3.7$--$12.8$ relative to Penalized CBO and $3.6$--$15.3$ relative to
CB2O. The additional cost is expected because the proposed scheme
evaluates the derivatives of $G$ and solves the linearly implicit system
at every iteration, while the comparison methods are derivative-free. The complete timing results are reported in
Appendix~\ref{appendix:runtime}.}

\section{Data Availability Statement}
The MATLAB code for reproducing all results is publicly available in this \href{https://github.com/yuhuazhumath/Constrained-CBO}{Github repository}.

\section{Conclusions}\label{sec:conclusion}

In this paper, we propose a  new CBO-based method for solving constrained non-convex minimization problem with equality constraints and potentially non-differentiable loss functions. Specifically, we augment the original CBO framework with a new forcing  term designed to guide particles toward the constraint set. \HZnew{On the theoretical side, we establish the well-posedness and mean-field limit for the anisotropic particle dynamics \eqref{constraincon}, deriving the corresponding anisotropic mean-field model \eqref{SDE}. We then study the long-time
behavior of the isotropic mean-field dynamics \eqref{SDE_iso} through the associated
Fokker--Planck equation \eqref{FPK_iso}. Our results establish that, under Assumption~\ref{wellbehave} and with a proper choice of parameters, the isotropic
mean-field distribution can reach arbitrary prescribed accuracy around the constrained minimizer $v^*$.} \HZnew{Notably, Assumption~\ref{wellbehave} (C) fits well with the
basic nature of the algorithm, while Assumption~\ref{wellbehave} (B)
provides sufficient structural conditions for the convergence analysis.} \HZnew{On the practical side, we proposed a linearly implicit algorithm
based on the continuous-in-time model and analyzed sufficient conditions
for its well-posedness, its stability properties, and its consistency with
the continuous particle dynamics. In particular, for affine equality
constraints, the implicit treatment of the constraint force does not impose
a stability restriction relating $\gamma$ and $\epsilon$. Numerical comparisons with Penalized CBO and CB2O show favorable
success rates and terminal errors for the proposed method in the more
challenging examples considered here. Additional sensitivity and ablation
studies indicate that the practical dependence on $\alpha$ and $\epsilon$
is less restrictive than the sufficient conditions in the convergence
analysis and that the independent noise improves exploration and
objective accuracy at an additional computational cost. The runtime
comparison confirms the additional per-iteration cost associated with the
derivative information and linearly implicit solve.}

\newpage
\appendix
\begin{center}
{\bf Appendix}
\end{center}

\section{Some details in the Proofs of Well-posedness and Mean-field limit}
\subsection{Proof of Theorem \ref{wellposedmicro}}\label{appenmicro}
Consider the microscopic model, which is governed  by the following equation:
\begin{align*}
	dV^{i,N}_{t}=&-\lambda \left(V^{i,N}_{t}-v_{\alpha}(\hat{\rho}_{t}^N)\right)\,dt-\dfrac{1}{\epsilon} \nabla G\left(V^{i,N}_{ t}\right)\,d t\\
    &+\sigma\text{diag}\left(V^{i,N}_{t} - v_{\alpha}(\hat{\rho}_{t}^{N})\right)\,dB^{i,N}_{t},
\end{align*}
with initial distribution $V^{i,N}_{0}\sim \rho_{0}$,
where $ i=1,\dots,N $. We can \HZnew{concatenate} $ \left\{V_{t}^{i,N}\right\}_{i=1}^{N} $ into one vector and put them in one equation. To be specific, we define
\begin{gather*}
	V_{t}=\left(\left(V_{t}^{1,N}\right)^{T},...,\left(V_{t}^{N,N}\right)^{T}\right)^{T}.
\end{gather*}
Then $ V_{t} $ is a vector in $ \mathbb{R}^{Nd} $ for each fixed $ t $ and it will satisfy the following equation:
\begin{gather}\label{oneeq}
	dV_{t}=-\lambda F_{N}(V_{t})\,dt-\dfrac{1}{\epsilon}L_N(V_{t})\,dt+\sigma M_{N}(V_{t})\,dB_{t}^{(N)}.
\end{gather}
Here $$
B_t^{(N)}
=
\left(
(B_t^{1,N})^T,\ldots,(B_t^{N,N})^T
\right)^T\in\R^{Nd},
$$ 
\begin{gather*}
	L_N(V_t)=\left(\left(\nabla G\left(V_t^{1,N}\right)\right)^T,...,\left(\nabla G\left(V_t^{N,N}\right)\right)^T\right)^T\in\mathbb{R}^{Nd},
\end{gather*}
\begin{gather*}
	M_{N}(V_t)=\text{diag}\left(F_{N}^{1}\left(V_t\right),...,F_{N}^{N}\left(V_t\right)\right)\in\mathbb{R}^{Nd\times Nd}
\end{gather*}
and
\begin{gather*}
	F_{N}(V_t)=\left(\left(F_{N}^{1}\left(V_t\right)\right)^{T},...,\left(F_{N}^{N}\left(V_t\right)\right)^{T}\right)^{T}\in\mathbb{R}^{Nd},
\end{gather*}
where
\begin{gather*}
	F_{N}^{i}(V_t)=\dfrac{\sum_{j\neq i}\left(V_t^{i,N}-V_t^{j,N}\right)\omega_{\alpha}\left(V_t^{j,N}\right)}{\sum_{j}\omega_{\alpha}(V_t^{j,N})}\in\mathbb{R}^{d}.
\end{gather*}
Thus it suffices to prove the well-posedness result of equation (\ref{oneeq}). The following theorem gives the well-posedness.

\begin{theorem}
	For each $N\in\mathbb{N}$, the stochastic differential equation
	(\ref{oneeq}) has a unique strong solution $\{V_t|t\geq0\}$
	for any $\mathcal{F}_0$-measurable initial condition $V_0$,
	independent of $B^{(N)}$, satisfying
	$\mathbb{E}[\|V_0\|_2^2]<\infty$.
\end{theorem}

\begin{proof}
	\HZnew{By Lemma~2.1 in \cite{CCTT}, $F_N$ is locally Lipschitz and has
linear growth under Assumption~\ref{assump1} (1)--(2). Since $M_N$
is obtained by placing the components of $F_N$ on the diagonal, the
same properties hold for $M_N$. Moreover, $L_N$ is globally Lipschitz and has at most linear growth by Assumption~\ref{assump1} (3). Further, one can obtain}
	\begin{gather*}
		-2\lambda V\cdot F_{N}(V)\leq 2\lambda\sqrt{N}\big\|V\big\|^2,\text{ and }
		\text{trace}(M_{N}M_{N}^{T})(V)=\big\|F_N(V)\big\|^2\leq 4N\big\|V\big\|^2.
	\end{gather*}
	Thus $
		-2\lambda V\cdot F_{N}(V)+\sigma^{2}\text{trace}(M_{N}M_{N}^{T})(V)\leq b_{N}\big\|V\big\|^{2},$ where $ b_{N} $ is a positive number that depends only on $ \lambda,\sigma,d $ and $ N $. Also, by the linear growth of $\nabla G$ following from
Assumption~\ref{assump1} (3), for $X\in\mathbb{R}^{d}$, $
    -\frac{2}{\epsilon}X\cdot\nabla G(X)
    \leq
    \frac{2}{\epsilon}\|X\|_2\|\nabla G(X)\|_2
    \leq
    \frac{C}{\epsilon}\left(1+\|X\|_2^2\right).$
Thus, $
    -\frac{2}{\epsilon}V\cdot L_N(V)
    \leq
    \frac{C}{\epsilon}
    \left(N+\|V\|_2^2\right).$ This implies that
	\begin{gather*}
    2V\cdot\left(
    -\lambda F_N(V)-\dfrac{1}{\epsilon}L_N(V)
    \right)
    +\sigma^2\operatorname{trace}(M_NM_N^T)(V)
    \leq
    \tilde b_N\left(1+\|V\|_2^2\right),
\end{gather*}
	where $\tilde b_N>0$ depends only on
$\lambda,\sigma,d,\epsilon,N,L$ and $\|\nabla G(0)\|_2$. Then we apply~\cite[Chapter~5, Theorems~3.1--3.2]{D} to finish the proof. Note that although \cite[Chapter~5, Theorems~3.1--3.2]{D} are stated for deterministic initial data, the same conclusion extends to an $\mathcal{F}_0$-measurable square-integrable random initial condition independent of the driving Wiener process by the standard localization argument, using globally Lipschitz truncations of the coefficients.
\end{proof}

\subsection{Proof of Theorem \ref{meanfieldwellposed}}\label{appenmean}
The following lemma is a slight extension of Lemma 3.2 in \cite{CCTT}.

\begin{lemma}\label{stability}
Let $\CE$ satisfy Assumption~\ref{assump1} (1)--(2) and
$\mu,\hat{\mu}\in\CP_4(\mathbb{R}^d)$ with $
    \int \|v\|_2^4\,d\mu(v),$ and $
    \int \|\hat v\|_2^4\,d\hat{\mu}(\hat v)
    \leq K.$ Then $
    \|v_\alpha(\mu)-v_\alpha(\hat{\mu})\|_2
    \leq c_0 W_2(\mu,\hat{\mu}),$ where $c_0>0$ depends only on
$\alpha,L,K,\underline{\CE}$ and $\overline{\CE}$.
\end{lemma}
\begin{proof}
Let $\pi\in\Pi(\mu,\hat{\mu})$ be an optimal coupling. Since $\CE$ is
bounded, there exist constants $c_\omega,C_\omega>0$ such that $
c_\omega\leq\omega_\alpha(v)\leq C_\omega
$ for $ v\in\mathbb{R}^d.$ Moreover, by Assumption~\ref{assump1} (2) and the mean-value theorem, $
|\omega_\alpha(v)-\omega_\alpha(\hat v)|
\leq
C(1+\|v\|_2+\|\hat v\|_2)\|v-\hat v\|_2.$ Hence, by Cauchy's inequality and the fourth-moment bound,
\begin{gather*}
\left|
\int\omega_\alpha\,d\mu
-\int\omega_\alpha\,d\hat{\mu}
\right|
\leq C_K W_2(\mu,\hat{\mu}),\\
\left\|
\int v\omega_\alpha(v)\,d\mu(v)
-\int \hat v\omega_\alpha(\hat v)\,d\hat{\mu}(\hat v)
\right\|_2
\leq C_K W_2(\mu,\hat{\mu}).
\end{gather*}
Since
$\int\omega_\alpha\,d\mu,\int\omega_\alpha\,d\hat{\mu}\geq c_\omega$
and
$\|\int \hat v\omega_\alpha(\hat v)\,d\hat{\mu}(\hat v)\|_2\leq C_K$,
the desired estimate follows from the definition of $v_\alpha$.
\end{proof}
\begin{proof}[Proof of Theorem~\ref{meanfieldwellposed}]
Throughout the proof, $C>0$ denotes a generic constant that may depend on $\epsilon$.

We first derive a moment bound for the consensus point. Let
$\mu\in\mathcal{P}_p(\mathbb{R}^d)$ with $p\geq1$, and denote the total weights by $
Z_\alpha(\mu)
:=
\int e^{-\alpha\mathcal E(v)}\,d\mu(v).$ Since $
e^{-\alpha\overline{\mathcal E}}
\leq
Z_\alpha(\mu)
\leq
e^{-\alpha\underline{\mathcal E}},$
the measure $
d\nu_\mu(v)
:=
\frac{e^{-\alpha\mathcal E(v)}}{Z_\alpha(\mu)}
\,d\mu(v)$ is a probability measure. By Jensen's inequality,
\begin{gather}\label{eq:va_moment}
\|v_\alpha(\mu)\|_2^p=
\left\|
\int v\,d\nu_\mu(v)
\right\|_2^p
\leq
\int\|v\|_2^p\,d\nu_\mu(v)\leq
e^{\alpha(
\overline{\mathcal E}-\underline{\mathcal E})}
\int\|v\|_2^p\,d\mu(v).
\end{gather}
Now we fix a filtered probability space. On this probability space, we further fix a $\mathcal F_0$-measurable random variable $V_0$ with law $\rho_0$ and
a $d$-dimensional Brownian motion $B$ independent of $V_0$. We
construct a sequence of processes by Picard iteration using induction. Set
\[
V_t^0:=V_0,
\qquad
\rho_t^0:=\rho_0,
\qquad
u_t^0:=v_\alpha(\rho_0).
\]
For $n\geq0$, assuming we have defined $V^n$, let
\[
\rho_t^n:=\operatorname{Law}(V_t^n),
\qquad
u_t^n:=v_\alpha(\rho_t^n),
\]
and define $V^{n+1}$ as the unique strong solution to
\begin{equation}\label{eq:picard_iteration}
\begin{aligned}
dV_t^{n+1}
=
-\lambda\left(V_t^{n+1}-u_t^n\right)\,dt
-\frac{1}{\epsilon}
\nabla G(V_t^{n+1})\,dt+
\sigma\operatorname{diag}
\left(V_t^{n+1}-u_t^n\right)\,dB_t,\quad  V_0^{n+1}
=V_0.
\end{aligned}
\end{equation}
This iteration is well defined. We can show this by induction. We know $u^0$ is continuous. If
$u^n$ is continuous, the drift and diffusion coefficients in
\eqref{eq:picard_iteration} are globally Lipschitz in the spatial
variable and have at most linear growth. Thus the equation has a
unique strong solution with continuous paths and finite fourth
moments. Thus, $
\rho^{n+1}
\in
\mathcal C\left(
[0,T];\mathcal P_4(\mathbb R^d)
\right).$ Lemma~\ref{stability} then implies that
$u^{n+1}\in\mathcal C([0,T];\mathbb R^d)$.

We next prove a fourth-moment bound for the Picard sequence that is uniform in $n$. Define $
M_n(t)
:=
\mathbb E\left[
\sup_{r\in[0,t]}
\|V_r^n\|_2^4
\right].$ Using the integral form of \eqref{eq:picard_iteration}, H\"older's inequality, the Burkholder--Davis--Gundy inequality, and
the linear growth of $\nabla G$, one obtains
\begin{equation}\label{eq:picard_moment_recursion}
M_{n+1}(t)
\leq
C\left(
1+\mathbb E\|V_0\|_2^4
\right)
+
C_T\int_0^t
\left(
1+M_{n+1}(s)+\|u_s^n\|_2^4
\right)\,ds.
\end{equation}
To be more specific, the diffusion term can be bounded as below.
\begin{align*}
\mathbb E\left[
\sup_{r\in[0,t]}
\left\|
\int_0^r
\operatorname{diag}
\left(V_s^{n+1}-u_s^n\right)
\,dB_s
\right\|_2^4
\right]&\leq
C\mathbb E\left[
\left(
\int_0^t
\left\|
\operatorname{diag}
\left(V_s^{n+1}-u_s^n\right)
\right\|_{\mathrm F}^2
\,ds
\right)^2
\right]\\
&=
C\mathbb E\left[
\left(
\int_0^t
\|V_s^{n+1}-u_s^n\|_2^2
\,ds
\right)^2
\right]\\
&\leq
C_T\int_0^t
\left(
M_{n+1}(s)+\|u_s^n\|_2^4
\right)\,ds.
\end{align*}
The drift terms are bounded in the same way using H\"older's inequality. By \eqref{eq:va_moment},
\begin{align}
\|u_s^n\|_2^4
&\leq
e^{\alpha(
\overline{\mathcal E}-\underline{\mathcal E})}
\mathbb E\left[\|V_s^n\|_2^4\right]
\leq
C_\alpha M_n(s).
\label{eq:picard_consensus_moment}
\end{align}
For $N\geq1$, let $
\mathcal M_N(t)
:=
1+\max_{0\leq n\leq N}M_n(t).$ Taking the maximum over
$n=0,\ldots,N-1$ in
\eqref{eq:picard_moment_recursion} and using
\eqref{eq:picard_consensus_moment}, we obtain $
\mathcal M_N(t)
\leq
C_0+C_T\int_0^t\mathcal M_N(s)\,ds,$ where $C_0$ is independent of $N$. Gronwall's inequality then
gives the uniform fourth moment bound,
\begin{gather}\label{eq:picard_uniform}
\sup_{n\geq0}
\mathbb E\left[
\sup_{t\in[0,T]}
\|V_t^n\|_2^4
\right]
\leq K_T
\end{gather}
for some $K_T>0$. By~\eqref{eq:picard_consensus_moment}, we also have $
\sup_{n\geq0}
\sup_{t\in[0,T]}
\|u_t^n\|_2^4
\leq C_\alpha K_T.$

We next prove that the sequence is Cauchy. For $n\geq1$, to simplify the notation, define
\[
Z_t^{n+1}
:=
V_t^{n+1}-V_t^n,
\qquad
\Delta u_t^n
:=
u_t^n-u_t^{n-1}.
\]
Subtracting two consecutive equations in
\eqref{eq:picard_iteration} gives
\begin{align*}
Z_t^{n+1}
={}&
-\lambda\int_0^t
\left(
Z_s^{n+1}-\Delta u_s^n
\right)\,ds-\frac{1}{\epsilon}
\int_0^t
\left(
\nabla G(V_s^{n+1})
-
\nabla G(V_s^n)
\right)\,ds\\
&+
\sigma\int_0^t
\operatorname{diag}
\left(
Z_s^{n+1}-\Delta u_s^n
\right)\,dB_s.
\end{align*}
Let $
D_{n+1}(t)
:=
\mathbb E\left[
\sup_{r\in[0,t]}
\|Z_r^{n+1}\|_2^2
\right].$ By the global Lipschitz continuity of $\nabla G$, Cauchy's inequality,
the Burkholder--Davis--Gundy inequality, we have $
D_{n+1}(t)
\leq
C\int_0^tD_{n+1}(s)\,ds
+
C\int_0^t
\|\Delta u_s^n\|_2^2\,ds.$ Gronwall's inequality then gives
\begin{gather}\label{eq:picard_difference}
D_{n+1}(t)
\leq
C_T\int_0^t
\|\Delta u_s^n\|_2^2\,ds.
\end{gather}

The uniform fourth-moment bound
\eqref{eq:picard_uniform} allows us to use
Lemma~\ref{stability} with a constant independent of $n$. Hence
\begin{align*}
\|\Delta u_s^n\|_2^2
&=
\left\|
v_\alpha(\rho_s^n)
-
v_\alpha(\rho_s^{n-1})
\right\|_2^2\leq
c_0^2W_2^2(\rho_s^n,\rho_s^{n-1})\leq
c_0^2
\mathbb E\left[
\|V_s^n-V_s^{n-1}\|_2^2
\right]\leq
c_0^2D_n(s).
\end{align*}
Plugging into \eqref{eq:picard_difference} gives $
D_{n+1}(t)
\leq
C_T\int_0^tD_n(s)\,ds.$ One can then use induction to prove,
\begin{gather}\label{eq:picard_factorial}
D_{n+1}(T)
\leq
D_1(T)
\frac{(C_TT)^n}{n!},
\qquad n\geq1.
\end{gather}
Since $
\sum_{n=1}^{\infty}
\sqrt{\frac{(C_TT)^n}{n!}}
<\infty,$ the sequence $\{V^n\}_{n\geq0}$ is Cauchy in $
L^2\left(
\Omega;
\mathcal C([0,T];\mathbb R^d)
\right).$ Thus, there exists an adapted continuous process $V$ such that
\begin{gather}\label{eq:picard_process_convergence}
\mathbb E\left[
\sup_{t\in[0,T]}
\|V_t^n-V_t\|_2^2
\right]
\rightarrow0.
\end{gather}

By choosing a subsequence that converges almost surely uniformly in
time and applying Fatou's lemma to
\eqref{eq:picard_uniform}, we obtain $
\mathbb E\left[
\sup_{t\in[0,T]}
\|V_t\|_2^4
\right]
\leq K_T.$ Let
\[
\rho_t:=\operatorname{Law}(V_t),
\qquad
u_t:=v_\alpha(\rho_t).
\]
Using Lemma~\ref{stability} and
\eqref{eq:picard_process_convergence},
\begin{align*}
\sup_{t\in[0,T]}
\|u_t^n-u_t\|_2\leq
c_0
\sup_{t\in[0,T]}
W_2(\rho_t^n,\rho_t)\leq
c_0
\left(
\mathbb E\left[
\sup_{t\in[0,T]}
\|V_t^n-V_t\|_2^2
\right]
\right)^{1/2}
\rightarrow0.
\end{align*}

We can now let $n\rightarrow \infty$ in
\eqref{eq:picard_iteration}. The drift terms converge because of
\eqref{eq:picard_process_convergence}, the uniform convergence
$u^n\rightarrow u$, and the Lipschitz continuity of $\nabla G$. For
the diffusion term, the Burkholder--Davis--Gundy inequality gives
\begin{align*}
&\mathbb E\left[
\sup_{t\in[0,T]}
\left\|
\int_0^t
\left[
\operatorname{diag}
\left(V_s^{n+1}-u_s^n\right)
-
\operatorname{diag}
\left(V_s-u_s\right)
\right]
\,dB_s
\right\|_2^2
\right]\\
&\qquad\leq
C\int_0^T
\mathbb E\left[
\left\|
(V_s^{n+1}-V_s)
-
(u_s^n-u_s)
\right\|_2^2
\right]\,ds
\rightarrow0.
\end{align*}
Thus $V$ satisfies
\begin{align}\label{eq:meanfield_integral}
V_t
={}&
V_0
-\lambda\int_0^t
\left(
V_s-v_\alpha(\rho_s)
\right)\,ds
-\frac{1}{\epsilon}
\int_0^t\nabla G(V_s)\,ds+
\sigma\int_0^t
\operatorname{diag}
\left(
V_s-v_\alpha(\rho_s)
\right)\,dB_s.
\end{align}
Hence $V$ is a strong solution to \eqref{SDE}. Because $V$ has continuous paths and $
\mathbb E\left[
\sup_{t\in[0,T]}\|V_t\|_2^4
\right]<\infty,$ the dominated convergence theorem gives $
\mathbb E\left[
\|V_t-V_s\|_2^4
\right]\rightarrow0$ as $t\rightarrow s.$
As a result, $
W_4^4(\rho_t,\rho_s)
\leq
\mathbb E\left[
\|V_t-V_s\|_2^4
\right]
\rightarrow0,$ and thus $
\rho
\in
\mathcal C\left(
[0,T];\mathcal P_4(\mathbb R^d)
\right).$

We next prove pathwise uniqueness. Suppose that $V$ and $\widehat V$
are two solutions of \eqref{SDE} on the same filtered probability
space, driven by the same Brownian motion and satisfying
$V_0=\widehat V_0$ almost surely. Let $
\rho_t=\operatorname{Law}(V_t),$ $\widehat\rho_t=\operatorname{Law}(\widehat V_t).$
By the same standard fourth-moment estimate used above,
every strong solution to \eqref{SDE} with initial law
$\rho_0\in\mathcal{P}_4(\mathbb{R}^d)$ satisfies $
\mathbb{E}\left[
\sup_{t\in[0,T]}\|V_t\|_2^4
\right]<\infty.$
Since the solutions have continuous paths, dominated convergence further gives $
\rho,\widehat\rho
\in
\mathcal C\left(
[0,T];\mathcal P_4(\mathbb R^d)
\right).$
Thus, there exists $K>0$ such that
\[
\sup_{t\in[0,T]}
\int\|v\|_2^4\,d\rho_t(v)
+
\sup_{t\in[0,T]}
\int\|v\|_2^4\,d\widehat\rho_t(v)
\leq K.
\]
Therefore, Lemma~\ref{stability} can be applied uniformly on
$[0,T]$. Define $
D(t)
:=
\mathbb E\left[
\sup_{r\in[0,t]}
\|V_r-\widehat V_r\|_2^2
\right].$ We can use the same technique used for
\eqref{eq:picard_difference} to derive
\begin{align*}
D(t)
\leq
C\int_0^t
\left(
D(s)
+
\left\|
v_\alpha(\rho_s)
-
v_\alpha(\widehat\rho_s)
\right\|_2^2
\right)\,ds.
\end{align*}
By Lemma~\ref{stability},
\begin{align*}
\left\|
v_\alpha(\rho_s)
-
v_\alpha(\widehat\rho_s)
\right\|_2^2\leq
c_0^2W_2^2(\rho_s,\widehat\rho_s)\leq
c_0^2
\mathbb E\left[
\|V_s-\widehat V_s\|_2^2
\right]
\leq
c_0^2D(s).
\end{align*}
Thus, $
D(t)
\leq
C\int_0^tD(s)\,ds.$ Since $D(0)=0$, Gronwall's inequality implies $D(t)=0$ for every
$t\in[0,T]$. Thus pathwise uniqueness holds.

Finally, let $\phi\in\mathcal C_c^2(\mathbb R^d)$. Applying It\^o's
formula to $\phi(V_t)$ in \eqref{eq:meanfield_integral}, taking
expectations, we obtain
\begin{align*}
\langle\phi,\rho_t\rangle-\langle\phi,\rho_0\rangle
={}&
-\lambda\int_0^t
\left\langle
\left(v-v_\alpha(\rho_s)\right)\cdot\nabla\phi,
\rho_s
\right\rangle\,ds-\frac{1}{\epsilon}\int_0^t
\left\langle
\nabla G(v)\cdot\nabla\phi,
\rho_s
\right\rangle\,ds\\
&+\frac{\sigma^2}{2}
\sum_{k=1}^d
\int_0^t
\left\langle
\left(v-v_\alpha(\rho_s)\right)_k^2
\partial_{kk}\phi,
\rho_s
\right\rangle\,ds.
\end{align*}
Hence $\rho$ is a weak solution to \eqref{FPK}. Moreover,
$\rho_t\rightarrow\rho_0$ in $\mathcal P_4(\mathbb R^d)$ as
$t\rightarrow0$. This completes the proof.
\end{proof}
\subsection{Proof of Theorem \ref{meanfieldthm}}\label{appenmeanfield}
We first prove the following lemma.
\begin{lemma}\label{momentbound}
	Let $ \CE $ satisfy Assumption \ref{assump1} and $ \rho_{0}\in \CP_{4}(\mathbb{R}^{d}) $. For any $ N\geq 2 $, assume that $ \{(V_{t}^{i,N})_{t\in[0,T]}\}_{i=1}^{N} $ is the unique solution to the particle system (\ref{constraincon}) with $ \rho_{0}^{\otimes N} $ distributed initial data $ \{V_{0}^{i,N}\}_{i=1}^{N} $. Then there exists a constant $ K>0 $ independent of $ N $ such that
	\begin{gather*}
		\sup_{i=1,...,N}\left\{\sup_{t\in[0,T]}\mathbb{E}\left[\big\|V_{t}^{i,N}\big\|^{2}+\big\|V_{t}^{i,N}\big\|^{4}\right]+\sup_{t\in[0,T]}\mathbb{E}\left[\big\|v_{\alpha}(\hat{\rho}_{t}^{N})\big\|^{2}+\big\|v_{\alpha}(\hat{\rho}^{N}_{t})\big\|^{4}\right]\right\}\leq K.
	\end{gather*}
	{Moreover,
	\begin{gather}\label{eq:path_moment}
		\sup_{i=1,...,N}
		\E\left[
		\sup_{t\in[0,T]}
		\big\|V_t^{i,N}\big\|_2^4
		\right]
		\leq K.
	\end{gather}}
\end{lemma}
\begin{proof}
	For each $ i $, we have
	\begin{gather*}
dV_t^{i,N}
=
-\lambda\left(
V_t^{i,N}-v_\alpha(\hat{\rho}_t^N)
\right)\,dt
-\frac{1}{\epsilon}\nabla G(V_t^{i,N})\,dt
+\sigma\operatorname{diag}\left(
V_t^{i,N}-v_\alpha(\hat{\rho}_t^N)
\right)\,dB_t^{i,N},\\
V_0^{i,N}\sim\rho_0.
\end{gather*}
	Now we pick $ p=1 $ or $ p=2 $. Then
	\begin{align*}
    \E \big\|\viti\big\|^{2p}
    \lesssim{}&
    1+\E \big\|V_{0}^{i,N}\big\|^{2p}
    +\E\left(\int_{0}^{t}\big\|\vit\big\|\,d\tau\right)^{2p}
    +\E\left(\int_{0}^{t}\big\|\emconst\big\|\,d\tau\right)^{2p}\\
    &+\E\big\|\int_{0}^{t}\text{diag}\left(\vit\right)\,dB_{\tau}^{i}\big\|^{2p}
    +\E\big\|\int_{0}^{t}\text{diag}\left(\emconst\right)\,dB_{\tau}^{i}\big\|^{2p}.
\end{align*}
	Here, we used the linear growth of $\nabla G$ following from
Assumption~\ref{assump1} (3). Now by Cauchy's inequality, $\E(\int_{0}^{t}\big\|\vit\big\|\,d\tau)^{2p}\leq t^{p}\cdot\E(\int_{0}^{t}\big\|\vit\big\|^{2}\,d\tau)^{p} $
	and $\E(\int_{0}^{t}\big\|\emconst\big\|\,d\tau)^{2p}\leq t^{p}\cdot\E(\int_{0}^{t}\big\|\emconst\big\|^{2}\,d\tau)^{p}.$
    Also, by the Burkholder--Davis--Gundy inequality, for $p=1,2$, we have $
\E\left\|
\int_0^t \operatorname{diag}\left(V_\tau^{i,N}\right)
\,dB_\tau^i
\right\|_2^{2p}
\lesssim
\E\left(
\int_0^t \|V_\tau^{i,N}\|_2^2\,d\tau
\right)^p,$ and $
\E\left\|
\int_0^t
\operatorname{diag}\left(v_\alpha(\hat{\rho}_\tau^N)\right)
\,dB_\tau^i
\right\|_2^{2p}
\lesssim
\E\left(
\int_0^t
\|v_\alpha(\hat{\rho}_\tau^N)\|_2^2\,d\tau
\right)^p.$ Thus 
	\begin{gather*}
		\E \big\|\viti\big\|^{2p}
\lesssim
1+\E \big\|V_{0}^{i,N}\big\|^{2p}
+\E\left(\int_{0}^{t}\big\|\vit\big\|^{2}\,d\tau\right)^{p}
+\E\left(\int_{0}^{t}\big\|\emconst\big\|^{2}\,d\tau\right)^{p},
	\end{gather*}
	Further, by H\"older's inequality, we have $
\E\left(\int_{0}^{t}\big\|V_\tau^{i,N}\big\|_2^{2}\,d\tau\right)^p
\leq
t^{p-1}\E\int_{0}^{t}\big\|V_\tau^{i,N}\big\|_2^{2p}\,d\tau$ and $
\E\left(\int_{0}^{t}\big\|v_\alpha(\hat\rho_\tau^N)\big\|_2^{2}\,d\tau\right)^p
\leq
t^{p-1}\E\int_{0}^{t}\big\|v_\alpha(\hat\rho_\tau^N)\big\|_2^{2p}\,d\tau.$ So we can deduce
	\begin{gather*}
		\E \big\|\viti\big\|^{2p}
\lesssim
1+\E \big\|V_{0}^{i,N}\big\|^{2p}
+\E\int_{0}^{t}\big\|\vit\big\|^{2p}\,d\tau
+\E\int_{0}^{t}\big\|\emconst\big\|^{2p}\,d\tau.
	\end{gather*}
	Thus
	\begin{gather}\label{2pmomentbound}
    \E\int\|v\|^{2p}\,d\hat{\rho}_{t}^{N}
    \lesssim
    1+
    \E\int\|v\|^{2p}\,d\hat{\rho}^{N}_{0}
    +\int_{0}^{t}
    \left(
    \E\int\|v\|^{2p}\,d\dempt
    \right)\,d\tau
    +\int_{0}^{t}
    \left(
    \E\|\emconst\|^{2p}
    \right)\,d\tau.
\end{gather}
	By Assumption~\ref{assump1} (1), one has
\begin{gather}\label{lem3.3}
\int \|v\|_2^2
\frac{\omega_\alpha(v)}
{\|\omega_\alpha\|_{L^1(\hat{\rho}_\tau^N)}}
\,d\hat{\rho}_\tau^N(v)
\leq
e^{\alpha(\overline{\CE}-\underline{\CE})}
\int \|v\|_2^2\,d\hat{\rho}_\tau^N(v).
\end{gather}
	Then we can calculate
	\begin{align*}
\|v_\alpha(\hat{\rho}_\tau^N)\|_2^{2p}
&\leq
\left(
\int \|v\|_2^2
\frac{\omega_\alpha(v)}
{\|\omega_\alpha\|_{L^1(\hat{\rho}_\tau^N)}}
\,d\hat{\rho}_\tau^N(v)
\right)^p\\
&\lesssim
\left(
\int \|v\|_2^2\,d\hat{\rho}_\tau^N(v)
\right)^p
\lesssim
\int \|v\|_2^{2p}\,d\hat{\rho}_\tau^N(v).
\end{align*}
 Combine the above inequality and (\ref{2pmomentbound}) leads to 
	\begin{gather*}
		\E\int\big\|v\big\|^{2p}\,d\hat{\rho}_{t}^{N}\lesssim \E\int\big\|v\big\|^{2p}\,d\hat{\rho}^{N}_{0}+\int_{0}^{t}\left(\E\int\big\|v\big\|^{2p}\,d\dempt\right)\,d\tau+1.
	\end{gather*}
	By applying Gronwall's inequality, it follows that $ \E\int|v|^{2p}\,d\hat{\rho}_{t}^{N} $ is bounded for $ t\in[0,T] $, and the bound does not depend on $ N $. \HZnew{Since the particle system is exchangeable,
$\E\|V_t^{i,N}\|_2^{2p}
=
\E\int\|v\|_2^{2p}\,d\hat{\rho}_t^N(v)$ for every
$i=1,\ldots,N$, and thus the above bound is also uniform in $i$.} Also, we know that $
		\big\|\emconst\big\|^{2p}\lesssim 1+\int\big\|v\big\|^{2p}\,d\hat{\rho}^{N}_{\tau},$
	which implies that $
		\E	\big\|v_{\alpha}(\hat{\rho}^{N}_{\tau})\big\|^{2p}\lesssim 1+\E\int\big\|v\big\|^{2p}\,d\hat{\rho}^{N}_{t}.$ So $ \E	\big\|v_{\alpha}(\hat{\rho}^{N}_{t})\big\|^{2p} $ is bounded for $ t\in [0,T]$ and the bound does not depend on $ N $.

{
Finally, we prove \eqref{eq:path_moment}. For the diffusion term, the
Burkholder--Davis--Gundy inequality and H\"older's inequality give
\begin{align*}
\E\left[
\sup_{t\in[0,T]}
\left\|
\int_0^t
\operatorname{diag}\left(
V_s^{i,N}
-v_\alpha(\hat\rho_s^N)
\right)
\,dB_s^{i,N}
\right\|_2^4
\right]&\lesssim
\E\left[
\left(
\int_0^T
\left\|
V_s^{i,N}
-v_\alpha(\hat\rho_s^N)
\right\|_2^2
\,ds
\right)^2
\right]\\
&\lesssim
T\int_0^T
\E\left[
\|V_s^{i,N}\|_2^4
+
\|v_\alpha(\hat\rho_s^N)\|_2^4
\right]\,ds.
\end{align*}
For the drift terms, H\"older's inequality gives
\begin{align*}
&\E\left[
\sup_{t\in[0,T]}
\left\|
\int_0^t
\left(
V_s^{i,N}
-v_\alpha(\hat\rho_s^N)
\right)\,ds
\right\|_2^4
\right]
&\lesssim
T^3\int_0^T
\E\left[
\|V_s^{i,N}\|_2^4
+
\|v_\alpha(\hat\rho_s^N)\|_2^4
\right]\,ds,
\end{align*}
and using the linear growth of $\nabla G$ following from
Assumption~\ref{assump1} (3),
\begin{align*}
\E\left[
\sup_{t\in[0,T]}
\left\|
\int_0^t
\nabla G(V_s^{i,N})\,ds
\right\|_2^4
\right]
&\lesssim
T^3\int_0^T
\E\left[
\|\nabla G(V_s^{i,N})\|_2^4
\right]\,ds\\
&\lesssim
C_T\int_0^T
\left(
1+\E\left[
\|V_s^{i,N}\|_2^4
\right]
\right)\,ds.
\end{align*}
Combining the above estimates with the integral form of
\eqref{constraincon}, we obtain
\begin{align*}
\E\left[
\sup_{t\in[0,T]}
\|V_t^{i,N}\|_2^4
\right]
\lesssim{}&
1+\E\left[\|V_0^{i,N}\|_2^4\right]\\
&+
C_T\int_0^T
\E\left[
\|V_s^{i,N}\|_2^4
+
\|v_\alpha(\hat\rho_s^N)\|_2^4
\right]\,ds.
\end{align*}
The moment estimates proved above and
$\rho_0\in\mathcal P_4(\mathbb R^d)$ show that the right-hand side is
bounded uniformly in $N$ and $i$. This proves \eqref{eq:path_moment}.
}
\end{proof}
With a slight abuse of notation, we use $
\hat{\rho}^{N}
:=
\frac{1}{N}\sum_{i=1}^{N}\delta_{V^{i,N}}
\in
\mathcal{P}\left(
\mathcal{C}([0,T];\mathbb{R}^{d})
\right)$ to denote the empirical measure on the path space, where
$V^{i,N}=(V_t^{i,N})_{t\in[0,T]}$. Its time marginal at time $t$
is exactly $\hat{\rho}_{t}^{N}$. As in \cite{HQ}{}, we then make the following definition.

\begin{definition}
Fix $\phi\in\mathcal{C}_{c}^{2}(\mathbb{R}^{d})$ and
$t\in[0,T]$. We restrict to $\mu\in
\mathcal{P}\left(\mathcal{C}([0,T];\mathbb{R}^{d})\right)$ satisfying $
\int_{\mathcal{C}([0,T];\mathbb{R}^{d})}
\sup_{\tau\in[0,T]}\|v(\tau)\|_2^4\,d\mu(v)<\infty$ and define the functional $
F_{\phi}$ by
\begin{align*}
F_{\phi}(\mu)
={}&
\left\langle \phi,\mu_{t}\right\rangle
-\left\langle \phi,\mu_{0}\right\rangle+\lambda\int_{0}^{t}
\left\langle
\left(v-v_{\alpha}(\mu_{\tau})\right)\cdot\nabla\phi(v),
\mu_{\tau}
\right\rangle\,d\tau\\
&+\dfrac{1}{\epsilon}\int_{0}^{t}
\left\langle
\nabla G(v)\cdot\nabla\phi(v),
\mu_{\tau}
\right\rangle\,d\tau-\dfrac{\sigma^{2}}{2}\int_{0}^{t}
\left\langle
\sum_{k=1}^{d}
\left(v-v_{\alpha}(\mu_{\tau})\right)_{k}^{2}
\partial_{kk}\phi(v),
\mu_{\tau}
\right\rangle\,d\tau,
\end{align*}
where $\mu_{\tau}$ denotes the time marginal of $\mu$.
\end{definition}
We can then prove the following proposition about the functional $ F_{\phi} $ defined above.

\begin{proposition}\label{prop}
	Let $ \mathcal{E} $ satisfy Assumption \ref{assump1} and $ \rho_{0}\in\mathcal{P}_{4}(\mathbb{R}^{d}) $. For any $ N\geq2 $, assume that $ \{(\viti)\}_{i=1}^{N} $ is the unique solution to (\ref{constraincon}) with $ \rho_{0}^{\otimes N} $ distributed initial data $ \{V_{0}^{i,N}\}_{i=1}^{N} $. There exists a constant $ C>0 $ depending only on $ \sigma,K,T $ and $ \big\|\nabla \phi\big\|_{\infty} $ such that $
		\mathbb{E}\left[|F_{\phi}(\hat{\rho}^N)|^{2}\right]\leq \frac{C}{N}.$
\end{proposition}

\begin{proof}
By the definition of $F_{\phi}$,
\begin{align*}
F_{\phi}(\hat{\rho}^{N})
={}&
\dfrac{1}{N}\sum_{i=1}^{N}\phi(V_{t}^{i,N})
-\dfrac{1}{N}\sum_{i=1}^{N}\phi(V_{0}^{i,N})\\
&+\lambda\int_{0}^{t}\dfrac{1}{N}\sum_{i=1}^{N}
\left(V_{\tau}^{i,N}-v_{\alpha}(\hat{\rho}_{\tau}^{N})\right)
\cdot\nabla\phi(V_{\tau}^{i,N})\,d\tau\\
&+\dfrac{1}{\epsilon}\int_{0}^{t}\dfrac{1}{N}\sum_{i=1}^{N}
\nabla G(V_{\tau}^{i,N})\cdot
\nabla\phi(V_{\tau}^{i,N})\,d\tau\\
&-\dfrac{\sigma^{2}}{2}\int_{0}^{t}\dfrac{1}{N}\sum_{i=1}^{N}
\sum_{k=1}^{d}
\left(V_{\tau}^{i,N}
-v_{\alpha}(\hat{\rho}_{\tau}^{N})\right)_{k}^{2}
\partial_{kk}\phi(V_{\tau}^{i,N})\,d\tau.
\end{align*}
On the other hand, the It\^o formula gives
\begin{align*}
\phi(V_{t}^{i,N})-\phi(V_{0}^{i,N})
={}&
-\lambda\int_{0}^{t}
\left(V_{\tau}^{i,N}
-v_{\alpha}(\hat{\rho}_{\tau}^{N})\right)
\cdot\nabla\phi(V_{\tau}^{i,N})\,d\tau\\
&-\dfrac{1}{\epsilon}\int_{0}^{t}
\nabla G(V_{\tau}^{i,N})\cdot
\nabla\phi(V_{\tau}^{i,N})\,d\tau\\
&+\sigma\int_{0}^{t}
\left(\nabla\phi(V_{\tau}^{i,N})\right)^{T}
\operatorname{diag}\left(
V_{\tau}^{i,N}
-v_{\alpha}(\hat{\rho}_{\tau}^{N})
\right)\,dB_{\tau}^{i,N}\\
&+\dfrac{\sigma^{2}}{2}\int_{0}^{t}
\sum_{k=1}^{d}
\left(V_{\tau}^{i,N}
-v_{\alpha}(\hat{\rho}_{\tau}^{N})\right)_{k}^{2}
\partial_{kk}\phi(V_{\tau}^{i,N})\,d\tau.
\end{align*}
Therefore,
\begin{gather*}
F_{\phi}(\hat{\rho}^{N})
=
\dfrac{\sigma}{N}
\sum_{i=1}^{N}\int_{0}^{t}
\left(\nabla\phi(V_{\tau}^{i,N})\right)^{T}
\operatorname{diag}\left(
V_{\tau}^{i,N}
-v_{\alpha}(\hat{\rho}_{\tau}^{N})
\right)\,dB_{\tau}^{i,N}.
\end{gather*}
Since $\{B^{i,N}\}_{i=1}^{N}$ are independent Wiener processes,
It\^o's isometry gives
\begin{align*}
\E\left[|F_{\phi}(\hat{\rho}^{N})|^{2}\right]
&=
\dfrac{\sigma^{2}}{N^{2}}
\sum_{i=1}^{N}
\E\int_{0}^{t}
\left\|
\left(\nabla\phi(V_{\tau}^{i,N})\right)^{T}
\operatorname{diag}\left(
V_{\tau}^{i,N}
-v_{\alpha}(\hat{\rho}_{\tau}^{N})
\right)
\right\|_{2}^{2}\,d\tau\\
&\leq
\dfrac{\sigma^{2}}{N^{2}}
\|\nabla\phi\|_{\infty}^{2}
\sum_{i=1}^{N}\int_{0}^{t}
\E\left[
\left\|V_{\tau}^{i,N}
-v_{\alpha}(\hat{\rho}_{\tau}^{N})\right\|_{2}^{2}
\right]\,d\tau\leq \dfrac{C}{N},
\end{align*}
where the last inequality follows from Lemma~\ref{momentbound}.
\end{proof}
Next we prove the tightness of $ \{\mathcal{L}(\hat{\rho}^{N})\}_{N\geq 2} $.
\begin{theorem}\label{tight}
	Under the same assumption as in Lemma \ref{momentbound}, the sequence
	$\{\mathcal{L}(\hat{\rho}^{N})\}_{N\geq 2}$ is tight in
	$\mathcal{P}\left(
	\mathcal{P}(\mathcal{C}([0,T];\mathbb{R}^{d}))
	\right)$.
\end{theorem}

\begin{proof}
It suffices to prove that
$\{\mathcal{L}(V^{1,N})\}_{N\geq2}$ is tight in
$\mathcal{P}(\mathcal{C}([0,T];\mathbb{R}^{d}))$
due to Proposition 2.2(ii) in \cite{S}.
We verify this directly using the Kolmogorov tightness criterion (see Corollary~16.9 in~\cite{kallenberg2002foundations}).

For any $0\leq s<t\leq T$, by \eqref{constraincon},
\begin{align*}
V_t^{1,N}-V_s^{1,N}
=&
-\lambda\int_s^t
\left(
V_\tau^{1,N}
-v_\alpha(\hat\rho_\tau^N)
\right)\,d\tau-\frac{1}{\epsilon}
\int_s^t
\nabla G(V_\tau^{1,N})\,d\tau\\
&+\sigma\int_s^t
\operatorname{diag}
\left(
V_\tau^{1,N}
-v_\alpha(\hat\rho_\tau^N)
\right)\,dB_\tau^{1,N}.
\end{align*}
By H\"older's inequality and Lemma~\ref{momentbound},
\begin{align*}
\E\left[
\left\|
\lambda\int_s^t
\left(
V_\tau^{1,N}
-v_\alpha(\hat\rho_\tau^N)
\right)\,d\tau
\right\|_2^4
\right]\leq
\lambda^4(t-s)^3
\int_s^t
\E\left[
\left\|
V_\tau^{1,N}
-v_\alpha(\hat\rho_\tau^N)
\right\|_2^4
\right]\,d\tau
\leq C(t-s)^4.
\end{align*}
Also, using the linear growth of $\nabla G$ following from
Assumption~\ref{assump1} (3),
\begin{align*}
\E\left[
\left\|
\frac{1}{\epsilon}
\int_s^t
\nabla G(V_\tau^{1,N})\,d\tau
\right\|_2^4
\right]\leq
\frac{(t-s)^3}{\epsilon^4}
\int_s^t
\E\left[
\|\nabla G(V_\tau^{1,N})\|_2^4
\right]\,d\tau
\leq C(t-s)^4.
\end{align*}
For the diffusion term, by the Burkholder--Davis--Gundy inequality, and
Lemma~\ref{momentbound},
\begin{align*}
&\E\left[
\left\|
\sigma\int_s^t
\operatorname{diag}
\left(
V_\tau^{1,N}
-v_\alpha(\hat\rho_\tau^N)
\right)\,dB_\tau^{1,N}
\right\|_2^4
\right]\\
&\qquad\leq
C\sigma^4
\E\left[
\left(
\int_s^t
\left\|
V_\tau^{1,N}
-v_\alpha(\hat\rho_\tau^N)
\right\|_2^2\,d\tau
\right)^2
\right]\\
&\qquad\leq
C(t-s)
\int_s^t
\E\left[
\left\|
V_\tau^{1,N}
-v_\alpha(\hat\rho_\tau^N)
\right\|_2^4
\right]\,d\tau
\leq C(t-s)^2.
\end{align*}
Combining the above estimates gives
\begin{gather}\label{eq:path_increment}
\sup_{N\geq2}
\E\left[
\|V_t^{1,N}-V_s^{1,N}\|_2^4
\right]
\leq
C_T|t-s|^2,
\qquad
0\leq s<t\leq T,
\end{gather}
where $C_T>0$ is independent of $N$, $s$, and $t$.
Since
$\mathcal{L}(V_0^{1,N})=\rho_0\in\mathcal{P}_4(\mathbb{R}^d)$
for every $N$, the Kolmogorov tightness criterion implies that
$\{\mathcal{L}(V^{1,N})\}_{N\geq2}$ is tight in
$\mathcal{P}(\mathcal{C}([0,T];\mathbb{R}^d))$.
Thus, Proposition 2.2(ii) in \cite{S} implies that
$\{\mathcal{L}(\hat{\rho}^{N})\}_{N\geq2}$ is tight in
$\mathcal{P}\left(
\mathcal{P}(\mathcal{C}([0,T];\mathbb{R}^{d}))
\right)$.
\end{proof}

Before we move on to the proof of Theorem~\ref{meanfieldthm}, we need some further preparations (especially~\eqref{limitmoment} and~\eqref{consensusconvergence}).

By Theorem~\ref{tight} and Prokhorov's theorem, every subsequence
of $\{\mathcal{L}(\hat{\rho}^{N})\}_{N\geq2}$ contains a further
subsequence (still denoted by the same index for simplicity), and
there exists $
\Pi\in
\mathcal{P}\left(
\mathcal{P}\left(
\mathcal{C}([0,T];\mathbb{R}^{d})
\right)
\right)$ such that $
\mathcal{L}(\hat{\rho}^{N})
\Rightarrow
\Pi
$ in $
\mathcal{P}\left(
\mathcal{P}\left(
\mathcal{C}([0,T];\mathbb{R}^{d})
\right)
\right).$ Since
$\mathcal{P}(\mathcal{C}([0,T];\mathbb{R}^{d}))$ is a Polish
space, the Skorokhod representation theorem gives random variables
$\widetilde{\rho}^{N}$ and $\widetilde{\rho}$ on a common
probability space such that $
\mathcal{L}(\widetilde{\rho}^{N})
=
\mathcal{L}(\hat{\rho}^{N}),$ $
\mathcal{L}(\widetilde{\rho})
=
\Pi,$
and $
\widetilde{\rho}^{N}
\rightarrow
\widetilde{\rho}
$ almost surely in $
\mathcal{P}\left(
\mathcal{C}([0,T];\mathbb{R}^{d})
\right).$
For simplicity, we omit the tildes in the following and work with the Skorohod representatives. As a result, for every fixed $t\in[0,T]$ and
$\phi\in\mathcal{C}_{b}(\mathbb{R}^{d})$,
\begin{gather}\label{slemma}
\left\langle
\phi,\hat{\rho}_{t}^{N}-\rho_{t}
\right\rangle
\rightarrow0
\end{gather}
almost surely.

Moreover, we now show that the limiting path measure $\rho$ has a finite
fourth moment with respect to the uniform norm. For
$x\in\mathcal{C}([0,T];\mathbb{R}^d)$, define
\begin{gather*}
\Psi(x)
:=
\sup_{t\in[0,T]}
\|x(t)\|_2^4,
\qquad
\Psi_R(x)
:=
\Psi(x)\wedge R.
\end{gather*}
Notice that
$\Psi_R\in
\mathcal{C}_b(\mathcal{C}([0,T];\mathbb{R}^d))$.
Since the Skorokhod representatives have the same laws as the
original empirical path measures, \eqref{eq:path_moment} gives $
\E\int
\Psi_R(x)\,d\hat\rho^N(x)
\leq K,$ for $ R>0.$
On the other hand, since
$\hat\rho^N\rightarrow\rho$ almost surely in
$\mathcal P(\mathcal C([0,T];\mathbb R^d))$, for every fixed $R>0$, $
\int\Psi_R(x)\,d\hat\rho^N(x)
\rightarrow
\int\Psi_R(x)\,d\rho(x)$
almost surely. Since $0\leq\Psi_R\leq R$, the dominated convergence
theorem gives $
\E\int\Psi_R(x)\,d\rho(x)
=
\lim_{N\rightarrow\infty}
\E\int\Psi_R(x)\,d\hat\rho^N(x)
\leq K.$ Letting $R\rightarrow\infty$ and applying the monotone convergence
theorem, we obtain
\begin{gather}\label{eq:limit_path_moment}
\E\int_{\mathcal C([0,T];\mathbb R^d)}
\sup_{t\in[0,T]}
\|x(t)\|_2^4
\,d\rho(x)
\leq K.
\end{gather}
In particular, $
\int_{\mathcal C([0,T];\mathbb R^d)}
\sup_{t\in[0,T]}
\|x(t)\|_2^4
\,d\rho(x)
<\infty $
almost surely.

We next show that this implies $ 
\rho\in
\mathcal C\left([0,T];\mathcal P_4(\mathbb R^d)\right) $ almost surely. Fix an outcome for which the above integral is finite.
For every $t\in[0,T]$,
\begin{align*}
\int_{\mathbb R^d}
\|v\|_2^4\,d\rho_t(v)
=
\int_{\mathcal C([0,T];\mathbb R^d)}
\|x(t)\|_2^4\,d\rho(x)\leq
\int_{\mathcal C([0,T];\mathbb R^d)}
\sup_{r\in[0,T]}
\|x(r)\|_2^4\,d\rho(x)
<\infty.
\end{align*}
Thus $\rho_t\in\mathcal P_4(\mathbb R^d)$ for every
$t\in[0,T]$. Moreover, using the coupling of $\rho_s$ and $\rho_t$
induced by the same path, one has $
W_4^4(\rho_t,\rho_s)
\leq
\int_{\mathcal C([0,T];\mathbb R^d)}
\|x(t)-x(s)\|_2^4\,d\rho(x).$ For every $x\in\mathcal C([0,T];\mathbb R^d)$, $
\|x(t)-x(s)\|_2^4
\rightarrow0 $ as $ t\rightarrow s,$
and also $
\|x(t)-x(s)\|_2^4
\leq
16\sup_{r\in[0,T]}\|x(r)\|_2^4.$ Then, by the dominated convergence theorem, $
W_4(\rho_t,\rho_s)\rightarrow0
$ as $t\rightarrow s.$ This proves $
\rho\in
\mathcal C\left([0,T];\mathcal P_4(\mathbb R^d)\right) $
almost surely. Finally, the fact that $\int_{\R^d}\|v\|_2^4\,d\rho_t = \int_{\mathcal{C}([0,T];\R^d)} \|x(t)\|_2^4\,d\rho(x)$ and \eqref{eq:limit_path_moment} also give the moment estimate,
\begin{gather}\label{limitmoment}
\sup_{t\in[0,T]}
\E\int\|v\|_2^4\,d\rho_t(v)
\leq K.
\end{gather}

We next show that, for each fixed $t\in[0,T]$,
\begin{gather}\label{consensusconvergence}
\E\left[
\left\|
v_{\alpha}(\hat{\rho}_{t}^{N})
-v_{\alpha}(\rho_{t})
\right\|_{2}^{2}
\right]
\rightarrow0.
\end{gather}
To see this, write
\[
A_N
=
\int v\omega_{\alpha}(v)\,d\hat{\rho}_{t}^{N}(v),
\qquad
A
=
\int v\omega_{\alpha}(v)\,d\rho_{t}(v),
\]
and
\[
Z_N
=
\int\omega_{\alpha}(v)\,d\hat{\rho}_{t}^{N}(v),
\qquad
Z
=
\int\omega_{\alpha}(v)\,d\rho_{t}(v).
\]
Since $\mathcal{E}$ is bounded, $
0<c_{\omega}
\leq
Z_N,Z
\leq C_{\omega}.$
Also, $\omega_{\alpha}$ is bounded and continuous, so
$Z_N\rightarrow Z$ almost surely and in $L^{2}$.

Let $\chi_R\in\mathcal{C}_{c}(\mathbb{R}^{d})$ satisfy
$0\leq\chi_R\leq1$, $\chi_R=1$ on $B_R(0)$ and
$\chi_R=0$ outside $B_{2R}(0)$. Then
$v\chi_R(v)\omega_{\alpha}(v)$ is bounded and continuous, and hence
\[
\int v\chi_R(v)\omega_{\alpha}(v)\,
d\hat{\rho}_{t}^{N}(v)
\rightarrow
\int v\chi_R(v)\omega_{\alpha}(v)\,
d\rho_{t}(v)
\]
almost surely. On the other hand, by Cauchy's inequality,
Lemma~\ref{momentbound}, and \eqref{limitmoment},
\begin{align*}
\E\left[
\left\|
\int v(1-\chi_R(v))\omega_{\alpha}(v)\,
d\hat{\rho}_{t}^{N}(v)
\right\|_{2}^{2}
\right]
&\leq
\dfrac{C}{R^{2}}
\E\int\|v\|_{2}^{4}\,d\hat{\rho}_{t}^{N}(v)
\leq\dfrac{C}{R^{2}},
\end{align*}
uniformly in $N$, and the same estimate holds for $\rho_t$.
Therefore $A_N\rightarrow A$ in $L^{2}$.
Since
\[
v_{\alpha}(\hat{\rho}_{t}^{N})=\dfrac{A_N}{Z_N},
\qquad
v_{\alpha}(\rho_t)=\dfrac{A}{Z},
\]
and $Z_N,Z$ are uniformly bounded away from zero,
\eqref{consensusconvergence} follows.

After all these preparations, we now prove
Theorem~\ref{meanfieldthm}.

\begin{proof}[Proof of Theorem \ref{meanfieldthm}]
Let $\rho$ be the Skorokhod representative of an arbitrary
subsequential limit constructed above. Thus, along the
corresponding subsequence, $
\hat{\rho}^{N}\rightarrow\rho
$ almost surely in $
\mathcal{P}\left(
\mathcal{C}([0,T];\mathbb{R}^{d})
\right).$

By the argument following \eqref{eq:limit_path_moment}, we have $
\rho\in
\mathcal C\left([0,T];\mathcal P_4(\mathbb R^d)\right) $
almost surely. In particular, condition (i) in
Definition~\ref{weakdef} is satisfied.

We next show that $\rho_t$ satisfies the weak formulation.
Fix $\phi\in\mathcal{C}_{c}^{2}(\mathbb{R}^{d})$.
By \eqref{slemma}, \eqref{consensusconvergence}, and the uniform
moment bounds in Lemma~\ref{momentbound}, one has
\begin{align}
\E\Bigg[
\Big|
&
\left(
\left\langle\phi,\hat{\rho}_{t}^{N}\right\rangle
-\left\langle\phi,\hat{\rho}_{0}^{N}\right\rangle
\right)
-
\left(
\left\langle\phi,\rho_t\right\rangle
-\left\langle\phi,\rho_0\right\rangle
\right)
\Big|
\Bigg]
&\rightarrow0, \label{limit1}\\
\E\Bigg[
\Big|
&
\int_{0}^{t}
\left\langle
\left(v-v_{\alpha}(\hat{\rho}_{\tau}^{N})\right)
\cdot\nabla\phi(v),
\hat{\rho}_{\tau}^{N}
\right\rangle\,d\tau -
\int_{0}^{t}
\left\langle
\left(v-v_{\alpha}(\rho_{\tau})\right)
\cdot\nabla\phi(v),
\rho_{\tau}
\right\rangle\,d\tau
\Big|
\Bigg]
&\rightarrow0, \label{limit2}\\
\E\Bigg[
\Big|
&
\int_{0}^{t}
\left\langle
\sum_{k=1}^{d}
\left(v-v_{\alpha}(\hat{\rho}_{\tau}^{N})\right)_{k}^{2}
\partial_{kk}\phi(v),
\hat{\rho}_{\tau}^{N}
\right\rangle\,d\tau \nonumber\\
&-
\int_{0}^{t}
\left\langle
\sum_{k=1}^{d}
\left(v-v_{\alpha}(\rho_{\tau})\right)_{k}^{2}
\partial_{kk}\phi(v),
\rho_{\tau}
\right\rangle\,d\tau
\Big|
\Bigg]
&\rightarrow0, \label{limit3}\\
\E\Bigg[
\Big|
&
\int_{0}^{t}
\left\langle
\nabla G(v)\cdot\nabla\phi(v),
\hat{\rho}_{\tau}^{N}
\right\rangle\,d\tau
-
\int_{0}^{t}
\left\langle
\nabla G(v)\cdot\nabla\phi(v),
\rho_{\tau}
\right\rangle\,d\tau
\Big|
\Bigg]
&\rightarrow0. \label{limit4}
\end{align}
Indeed, \eqref{limit1} follows from \eqref{slemma}. For
\eqref{limit2}, the functions
$v\cdot\nabla\phi(v)$ and $\nabla\phi(v)$ are bounded and
continuous, and the result follows from \eqref{slemma} and
\eqref{consensusconvergence}. For \eqref{limit3}, one expands
the square and applies the same argument together with the uniform
fourth-moment estimate. Finally,
$\nabla G(v)\cdot\nabla\phi(v)$ is bounded and continuous since
$\phi$ has compact support and $G\in\mathcal{C}^{1}$, so
\eqref{limit4} follows directly from \eqref{slemma}.
The convergence with respect to time follows from the dominated
convergence theorem and Lemma~\ref{momentbound}.

Therefore,
\begin{gather}\label{Fconvergence}
\E\left[
\left|
F_{\phi}(\hat{\rho}^{N})-F_{\phi}(\rho)
\right|
\right]
\rightarrow0.
\end{gather}
By Proposition~\ref{prop},
\[
\E\left[
|F_{\phi}(\hat{\rho}^{N})|
\right]
\leq
\left(
\E\left[
|F_{\phi}(\hat{\rho}^{N})|^{2}
\right]
\right)^{1/2}
\leq
\sqrt{\dfrac{C}{N}}.
\]
As a result,
\begin{align*}
\E\left[|F_{\phi}(\rho)|\right]
&\leq
\E\left[
\left|
F_{\phi}(\rho)-F_{\phi}(\hat{\rho}^{N})
\right|
\right]
+
\E\left[
|F_{\phi}(\hat{\rho}^{N})|
\right]\rightarrow0.
\end{align*}
Hence $F_{\phi}(\rho)=0$ almost surely. By a standard separability argument, first considering rational $t$ and a countable dense subset of $\mathcal{C}_c^2(\mathbb{R}^d)$ and then using continuity, the above identity holds almost surely for all $t\in[0,T]$ and $\phi\in\mathcal{C}_c^2(\mathbb{R}^d)$. Thus $\rho_t$ is almost
surely a weak solution to the Fokker--Planck equation
\eqref{FPK}. Since the initial particles are independently distributed according
to $\rho_0$, their empirical initial measures converge to $\rho_0$,
and therefore the limiting weak solution has initial data $\rho_0$.

Let $(\bar{\rho}_t)_{t\in[0,T]}$ denote the unique weak solution
to \eqref{FPK} with initial data $\rho_0$. Since the
subsequential limit constructed above is almost surely a weak
solution with the same initial data, Lemma~\ref{unique} gives
\begin{gather}\label{eq:marginal_limit_identification}
\rho_t=\bar{\rho}_t,
\qquad
t\in[0,T],
\end{gather}
almost surely. Fix $t\in[0,T]$. The slicing map
\[
\Gamma_t:
\mathcal{P}\left(
\mathcal{C}([0,T];\mathbb{R}^d)
\right)
\longrightarrow
\mathcal{P}(\mathbb{R}^d),
\qquad
\Gamma_t(\mu)=\mu_t,
\]
is continuous because the evaluation map
$v\mapsto v(t)$ is continuous. Therefore, along every
subsequence for which the laws of the empirical path measures
converge, the continuous mapping theorem and
\eqref{eq:marginal_limit_identification} give $
\mathcal{L}(\hat{\rho}_t^N)
\Longrightarrow
\mathcal{L}(\rho_t)
=
\delta_{\bar{\rho}_t}$ in $
\mathcal{P}\big(\mathcal{P}(\mathbb{R}^d)\big).$ Since every subsequence contains a further subsequence with this
limit, the whole sequence satisfies $
\mathcal{L}(\hat{\rho}_t^N)
\Rightarrow
\delta_{\bar{\rho}_t}.$
Since the limiting measure is deterministic, this is equivalent
to convergence in probability of $\hat{\rho}_t^N$ to
$\bar{\rho}_t$ in $\mathcal{P}(\mathbb{R}^d)$. This completes the
proof.
\end{proof}

\subsection{Some auxiliary results used in the proof of Theorem \ref{meanfieldthm}}

\begin{theorem}\label{linearwellposed}
Let $T>0$, $b\in\mathcal{C}([0,T];\mathbb{R}^d)$ and
$\rho_0\in\mathcal{P}_2(\mathbb{R}^d)$. Suppose that $G$ satisfies
Assumption~\ref{assump1} (3). Then the linear Fokker--Planck
equation
\begin{gather}\label{linear}
\partial_t\rho_t
=
\lambda\nabla\cdot\big((v-b_t)\rho_t\big)
+
\frac{1}{\epsilon}\nabla\cdot\big(\nabla G(v)\rho_t\big)
+
\frac{\sigma^2}{2}
\sum_{k=1}^d
\partial_{x_kx_k}
\left(
(v-b_t)_k^2\rho_t
\right),
\end{gather}
with initial data $\rho_0$, has a unique weak solution $
\rho\in\mathcal{C}\left([0,T];\mathcal{P}_2(\mathbb{R}^d)\right).$
Here, by a weak solution to \eqref{linear}, we mean
$\rho\in\mathcal{C}([0,T];\mathcal{P}_2(\mathbb{R}^d))$ satisfying,
for every $\phi\in\mathcal{C}_c^2(\mathbb{R}^d)$ and $t\in[0,T]$,
the same integral identity as in Definition~\ref{weakdef}, with  $v_\alpha(\rho_s)$ replaced by $b_s$.
\end{theorem}

\begin{proof}
Define $
\beta(t,v)
:=
-\lambda(v-b_t)-\frac{1}{\epsilon}\nabla G(v)$
and $
\Sigma(t,v)
:=
\sigma\operatorname{diag}(v-b_t).$ We further denote
\begin{gather*}
a(t,v)
:=
\Sigma(t,v)\Sigma(t,v)^{\top}
=
\sigma^2
\operatorname{diag}
\left(
(v-b_t)_1^2,\ldots,(v-b_t)_d^2
\right).
\end{gather*}
Since $\nabla G$ is globally Lipschitz, for all
$v,w\in\mathbb{R}^d$ and $t\in[0,T]$,
\begin{align*}
\|\beta(t,v)-\beta(t,w)\|_2
&\leq
\left(
\lambda+\frac{L_G}{\epsilon}
\right)
\|v-w\|_2,
\\
\|\Sigma(t,v)-\Sigma(t,w)\|_{\mathrm F}
&=
\sigma\|v-w\|_2.
\end{align*}
Moreover, since $b$ is continuous on $[0,T]$, it is bounded, and the
linear growth of $\nabla G$ following from
Assumption~\ref{assump1} (3) implies that $\beta$ and $\Sigma$ have at
most linear growth. Therefore, the SDE
\begin{equation}\label{linearSDE}
\begin{aligned}
dV_t
=
-\lambda(V_t-b_t)\,dt
-\frac{1}{\epsilon}\nabla G(V_t)\,dt
+\sigma\operatorname{diag}(V_t-b_t)\,dB_t,\qquad 
V_0\sim\rho_0,
\end{aligned}
\end{equation}
has a unique strong solution. Standard moment estimates then give $
\sup_{t\in[0,T]}
\mathbb{E}\big[\|V_t\|_2^2\big]
<\infty.
$ Moreover one has, $
\mathbb{E}\big[\|V_t-V_s\|_2^2\big]\rightarrow0$ as $t\rightarrow s.$
Thus $
\rho_t:=\operatorname{Law}(V_t)
\in
\mathcal{C}\left([0,T];\mathcal{P}_2(\mathbb{R}^d)\right).$
Applying It\^o's formula to
$\phi(V_t)$ for $\phi\in\mathcal{C}_c^2(\mathbb{R}^d)$ shows that
$\rho$ is a weak solution to \eqref{linear}. This proves existence.

We now prove uniqueness of the weak solution to the linear FPK equation. Let $
\widetilde\rho
\in
\mathcal{C}\left([0,T];\mathcal{P}_2(\mathbb{R}^d)\right)$ be any weak solution to \eqref{linear} with initial data $\rho_0$.
Since $[0,T]\ni t\mapsto\widetilde\rho_t$ is continuous in
$\mathcal{P}_2(\mathbb{R}^d)$,
\begin{gather}\label{linearweakmoment}
\sup_{t\in[0,T]}
\int\|v\|_2^2\,d\widetilde\rho_t(v)
<\infty.
\end{gather}
Let
\[
\mathcal{L}_t\phi(v)
=
\beta(t,v)\cdot\nabla\phi(v)
+
\frac12
a(t,v):\nabla^2\phi(v).
\]
The growth estimates of $\beta$ and $\Sigma$, and the definition of $a$ imply
\begin{gather*}
\|\beta(t,v)\|_2
\leq C(1+\|v\|_2),\qquad
\|a(t,v)\|_{\mathrm F}
\leq
\sigma^2\|v-b_t\|_2^2
\leq C(1+\|v\|_2^2).
\end{gather*}
As a result, by \eqref{linearweakmoment},
\begin{gather}\label{superpositionintegrability}
\int_0^T
\int_{\mathbb{R}^d}
\left(
\|\beta(t,v)\|_2+\|a(t,v)\|_{\mathrm F}
\right)
\,d\widetilde\rho_t(v)\,dt
<\infty.
\end{gather}

The integral weak formulation for time-independent test functions
implies the usual distributional formulation
\begin{gather*}
\int_0^T
\int_{\mathbb{R}^d}
\left(
\partial_t\psi(t,v)
+
\mathcal{L}_t\psi(t,v)
\right)
\,d\widetilde\rho_t(v)\,dt
=
0
\end{gather*}
for every
$\psi\in\mathcal{C}_c^{1,2}((0,T)\times\mathbb{R}^d)$.
Indeed, this first follows for simple test functions of the form
$\psi(t,v)=\eta(t)\phi(v)$ by integration by parts in time, and the
general case follows by a standard density argument.

Then, by
\cite[Theorem~2.5]{trevisan2016well}, there exists a continuous stochastic
process $\widetilde V_t$ on some probability space such that $
    \operatorname{Law}(\widetilde V_t)=\widetilde\rho_t,$ $ t\in[0,T],$
and the law of $\widetilde V$ on
$\mathcal{C}([0,T];\mathbb{R}^d)$ solves the martingale problem
associated with $\mathcal{L}_t$ (for the definition of the martingale
problem, please see
\cite[Definition~2.4]{trevisan2016well}).
Notice that the covariance matrix in $\mathcal{L}_t$ satisfies $
    a(t,v)=\Sigma(t,v)\Sigma(t,v)^{T}.$ By \cite[Chapter~5, Theorem~3.3]{EK}, $\widetilde V$ can therefore be realized as a weak solution to \eqref{linearSDE}.

On the other hand, as shown before, the drift and diffusion coefficients
of \eqref{linearSDE} are globally Lipschitz in the spatial variable.
Therefore, pathwise uniqueness holds, and hence uniqueness in law
follows from \cite[Chapter~5, Theorems~3.6--3.7]{EK}. As a result,
$\widetilde V$ has the same law as the strong solution $V$ of
\eqref{linearSDE}. Therefore,
\[
    \widetilde\rho_t
    =
    \operatorname{Law}(\widetilde V_t)
    =
    \operatorname{Law}(V_t)
    =
    \rho_t,
    \qquad t\in[0,T].
\]
This proves the uniqueness.
\end{proof}

\begin{lemma}\label{unique}
Assume that $
\rho^1,\rho^2
\in
\mathcal{C}\left(
[0,T];\mathcal{P}_4(\mathbb{R}^d)
\right)$ are two weak solutions to \eqref{FPK} in the sense of
Definition~\ref{weakdef}, with the same initial data $\rho_0$. Then $
\sup_{t\in[0,T]}
W_2\left(\rho_t^1,\rho_t^2\right)
=
0.$
\end{lemma}

\begin{proof}
For $i=1,2$, define $
b_t^i
:=
v_\alpha(\rho_t^i).$
Since $
\rho^i
\in
\mathcal{C}\left(
[0,T];\mathcal{P}_4(\mathbb{R}^d)
\right),$ the fourth moments of $\rho_t^i$ are uniformly bounded on $[0,T]$.
Therefore, Lemma~\ref{stability} implies that $
b^i\in\mathcal{C}([0,T];\mathbb{R}^d).$

Let $V_0$ be an $\mathcal{F}_0$-measurable random variable with law
$\rho_0$, and let $B$ be a $d$-dimensional Brownian motion independent
of $V_0$. On the same filtered probability space, let
$\widetilde V^1$ and $\widetilde V^2$ be the unique strong solutions
of
\begin{equation}\label{frozenSDEunique}
\begin{aligned}
d\widetilde V_t^i
={}&
-\lambda\left(\widetilde V_t^i-b_t^i\right)\,dt
-\frac{1}{\epsilon}\nabla G(\widetilde V_t^i)\,dt+
\sigma\operatorname{diag}
\left(\widetilde V_t^i-b_t^i\right)\,dB_t,
\\
\widetilde V_0^i
={}&V_0,
\qquad i=1,2.
\end{aligned}
\end{equation}
In particular, the two processes use both the same initial random
variable and the same Brownian motion.
Let $
\widetilde\rho_t^i
:=
\operatorname{Law}(\widetilde V_t^i).$
$\widetilde\rho^i$ solves the frozen linear
FPK equation
\begin{align*}
\partial_t\widetilde\rho_t^i
=
\lambda\operatorname{div}
\left(
(v-b_t^i)\widetilde\rho_t^i
\right)
+
\frac{1}{\epsilon}
\operatorname{div}
\left(
\nabla G(v)\widetilde\rho_t^i
\right)+
\frac{\sigma^2}{2}
\sum_{k=1}^d
\partial_{x_kx_k}
\left(
(v-b_t^i)_k^2
\widetilde\rho_t^i
\right),
\qquad
\widetilde\rho_0^i=\rho_0.
\end{align*}
Since $\rho^i$ solves \eqref{FPK} and
$b_t^i=v_\alpha(\rho_t^i)$, $\rho^i$ solves exactly the same
frozen linear equation. Theorem~\ref{linearwellposed} therefore gives
\begin{gather}\label{frozenlawidentification}
\widetilde\rho_t^i
=
\rho_t^i,
\qquad
t\in[0,T],
\quad i=1,2.
\end{gather}
We now compare the two processes. Set
\[
Z_t
:=
\widetilde V_t^1-\widetilde V_t^2,
\qquad
\Delta b_t
:=
b_t^1-b_t^2.
\]
Subtracting the two equations in \eqref{frozenSDEunique} gives
\begin{align*}
Z_t
=
-\lambda\int_0^t
\left(Z_s-\Delta b_s\right)\,ds-\frac{1}{\epsilon}
\int_0^t
\left(
\nabla G(\widetilde V_s^1)
-
\nabla G(\widetilde V_s^2)
\right)\,ds+
\sigma\int_0^t
\operatorname{diag}
\left(Z_s-\Delta b_s\right)\,dB_s.
\end{align*}
Using the global Lipschitz continuity of $\nabla G$, $
\|\operatorname{diag}(z)\|_{\mathrm F}
=
\|z\|_2,$ Cauchy's inequality, and the Burkholder--Davis--Gundy inequality, we
obtain
\begin{gather}\label{nonlinearFPKcoupling}
D(t)
\leq
C\int_0^t
\left(
D(s)+\|\Delta b_s\|_2^2
\right)\,ds,
\end{gather}
where $
D(t)
:=
\mathbb{E}
\left[
\sup_{r\in[0,t]}\|Z_r\|_2^2
\right],$
and $C>0$ depends only on
$T,\lambda,\sigma,\epsilon$ and the Lipschitz constant of $\nabla G$. By Lemma~\ref{stability},
\begin{align*}
\|\Delta b_s\|_2^2
=
\left\|
v_\alpha(\rho_s^1)
-
v_\alpha(\rho_s^2)
\right\|_2^2\leq
c_0^2
W_2^2(\rho_s^1,\rho_s^2).
\end{align*}
Using \eqref{frozenlawidentification} and the coupling
$(\widetilde V_s^1,\widetilde V_s^2)$, we further have
\begin{align*}
W_2^2(\rho_s^1,\rho_s^2)=
W_2^2(\widetilde\rho_s^1,\widetilde\rho_s^2)\leq
\mathbb{E}
\left[
\|\widetilde V_s^1-\widetilde V_s^2\|_2^2
\right]\leq D(s).
\end{align*}
As a result, $
\|\Delta b_s\|_2^2
\leq c_0^2D(s).
$ Plugging this into \eqref{nonlinearFPKcoupling} gives $
D(t)
\leq
C\int_0^tD(s)\,ds.$
Since $D(0)=0$, Gronwall's inequality gives $
D(t)=0,
$ for $t\in[0,T].$
So $
W_2(\rho_t^1,\rho_t^2)=0,
$ for $t\in[0,T],$ which proves the result.
\end{proof}
\section{Proof of Lemma~\ref{lem:assumpB-distance}}\label{appen:proofassumpB}
\begin{proof}We verify Assumption~\ref{wellbehave} (B) one by one.
{\paragraph{Verification of Assumption~\ref{wellbehave} (B1)}We first prove when $K$ is convex, the distance function $g$ is convex. Given $\lambda_1,\lambda_2 > 0$ with $\lambda_1 + \lambda_2 = 1$ and $v_1, v_2\in \mathbb{R}^d$. By the definition of distance function, for any $\epsilon > 0$, there exist $u_1$ and $u_2$ in $K$ such that,
\begin{gather*}
\|v_1 - u_1\|_2 \leq g(v_1) + \epsilon,\qquad
\|v_2 - u_2\|_2 \leq g(v_2) + \epsilon.
\end{gather*}
Thus one has
\begin{align*}
g(\lambda_1 v_1 + \lambda_2 v_2)&\leq \|\lambda_1 v_1 + \lambda_2 v_2 - \lambda_1 u_1 - \lambda_2 u_2\|_2\\
&\leq \lambda_1 \|v_1 - u_1\|_2 + \lambda_2 \|v_2 - u_2\|_2\\
&\leq \lambda_1 g(v_1) + \lambda_2 g(v_2) + \epsilon.
\end{align*}
The above inequality holds for all $\epsilon > 0$. This gives the convexity of $g$. Also, one knows that $x^2$ is a convex increasing function when $x\geq 0$. As a result, the composition $G = g^2$ is convex, thereby satisfying Assumption~\ref{wellbehave} (B1).}

{\paragraph{Verification of Assumption~\ref{wellbehave} (B2)}
Since we are dealing with a compact convex set $K\subset\R^{d}$, we can
conveniently work with the Moreau envelope. Let $P_K$ denote the metric
projection onto $K$. Then
\[
G(v) = \dist(v,K)^{2} = \|v - P_K(v)\|^{2}, \qquad v\in\R^{d}.
\]
Let $\iota_K$ be the indicator function of $K$, that is,
\[
\iota_K(v) =
\begin{cases}
0, & v\in K,\\
+\infty, & v\notin K.
\end{cases}
\]
The Moreau envelope of $\iota_K$ is defined by
\begin{gather*}
M_{\iota_K}(u)
=
\inf_{v\in\R^{d}}
\biggl\{
\iota_K(v) + \frac{1}{2}\,\|v - u\|^{2}
\biggr\}.
\end{gather*}
Since $\iota_K(v)=0$ for $v\in K$ and $+\infty$ otherwise, we obtain
\begin{gather*}
M_{\iota_K}(u)
=
\inf_{v\in K} \frac{1}{2}\,\|v - u\|^{2}
=
\frac{1}{2}\,\|u - P_K(u)\|^{2}
=
\frac{1}{2} G(u),
\end{gather*}
so $M_{\iota_K} = G/2$.}

{We now use~\cite[Proposition~12.30]{bauschke2017}, which yields $
\nabla M_{\iota_K}(u) = u - \mathrm{prox}_{\iota_K}(u),$ where $\mathrm{prox}_{\iota_K}$ is the proximal operator of $\iota_K$. In the
present case one has $\mathrm{prox}_{\iota_K} = P_K$, and therefore $
\nabla G(v)
=
2 \nabla M_{\iota_K}(v)
=
2\bigl(v - P_K(v)\bigr),$ so that
\begin{gather*}
\|\nabla G(v)\|^{2}
=
4 \|\nabla M_{\iota_K}(v)\|^{2}
=
4 \|v - P_K(v)\|^{2}
=
4 G(v),
\end{gather*}
for all $v\in\R^{d}$. Hence the inequality in Assumption~\ref{wellbehave}
(B2) is satisfied with equality.}

{It remains to prove that \HZnew{$G\in \mathcal{C}^{1,1}_{*}(\mathbb{R}^{d})$}. In
other words, we need to show that $|\partial_k G(x)|\leq C (1+|x_k|)$ for all
$k=1,\dots,d$, \HZnew{and
$\nabla G$ is globally Lipschitz.}}

{The first bound follows immediately from the expression
$\nabla G(v) = 2\bigl(v - P_K (v)\bigr)$:
\[
|\partial_k G(v)|
=
2\,|v_k - P_K (v)_k|
\leq
2\bigl(|v_k| + \sup_{u\in K} |u_k|\bigr)
\le C (1+|v_k|),
\]
\HZnew{where $\sup_{u\in K}|u_k|<\infty$ because $K$ is compact,
and $C>0$ is independent of $v$.}}

{For the second bound, we recall that the gradient of a Moreau envelope is
always globally Lipschitz; in particular, $\nabla M_{\iota_K}$ is
$1$–Lipschitz, hence $\nabla G = 2\nabla M_{\iota_K}$ is globally
Lipschitz.

Hence \HZnew{$G\in\mathcal{C}^{1,1}_{*}(\mathbb{R}^{d})$}, and Assumption~\ref{wellbehave}
(B2) is verified.
}

\paragraph{Verification of Assumption~\ref{wellbehave} (B3)}
{We first recall the relevant notions from convex geometry.}

\begin{definition}[Intrinsic volumes, {\cite[Definition~1.1]{lotz2020concentration}}]
Let $S\subset\R^d$ be a nonempty compact convex set. For each index
$j=0,1,\dots,d$, let $P_j \in\R^{d\times d}$ be the orthogonal projector
onto a fixed $j$–dimensional subspace of $\R^d$, and let
$Q\in\R^{d\times d}$ be a random rotation matrix distributed according to
the Haar measure on the special orthogonal group $\mathrm{SO}(d)$. The
\emph{intrinsic volumes} of $S$ are defined by
\[
V_j(S):= \binom{d}{j} \frac{\kappa_d}{\kappa_j \kappa_{d-j}}
\,\mathbb{E}_{Q} \bigl[\operatorname{Vol}_j (P_j Q S)\bigr],
\qquad j=0,\dots,d,
\]
where $\kappa_d$ denotes the volume of the $d$--dimensional Euclidean unit ball and $\operatorname{Vol}_j(\cdot)$ denotes the $j$--dimensional Lebesgue measure on the $j$--dimensional range of $P_j$. In particular, $V_d(S)$ coincides with the $d$–dimensional Lebesgue measure of $S$.
\end{definition}

\begin{definition}[Parallel sets]
Let $K\subset\R^d$ be a nonempty compact convex set and let
$\operatorname{dist}(\cdot,K)$ denote the Euclidean distance to $K$. For
$r\ge0$ the \emph{parallel set} of $K$ at distance $r$ is defined by $
K_{\le r} := \{v\in\R^d : \operatorname{dist}(v, K)\le r\}.$
\end{definition}

We will use the following standard results.

\begin{proposition}[Steiner formula, {\cite[Fact~2.1]{lotz2020concentration}}]\label{res1}
Let $K\subset\R^d$ be a nonempty compact convex set. For every $r\ge0$ one has $
V_{d}(K_{\leq r}) = \sum_{j = 0}^d \kappa_{d - j} V_{j}(K)\, r^{d-j}.$
\end{proposition}

\begin{proposition}[{\cite[Corollary~2.5]{rataj2010volume}}]\label{res2}
Let $K\subset\R^d$ be a nonempty compact convex set. For all $r>0$, except possibly for a countable set of values of $r$, one has $
\frac{d}{dr}V_d (K_{\leq r}) = \mathcal{H}^{d-1}(\partial K_{\leq r}),$ where $\mathcal{H}^{d-1}(\cdot)$ denotes the $(d-1)$–dimensional
Hausdorff measure.
\end{proposition}

We now prove Assumption~\ref{wellbehave}\,(B3) holds. Recall that
$G = g^{2}$ with $g(v) = \operatorname{dist}(v,K)$, and from the
verification of (B2) we have
\[
\|\nabla G(v)\| = \sqrt{4 G(v)} = 2g(v), \qquad v\in\R^d.
\]
At points where $g$ is differentiable and $g(v)>0$, $
\nabla G(v) = 2 g(v)\,\nabla g(v),$
so $
2g(v) = \|\nabla G(v)\| = \|2g(v)\,\nabla g(v)\|
= 2g(v)\,\|\nabla g(v)\|,$ and hence $\|\nabla g(v)\| = 1$ at all such points.

By the coarea formula applied to the Lipschitz function $g$ we obtain
\begin{align*}
\int_{\{G\in (0,1)\}} \frac{1}{\|\nabla G(v)\|}\,dv
&= \int_{\{G\in(0,1)\}}  \frac{1}{2 g(v)} \,\|\nabla g(v)\|\,dv \\
&= \int_0^1 \int_{\{g(v) = t\}} \frac{1}{2 g(v)} \,d\mathcal{H}^{d-1}(v)\,dt\\
&= \int_0^1\frac{1}{2t} \,\mathcal{H}^{d-1}\bigl(\{v\in\R^d : g(v)=t\}\bigr)\,dt\\
&= \int_0^1\frac{1}{2t} \,\mathcal{H}^{d-1}(\partial K_{\leq t})\,dt,
\end{align*}
where in the last step we used $\{v : g(v)=t\} = \partial K_{\le t}$.

From Propositions~\ref{res1} and~\ref{res2} we have, for all $t>0$
outside a countable set,
\begin{align*}
\mathcal{H}^{d-1}(\partial K_{\leq t})
= \frac{d}{dt} V_d (K_{\leq t})= \sum_{j=0}^{d-1}(d-j)\,\kappa_{d-j}\,V_{j}(K)\, t^{d-j-1}.
\end{align*}
Therefore,
\[
\int_{\{G\in (0,1)\}} \frac{1}{\|\nabla G(v)\|}\,dv
= \frac{1}{2}\sum_{j=0}^{d-1}(d-j)\,\kappa_{d-j}\,V_{j}(K)\int_0^1  t^{d-j-2}\,dt.
\]
For $j\le d-2$ the integral $\int_0^1 t^{d-j-2}\,dt$ is finite, whereas
for $j=d-1$ the integrand behaves like $t^{-1}$ near $0$ and the integral
diverges. Hence finiteness of the above expression requires $
V_{d-1}(K) = 0.$ A sufficient condition for this is that $K$ is contained in an affine
subspace of dimension less than or equal to $d-2$, in which case the
$(d-1)$–dimensional intrinsic volume $V_{d-1}(K)$ vanishes. Under this
condition, Assumption~\ref{wellbehave}\,(B3) holds.

\end{proof}

\section{Proof of Lemma \ref{C1}}
\label{C1details}

\begin{proof}
\HZnew{
We first prove the existence of $\tau_2$.
For $u>0$ sufficiently small, Assumption~\ref{wellbehave} (C1) and
Assumption~\ref{wellbehave} (A2) give $
	|\underline{\mathcal{E}_{u}}|
	=
	|\mathcal{E}(v_u)-\mathcal{E}(v^*)|\leq
	C\|v_u-v^*\|_2^\beta\leq
	C\tau_1(u)^\beta.$
Therefore, taking $
	\tau_2(x)=C\tau_1(x)^\beta$ proves the first result. Clearly, $\tau_2$ is non-negative and increasing,
and $\tau_2(x)\rightarrow0$ as $x\rightarrow0$.}

\HZnew{We next prove the existence of $\tau_3$. Let
\[
	v_1\in B(v_u,r)\cap\{G=u\},
	\qquad
	v_2\in B(v^*,r).
\]
For $u,r>0$ sufficiently small, both $v_1$ and $v_2$ belong to
$B(v^*,r_0)$. Thus, by Assumption~\ref{wellbehave} (A2),
\begin{align*}
	|\mathcal{E}(v_1)-\mathcal{E}(v_2)|
	&\leq
	C\|v_1-v_2\|_2^\beta\\
	&\leq
	C\left(
	\|v_1-v_u\|_2
	+\|v_u-v^*\|_2
	+\|v^*-v_2\|_2
	\right)^\beta\\
	&\leq
	C\left(
	2r+\tau_1(u)
	\right)^\beta\\
	&\leq
	C\left(
	r^\beta+\tau_1(u)^\beta
	\right)\\
	&\leq
	C\left(
	\max\{u,r\}^\beta
	+
	\tau_1(\max\{u,r\})^\beta
	\right),
\end{align*}
where $C>0$ denotes a generic constant.
Therefore,
\begin{align*}
	|\mathcal{E}_r^u-\mathcal{E}_r|
	&\leq
	\sup_{\substack{
	v_1\in B(v_u,r)\cap\{G=u\}\\
	v_2\in B(v^*,r)
	}}
	|\mathcal{E}(v_1)-\mathcal{E}(v_2)|\\
	&\leq
	C\left(
	\max\{u,r\}^\beta
	+
	\tau_1(\max\{u,r\})^\beta
	\right).
\end{align*}
Thus, one may take $
	\tau_3(x)
	=
	C\left(
	x^\beta+\tau_1(x)^\beta
	\right).$ It is clear that $\tau_3$ is non-negative and increasing and
$\tau_3(x)\rightarrow0$ as $x\rightarrow0$.}

\HZnew{Finally, the existence of $\tau_4$ follows by the same argument.
Indeed, for
\[
	v_1\in B(v_u,r)\cap\{G=u\},
	\qquad
	v_2\in B(v^*,r)\cap\{G=0\},
\]
one has
\begin{align*}
	|\mathcal{E}(v_1)-\mathcal{E}(v_2)|
	\leq
	C\left(
	r^\beta+\tau_1(u)^\beta
	\right)\leq
	C\left(
	\max\{u,r\}^\beta
	+
	\tau_1(\max\{u,r\})^\beta
	\right).
\end{align*}
Consequently,
\[
	|\mathcal{E}_r^u-\mathcal{E}_r^0|
	\leq
	C\left(
	\max\{u,r\}^\beta
	+
	\tau_1(\max\{u,r\})^\beta
	\right),
\]
and hence one may take $
	\tau_4(x)
	=
	C\left(
	x^\beta+\tau_1(x)^\beta
	\right).$ This finishes the proof.
}
\end{proof}

\section{Explanations for expanding the test function space}
\label{testextension}

\HZnew{
We follow the same argument as in \cite{FKR1}, Remark~1.
Let $\phi\in\mathcal{C}_{*}^{2}(\mathbb{R}^{d})$ and let
$\chi\in\mathcal{C}_{c}^{\infty}(\mathbb{R}^{d})$ satisfy
$0\leq\chi\leq1$, $\chi=1$ on $B(0,1)$ and $\chi=0$ outside
$B(0,2)$. For $R>0$, define
\[
    \chi_R(v)=\chi(v/R),
    \qquad
    \phi_R(v)=\chi_R(v)\phi(v).
\]
Then $\phi_R\in\mathcal{C}_{c}^{2}(\mathbb{R}^{d})$. Moreover, by the
definition of $\mathcal{C}_{*}^{2}(\mathbb{R}^{d})$,
\begin{gather*}
    |\phi(v)|\leq C(1+\|v\|_2^2),\qquad
    \|\nabla\phi(v)\|_2\leq C(1+\|v\|_2),\qquad
    |\Delta\phi(v)|\leq C.
\end{gather*}
Using $
    \|\nabla\chi_R\|_{\infty}\leq C/R,$ and $
    \|\Delta\chi_R\|_{\infty}\leq C/R^2,$ we obtain, uniformly in $R\geq1$,
\begin{gather}\label{cutoff_bound}
    |\phi_R(v)|\leq C(1+\|v\|_2^2),\qquad
    \|\nabla\phi_R(v)\|_2\leq C(1+\|v\|_2),\qquad
    |\Delta\phi_R(v)|\leq C.
\end{gather}
Furthermore,
\[
    \phi_R(v)\rightarrow\phi(v),\qquad
    \nabla\phi_R(v)\rightarrow\nabla\phi(v),\qquad
    \Delta\phi_R(v)\rightarrow\Delta\phi(v)
\]
pointwise as $R\rightarrow\infty$.}

\HZnew{We now apply Definition~\ref{weakdef_iso} with $\phi_R$. Since
$\rho\in\mathcal{C}([0,T],\mathcal{P}_{4}(\mathbb{R}^{d}))$, its
moments up to order four are uniformly bounded on $[0,T]$.
Moreover, $\mathcal{E}$ is bounded, and hence $
    \sup_{t\in[0,T]}\|v_\alpha(\rho_t)\|_2<\infty.$ Since $\nabla G$ has at most linear growth, all the integrands in
Definition~\ref{weakdef_iso}, with $\phi_R$ in place of $\phi$, are
bounded in absolute value by $
    C(1+\|v\|_2^2),$ uniformly in $R$ and $t\in[0,T]$. Therefore, by the dominated
convergence theorem, we may let $R\rightarrow\infty$ in the weak
formulation and obtain
\begin{align*}
\langle\phi,\rho_t\rangle-\langle\phi,\rho_0\rangle
={}&-\lambda\int_0^t
\left\langle
(v-v_\alpha(\rho_s))\cdot\nabla\phi,\rho_s
\right\rangle ds-\frac{1}{\epsilon}\int_0^t
\left\langle\nabla G\cdot\nabla\phi,\rho_s\right\rangle ds\\
&+\frac{\sigma^2}{2}\int_0^t
\left\langle
\|v-v_\alpha(\rho_s)\|_2^2\Delta\phi,\rho_s
\right\rangle ds.
\end{align*}
Thus, the equality in Definition~\ref{weakdef_iso} also holds for any
$\phi\in\mathcal{C}_{*}^{2}(\mathbb{R}^{d})$.
}

\section{Proof of Lemma \ref{energy}}\label{appen:energy}

\begin{proof}
    By Remark~\ref{remark:extension}, we may use
    $\phi(v)=\dfrac{1}{2}\|v-v^{*}\|_2^{2}$ in
    Definition~\ref{weakdef_iso}. Since $
    \nabla\phi(v)=v-v^{*},$ and $
    \Delta\phi(v)=d,$ one obtains, for a.e. $t\in(0,T)$,
    \begin{equation}\label{lemma1main}
    \begin{aligned}
        \dfrac{d}{dt}\mathcal{V}\left(\rho_{t}\right)
        =&-\lambda
        \int
        \Bigl\langle
        v-v_{\alpha}\left(\rho_{t}\right),
        v-v^{*}
        \Bigr\rangle
        \,d\rho_{t}-\dfrac{1}{\epsilon}
        \int
        \Bigl\langle
        \nabla G,v-v^{*}
        \Bigr\rangle
        \,d\rho_{t}+\HZnew{
        \dfrac{d\sigma^{2}}{2}
        \int
        \|v-v_{\alpha}(\rho_t)\|_2^{2}
        \,d\rho_{t}.
        }
    \end{aligned}
    \end{equation}
Notice that
    \begin{align*}
        -\lambda
        \int
        \Bigl\langle
        v-v_{\alpha}\left(\rho_{t}\right),
        v-v^{*}
        \Bigr\rangle
        \,d\rho_t(v)&=
        -\lambda
        \int
        \|v-v^{*}\|_2^2
        \,d\rho_t(v)
        +
        \lambda
        \int
        \Bigl\langle
        v-v^{*},
        v_{\alpha}\left(\rho_{t}\right)-v^{*}
        \Bigr\rangle
        \,d\rho_t(v)
        \\
        &=
        -2\lambda\mathcal{V}\left(\rho_{t}\right)
        +
        \lambda
        \Bigl\langle
        \int(v-v^{*})\,d\rho_t(v),
        v_{\alpha}\left(\rho_{t}\right)-v^{*}
        \Bigr\rangle.
    \end{align*}
    Therefore,
    \begin{equation}\label{lemma1ineq1}
    \begin{aligned}
        -\lambda
        \int
        \Bigl\langle
        v-v_{\alpha}\left(\rho_{t}\right),
        v-v^{*}
        \Bigr\rangle
        \,d\rho_t(v)\leq
        -2\lambda\mathcal{V}\left(\rho_{t}\right)
        +
        \lambda
        \sqrt{2\mathcal{V}\left(\rho_{t}\right)}
        \,
        \|v_{\alpha}\left(\rho_{t}\right)-v^{*}\|_2.
    \end{aligned}
    \end{equation}
For the last term on the right-hand side of
    \eqref{lemma1main}, we have
    \HZnew{
    \begin{equation}\label{lemma1ineq2}
    \begin{aligned}
        &\dfrac{d\sigma^{2}}{2}
        \int
        \|v-v_{\alpha}(\rho_t)\|_2^{2}
        \,d\rho_t(v) \\
        &=
        \dfrac{d\sigma^{2}}{2}
        \bigg(
        \int\|v-v^{*}\|_2^{2}\,d\rho_t(v)
        -
        2
        \Bigl\langle
        \int(v-v^{*})\,d\rho_t(v),
        v_{\alpha}(\rho_t)-v^{*}
        \Bigr\rangle
        +
        \|v_{\alpha}(\rho_t)-v^{*}\|_2^{2}
        \bigg)
        \\
        &\leq
        d\sigma^{2}
        \bigg(
        \mathcal{V}(\rho_t)
        +
        \sqrt{2\mathcal{V}(\rho_t)}
        \|v_{\alpha}(\rho_t)-v^{*}\|_2
        +
        \dfrac12
        \|v_{\alpha}(\rho_t)-v^{*}\|_2^{2}
        \bigg).
    \end{aligned}
    \end{equation}}
Combining \eqref{lemma1ineq1} and
    \eqref{lemma1ineq2} with \eqref{lemma1main} gives
    \eqref{eq:dyn}.
\end{proof}

\section{Lemmas used in Laplace's principle}
In this section, we present the two lemmas (Lemma~\ref{lp1} and Lemma~\ref{lp2}) used in the proof of Lemma~\ref{lp}.
\subsection{Lemma \ref{lp1} and proof}\label{lp1details}
\begin{lemma}\label{lp1}
	Fix $r\in(0,R_0)$ small enough. \HZnew{Assume in addition that
	$\rho_t\big(B(v^*,r)\big)>0$.} For any $q>0$ satisfying
	$q+\mathcal{E}^{0}_{r}<\mathcal{E}_{\infty}$,
	\begin{align*}
		\int_{\{G=0\}}
		\dfrac{\|v-v^{*}\|_{2}}
		{\|\wt\|_{L^{1}(\rho_{t})}}
		e^{-\alpha\mathcal{E}(v)}\dvt\leq
		\HZnew{
		\dfrac{(q+\ec)^{\mu}}{\eta}
		+\dfrac{e^{-\alpha\left(q-\tau_{3}(r)\right)}}
		{\rt\left(B(v^{*},r)\right)}
		\int_{\{G=0\}}\|v-v^{*}\|_{2}\dvt.
		}
	\end{align*}
\end{lemma}
\begin{proof}
	Let $
		\tilde{r}=\frac{(q+\ec)^{\mu}}{\eta}.$
	We first verify that
	\begin{enumerate}
		\item $\tilde{r}\geq r$;
		\item $\mathcal{E}(v)-\ec\geq q$ for all
		$v\in\{G=0\}\cap B(v^{*},\tilde{r})^{c}$.
	\end{enumerate}

	For (1), by Assumption \ref{wellbehave} (C1), there exists
	$v\in\partial B(v^{*},r)\cap\{G=0\}$. Since $r<R_0$,
	Assumption \ref{wellbehave} (C2) gives
	\[
		r
		=\|v-v^{*}\|_2
		\leq
		\dfrac{(\mathcal{E}(v)-\underline{\mathcal{E}_{0}})^{\mu}}{\eta}
		\leq
		\dfrac{(\ec-\underline{\mathcal{E}_{0}})^{\mu}}{\eta}
		\leq\tilde r,
	\]
	where $\underline{\mathcal{E}_{0}}=\mathcal{E}(v^{*})=0$.

	For (2), let
	$v\in\{G=0\}\cap B(v^{*},\tilde r)^c$.
	\HZnew{If $\|v-v^{*}\|_2\leq R_0$, Assumption
	\ref{wellbehave} (C2) yields
	\[
		\mathcal{E}(v)-\underline{\mathcal{E}_{0}}
		\geq
		(\eta\|v-v^{*}\|_2)^{1/\mu}
		\geq
		(\eta\tilde r)^{1/\mu}
		=q+\ec.
	\]
	If instead $\|v-v^{*}\|_2> R_0$, the second part of
	Assumption \ref{wellbehave} (C2), together with
	$q+\ec<\mathcal{E}_{\infty}$, gives $
		\mathcal{E}(v)-\underline{\mathcal{E}_{0}}
		>\mathcal{E}_{\infty}>q+\ec.$
	In either case,
	$\mathcal{E}(v)-\ec\geq q$.}

	We next notice that
	\begin{equation}\label{wt}
    \begin{aligned}
		\|\wt\|_{L^{1}(\rho_t)}
		&=\int e^{-\alpha\mathcal{E}(v)}\dvt\geq
		\int_{B(v^{*},r)}
		e^{-\alpha\mathcal{E}(v)}\dvt\geq
		e^{-\alpha\er}\rt\big(B(v^{*},r)\big).
    \end{aligned}
	\end{equation}
	By Lemma \ref{C1},
	$\er\leq\ec+\tau_3(r)$.

	We now split
	\[
		\{G=0\}
		=
		\big(\{G=0\}\cap B(v^{*},\tilde r)\big)
		\cup
		\big(\{G=0\}\cap B(v^{*},\tilde r)^c\big).
	\]
	For the first part,
	\begin{align*}
		&\int_{\{G=0\}\cap B(v^{*},\tilde r)}
		\dfrac{\|v-v^{*}\|_2}
		{\|\wt\|_{L^{1}(\rho_t)}}
		e^{-\alpha\mathcal E(v)}\dvt\leq \tilde r.
	\end{align*}
	For the second part, using \eqref{wt}, Lemma \ref{C1},
	and (2), we obtain
	\begin{align*}
		&\int_{\{G=0\}\cap B(v^{*},\tilde r)^c}
		\dfrac{\|v-v^{*}\|_2}
		{\|\wt\|_{L^{1}(\rho_t)}}
		e^{-\alpha\mathcal E(v)}\dvt\leq
		\dfrac{e^{-\alpha(q-\tau_3(r))}}
		{\rt(B(v^{*},r))}
		\int_{\{G=0\}}\|v-v^{*}\|_2\dvt.
	\end{align*}
	Combining the two estimates proves the result.
\end{proof}

\subsection{Lemma \ref{lp2} and proof}\label{lp2details}
\begin{lemma}\label{lp2}
	Fix $0<u<u_0$ and $r\in(0,R_0)$ small enough. \HZnew{Assume in addition that
	$\rho_t\big(B(v^*,r)\big)>0$.} Suppose $q>0$
	satisfies $
		q+\mathcal{E}_{r}^{\tilde{u}}
		-\underline{\mathcal{E}_{\tilde{u}}}
		<\mathcal{E}_{\infty},
		\ 
		\forall\tilde{u}\in(0,u).$
	Then
	\begin{align*}
		&\int_{\{G\in(0,u)\}}
		\dfrac{\|v-v^{*}\|_2}
		{\|\wt\|_{L^{1}(\rho_t)}}
		e^{-\alpha\mathcal{E}(v)}\dvt\\
		&\leq
		\HZnew{
		\dfrac{\left(q+\ec+\tau_{2}(u)
		+\tau_{4}(\max\{u,r\})\right)^{\mu}}{\eta}
		}\\
		&\quad+
		\HZnew{
		\dfrac{e^{-\alpha\left(q-\tau_{3}(\max\{u,r\})\right)}}
		{\rt\left(B(v^{*},r)\right)}
		\int_{\{G\in(0,u)\}}\|v-v_{G(v)}\|_2\dvt
		}+\HZnew{\tau_{1}(u)}.
	\end{align*}
\end{lemma}
\begin{proof}
	By the co-area formula, one has
	\begin{align*}
		&\int_{\{G\in(0,u)\}}
		\dfrac{\|v-v^{*}\|_2}
		{\|\wt\|_{L^{1}(\rho_{t})}}
		e^{-\alpha\mathcal{E}(v)}\dvt\\
		&=
		\int_{0}^{u}d\tilde{u}
		\int_{\{G(v)=\tilde{u}\}}
		\dfrac{\|v-v^{*}\|_2}
		{\|\wt\|_{L^{1}(\rho_{t})}\|\nabla G(v)\|_2}
		e^{-\alpha\mathcal{E}(v)}
		\rho_t(v)\dH.
	\end{align*}
	Here, we use Assumption~\ref{wellbehave} (B3) and the fact that
	$\rho_t$ admits a density.

	Fix $\tilde{u}\in(0,u)$ and set
	\begin{gather}\label{rtutilde}
		\tilde{r}_{\tilde{u}}
		=
		\dfrac{
		\left(
		q+\mathcal{E}_{r}^{\tilde{u}}
		-\underline{\mathcal{E}_{\tilde{u}}}
		\right)^{\mu}}
		{\eta}.
	\end{gather}
	We first claim that
	\HZnew{
	\begin{gather}\label{lp2gap}
		\mathcal{E}(v)-\mathcal{E}_{r}^{\tilde{u}}
		\geq q,
		\qquad
		\forall
		v\in
		\{G=\tilde{u}\}
		\cap
		B(v_{\tilde{u}},\tilde r_{\tilde u})^{c}.
	\end{gather}
	}
	Indeed, if
	$\|v-v_{\tilde{u}}\|_2\leq R_0$, then Assumption
	\ref{wellbehave} (C2) and \eqref{rtutilde} give
	\begin{align*}
		\mathcal{E}(v)-\underline{\mathcal{E}_{\tilde{u}}}
		\geq
		\left(
		\eta\|v-v_{\tilde{u}}\|_2
		\right)^{1/\mu}\geq
		q+\mathcal{E}_{r}^{\tilde{u}}
		-\underline{\mathcal{E}_{\tilde{u}}},
	\end{align*}
	which proves \eqref{lp2gap}. On the other hand, if
	$\|v-v_{\tilde{u}}\|_2> R_0$, then by the second part of
	Assumption~\ref{wellbehave} (C2) and the assumption of the lemma, $
		\mathcal{E}(v)-\underline{\mathcal{E}_{\tilde{u}}}>
		\mathcal{E}_{\infty}>
		q+\mathcal{E}_{r}^{\tilde{u}}
		-\underline{\mathcal{E}_{\tilde{u}}},$
	and \eqref{lp2gap} follows again.

	We next obtain a uniform upper bound for
	$\tilde r_{\tilde u}$. By Lemma~\ref{C1},
	\HZnew{
	\begin{align*}
		\mathcal{E}_{r}^{\tilde{u}}
		-\underline{\mathcal{E}_{\tilde{u}}}
		&=
		\ec+
		\left(
		\mathcal{E}_{r}^{\tilde{u}}-\ec
		\right)
		-\underline{\mathcal{E}_{\tilde{u}}}\\
		&\leq
		\ec+
		\tau_{4}(\max\{\tilde{u},r\})
		+\tau_{2}(\tilde{u})\leq
		\ec+
		\tau_{4}(\max\{u,r\})
		+\tau_{2}(u).
	\end{align*}
	}
	Therefore,
	\begin{gather}\label{rtu_bound}
		\tilde r_{\tilde u}
		\leq
		\tilde r
		:=
		\HZnew{
		\dfrac{
		\left(
		q+\ec+\tau_2(u)
		+\tau_4(\max\{u,r\})
		\right)^\mu
		}{\eta}.
		}
	\end{gather}

	By the triangle inequality, $
		\|v-v^*\|_2
		\leq
		\|v-v_{\tilde u}\|_2
		+
		\|v_{\tilde u}-v^*\|_2.$
	Hence, for the inner integral, we split the first term into
	$B(v_{\tilde u},\tilde r_{\tilde u})$ and its complement.
	For the part inside the ball, by \eqref{rtu_bound},
	\begin{align*}
		&\int_{\{G=\tilde u\}\cap
		B(v_{\tilde u},\tilde r_{\tilde u})}
		\dfrac{\|v-v_{\tilde u}\|_2}
		{\|\wt\|_{L^1(\rho_t)}\|\nabla G(v)\|_2}
		e^{-\alpha\mathcal E(v)}
		\rho_t(v)\dH\\
		&\leq
		\tilde r
		\int_{\{G=\tilde u\}}
		\dfrac{1}
		{\|\wt\|_{L^1(\rho_t)}\|\nabla G(v)\|_2}
		e^{-\alpha\mathcal E(v)}
		\rho_t(v)\dH.
	\end{align*}

	For the part outside the ball, first notice that
	\begin{gather}\label{wt_lp2}
		\|\wt\|_{L^1(\rho_t)}
		\geq
		e^{-\alpha\er}
		\rho_t\big(B(v^*,r)\big).
	\end{gather}
	Moreover, Lemma~\ref{C1} gives $
		\er
		\leq
		\mathcal E_r^{\tilde u}
		+
		\tau_3(\max\{\tilde u,r\})
		\leq
		\mathcal E_r^{\tilde u}
		+
		\tau_3(\max\{u,r\}).$ Together with \eqref{lp2gap}, this yields $
		\mathcal E(v)-\er
		\geq
		q-\tau_3(\max\{u,r\})$
	on
	$\{G=\tilde u\}\cap
	B(v_{\tilde u},\tilde r_{\tilde u})^c$.
	Consequently,
	\begin{align*}
		&\int_{\{G=\tilde u\}\cap
		B(v_{\tilde u},\tilde r_{\tilde u})^c}
		\dfrac{\|v-v_{\tilde u}\|_2}
		{\|\wt\|_{L^1(\rho_t)}\|\nabla G(v)\|_2}
		e^{-\alpha\mathcal E(v)}
		\rho_t(v)\dH\\
		&\leq
		\HZnew{
		\dfrac{
		e^{-\alpha(q-\tau_3(\max\{u,r\}))}
		}{
		\rho_t(B(v^*,r))
		}
		\int_{\{G=\tilde u\}}
		\dfrac{\|v-v_{\tilde u}\|_2}
		{\|\nabla G(v)\|_2}
		\rho_t(v)\dH.
		}
	\end{align*}

	Finally, Assumption~\ref{wellbehave} (C1) gives $
		\|v_{\tilde u}-v^*\|_2
		\leq
		\tau_1(\tilde u)
		\leq
		\tau_1(u),$
	and hence
	\begin{align*}
		&\int_{\{G=\tilde u\}}
		\dfrac{\|v_{\tilde u}-v^*\|_2}
		{\|\wt\|_{L^1(\rho_t)}\|\nabla G(v)\|_2}
		e^{-\alpha\mathcal E(v)}
		\rho_t(v)\dH\\
		&\leq
		\tau_1(u)
		\int_{\{G=\tilde u\}}
		\dfrac{1}
		{\|\wt\|_{L^1(\rho_t)}\|\nabla G(v)\|_2}
		e^{-\alpha\mathcal E(v)}
		\rho_t(v)\dH.
	\end{align*}

	Combining the above estimates and integrating with respect to
	$\tilde u$ over $(0,u)$, the co-area formula gives
	\begin{align*}
		&\int_{\{G\in(0,u)\}}
		\dfrac{\|v-v^*\|_2}
		{\|\wt\|_{L^1(\rho_t)}}
		e^{-\alpha\mathcal E(v)}
		\dvt\\
		&\leq
		\left(\tilde r+\tau_1(u)\right)
		\dfrac{
		\int_{\{G\in(0,u)\}}
		e^{-\alpha\mathcal E(v)}\dvt
		}{
		\|\wt\|_{L^1(\rho_t)}
		}+
		\dfrac{
		e^{-\alpha(q-\tau_3(\max\{u,r\}))}
		}{
		\rho_t(B(v^*,r))
		}
		\int_{\{G\in(0,u)\}}
		\|v-v_{G(v)}\|_2\dvt\\
		&\leq
		\tilde r+\tau_1(u)
		+
		\dfrac{
		e^{-\alpha(q-\tau_3(\max\{u,r\}))}
		}{
		\rho_t(B(v^*,r))
		}
		\int_{\{G\in(0,u)\}}
		\|v-v_{G(v)}\|_2\dvt.
	\end{align*}
	Substituting the definition of $\tilde r$ in
	\eqref{rtu_bound} finishes the proof.
\end{proof}
\section{Proof of Laplace Principle: Lemma \ref{lp}}
\label{appendix:laplace_proof}

\begin{proof}
	By the definition of the consensus point,
	\begin{align*}
		\|\con-v^{*}\|_2
		=
		\left\|
		\int(v-v^{*})
		\dfrac{e^{-\alpha\mathcal{E}(v)}}
		{\|\wt\|_{L^{1}(\rt)}}\dvt
		\right\|_2\leq
		\int
		\dfrac{\|v-v^{*}\|_2}
		{\|\wt\|_{L^{1}(\rt)}}
		e^{-\alpha\mathcal{E}(v)}\dvt.
	\end{align*}
	We decompose the last integral as
	\begin{align*}
		\|\con-v^{*}\|_2
		\leq&
		\int_{\{G=0\}}
		\dfrac{\|v-v^{*}\|_2}
		{\|\wt\|_{L^{1}(\rt)}}
		e^{-\alpha\mathcal{E}(v)}\dvt+
		\int_{\{G\in(0,u)\}}
		\dfrac{\|v-v^{*}\|_2}
		{\|\wt\|_{L^{1}(\rt)}}
		e^{-\alpha\mathcal{E}(v)}\dvt\\
		&+
		\int_{\{G\geq u\}}
		\dfrac{\|v-v^{*}\|_2}
		{\|\wt\|_{L^{1}(\rt)}}
		e^{-\alpha\mathcal{E}(v)}\dvt.
	\end{align*}

	Applying Lemma~\ref{lp1} to the first term and
	Lemma~\ref{lp2} to the second one gives
	\begin{align*}
	\|\con-v^{*}\|_2
	\leq&
	\dfrac{(q+\ec)^\mu}{\eta}
	+
	\dfrac{
	e^{-\alpha(q-\tau_3(r))}
	}{
	\rt(B(v^{*},r))
	}
	\int_{\{G=0\}}
	\|v-v^{*}\|_2\dvt\\
	&+
	\dfrac{
	\left(
	q+\ec+\tau_2(u)+\tau_4(\max\{u,r\})
	\right)^\mu
	}{\eta}\\
	&+
	\dfrac{
	e^{-\alpha(q-\tau_3(\max\{u,r\}))}
	}{
	\rt(B(v^{*},r))
	}
	\int_{\{G\in(0,u)\}}
	\|v-v_{G(v)}\|_2\dvt\\
	&+\tau_1(u)
	+
	\int_{\{G\geq u\}}
	\dfrac{\|v-v^{*}\|_2}
	{\|\wt\|_{L^{1}(\rt)}}
	e^{-\alpha\mathcal{E}(v)}\dvt.
	\end{align*}
	Since $\tau_2,\tau_4$ are non-negative, $
	(q+\ec)^\mu
	\leq
	\left(
	q+\ec+\tau_2(u)+\tau_4(\max\{u,r\})
	\right)^\mu.$
	Combining the first and third terms therefore gives
	\eqref{eq_11}.
\end{proof}
\section{The Complete Proof of Lemma \ref{ball}}
\label{balldetails}

\begin{proof}
	Since $0\leq\phi_r\leq1$ and $\phi_r(v)=0$ whenever
	$\|v-v^*\|_2\geq r$, one has
	\begin{gather}\label{mollimass}
		\rho_t\big(B(v^*,r)\big)
		\geq
		\int\phi_r(v)\,d\rho_t(v).
	\end{gather}
	Thus, it suffices to derive a lower bound for the right-hand side.

	Since $\phi_r\in\mathcal{C}_c^\infty(\mathbb{R}^d)$, we can use
	$\phi_r$ as a test function in Definition~\ref{weakdef_iso}. Hence,
	for a.e. $t\in(0,T)$,
	\begin{gather}\label{mollidyn}
		\dfrac{d}{dt}
		\int\phi_r(v)\,d\rho_t(v)
		=
		\int
		\big(
		T_1(v)+T_2(v)+T_3(v)
		\big)\,d\rho_t(v),
	\end{gather}
	where
	\begin{align*}
		T_1(v)
		=
		-\lambda
		\big(v-v_\alpha(\rho_t)\big)
		\cdot\nabla\phi_r(v),\ 
		\HZnew{
		T_2(v)
		=
		\dfrac{\sigma^2}{2}
		\|v-v_\alpha(\rho_t)\|_2^2
		\Delta\phi_r(v),
		}\ 
		\HZnew{
		T_3(v)
		=
		-\dfrac{1}{\epsilon}
		\nabla G(v)\cdot\nabla\phi_r(v).}
	\end{align*}

	For $\|v-v^*\|_2<r$, direct differentiation of
	\eqref{molli} gives
	\begin{gather}\label{gradmolli}
		\nabla\phi_r(v)
		=
		-2r^2
		\dfrac{v-v^*}
		{\big(r^2-\|v-v^*\|_2^2\big)^2}
		\phi_r(v).
	\end{gather}
	Therefore,
	\HZnew{
	\begin{align*}
		T_3(v)
		&=
		\dfrac{2r^2}{\epsilon}
		\dfrac{
		\langle\nabla G(v),v-v^*\rangle
		}{
		\big(r^2-\|v-v^*\|_2^2\big)^2
		}
		\phi_r(v)
		\geq0,
	\end{align*}
	}
	where the last inequality follows from
	Assumption~\ref{wellbehave} (B1).
	For $\|v-v^*\|_2\geq r$, one has
	$\nabla\phi_r(v)=0$, and therefore $T_3(v)=0$.
	Consequently,
	\begin{gather}\label{T3positive}
		T_3(v)\geq0,
		\qquad \forall v\in\mathbb{R}^d.
	\end{gather}

	\HZnew{
	The remaining terms $T_1$ and $T_2$ coincide with those in
	Proposition~4.6 of \cite{FKR} for the choice $H\equiv1$.
	Using the same decomposition of $B(v^*,r)$ into the three regions
	in the proof therein, one obtains
	\begin{gather}\label{T12bound}
		T_1(v)+T_2(v)
		\geq
		-a\phi_r(v),
		\qquad \forall v\in\mathbb{R}^d,
	\end{gather}
	where $a$ is the constant stated in Lemma~\ref{ball}.
	}
	Combining \eqref{mollidyn}, \eqref{T3positive}, and
	\eqref{T12bound}, we obtain
	\begin{align*}
		\dfrac{d}{dt}
		\int\phi_r(v)\,d\rho_t(v)
		&\geq
		-a
		\int\phi_r(v)\,d\rho_t(v).
	\end{align*}
	By Gronwall's inequality, $
		\int\phi_r(v)\,d\rho_t(v)
		\geq
		\left(
		\int\phi_r(v)\,d\rho_0(v)
		\right)e^{-at}.$ Combining this inequality with \eqref{mollimass} finishes the proof.
\end{proof}
\section{Proof of Lemma \ref{G}}\label{Gdetails}
\begin{proof}
\HZnew{Fix $\epsilon>0$ and let $(\rho_t)_{t\in[0,T]}$
	be any weak solution satisfying the bounds in Lemma \ref{G}.
	Throughout the proof, $B$ and $\tilde B$ denote the numbers in the statement of the lemma. We then pick a standard mollifier $\eta\in\mathcal{C}^\infty_{c}(\R^d)$ that satisfies
    \begin{gather*}
        \eta\geq 0,\ \int \eta = 1,\ \mathrm{supp}\,\eta\subseteq B_1(0).
    \end{gather*}
    Also set
    \begin{gather*}
    \eta_{h}(z) = h^{-d}\eta(z/h),\ G_{h} = G * \eta_h.
    \end{gather*}
    We first verify $G_h\in \mathcal{C}^{2}_{*}$ for all $h\in(0,1)$. Clearly, $G_h\in\mathcal{C}^{\infty}(\R^d)\subset \mathcal{C}^{2}(\R^d)$. Also, by the property of convolution, we have $
        \nabla G_h = (\nabla G) * \eta_h.$
    Since $G\in \mathcal{C}^{1,1}_*(\R^d)$. $\nabla G$ is globally Lipschitz. Thus
    \begin{gather*}
    |\partial_k G_h(v)| = \left|\int \eta_h (z) \partial_k G(v-z)\,dz\right|\leq C\int \eta_h(z) (1 + |v_k - z_k|)\,dz\leq C(1 + |v_k|),
    \end{gather*}
    where the first inequality is due to $G\in \mathcal{C}^{1,1}_*(\R^d)$. One further has
    \begin{gather*}
    \|\nabla G_h(v) - \nabla G_h(w)\|_2 \leq \int \eta_h(z) \|\nabla G(v-z) - \nabla G(w-z)\|_2\,dz\leq L_G \|v-w\|_2,
    \end{gather*}
    where the last inequality is because $\nabla G$ is globally Lipschitz. Since $G_h$ is smooth, applying mean-value theorem and the definition of classical derivative on $\partial_k G_h$, we have $
    \sup_{v\in\R^d} |\partial_{kk} G_h(v)|\leq L_G.$ This proves $G_h\in \mathcal{C}^2_*(\R^d)$ for $h\in(0,1)$.}

    \HZnew{As mentioned in Remark~\ref{remark:extension}, the equality in
        Definition~\ref{weakdef_iso} will hold for any test function
        $\phi\in\mathcal{C}_{*}^{2}\big(\mathbb{R}^{d}\big)$.
        Thus we can let $\phi=G_h$ to obtain for $0\leq s<r\leq T$,
        \begin{align*}
        F_h(r)-F_h(s)
        =&-\lambda\int_s^r
        \left\langle
        (v-\con)\cdot\nabla G_h,\rho_t
        \right\rangle\,dt\\
        &-\frac{1}{\epsilon}\int_s^r
        \left\langle
        \nabla G\cdot\nabla G_h,\rho_t
        \right\rangle\,dt+\frac{\sigma^2}{2}\int_s^r
        \left\langle
        \|v-\con\|_2^2\Delta G_h,\rho_t
        \right\rangle\,dt,
        \end{align*}
        where $
        F_h(t):=\int G_h\,d\rho_t.$}
        
        \HZnew{Next, we take $h\rightarrow 0+$. We first show that for the LHS, $F_h(t)$ converges uniformly to $F(t):=\int G\,d\rho_t$ for $t\in [0,T]$. Since $\nabla G$ is globally Lipschitz, there exists absolute constant $M_G,L_G >0$ such that $
        \|\nabla G(v)\|_2\leq M_G + L_G \|v - v^*\|_2.$
        Then one can use fundamental theorem of calculus to obtain
        \begin{align*}
        |G(v-z) - G(v)| &=\left|\int_{0}^1 \nabla G(v - \theta z)\cdot z\,d\theta \right|\\
        &\leq \|z\|_2\left|\int_0^1 \|\nabla G(v-\theta z)\|_2\,d\theta \right|\leq \|z\|_2 \left(M_G + L_G \|v-v^*\|_2 + L_G\|z\|_2\right)
        \end{align*}
        Then
        \begin{align*}
        |G_h(v) - G(v)|&=\left|\int \eta_h(z)(G(v-z) - G(v))\,dz\right|\\
        &\leq \left|\int \eta_h(z) \cdot \|z\|_2\cdot (M_G + L_G \|v-v^*\|_2 + L_G \|z\|_2)\right|\\
        &\leq \left|\int \eta_h(z) \cdot h\cdot (M_G + L_G \|v-v^*\|_2 + L_G h)\right|\\
        &\leq C h(1 + \|v-v^*\|_2) + Ch^2.
        \end{align*}
        Thereby
        \begin{align*}
        |F_h(t) - F(t)| &\leq  \int |G_h(v) - G(v)|\,d\rho_t \\
        &\leq C\int \left(h(1 + \|v-v^*\|_2) +h^2\right)\,d\rho_t\\
        & \leq C (h+h^2) + Ch\int \|v-v^*\|_2\,d\rho_t\\
        &\leq C (h+h^2) + Ch\left(\int \|v-v^*\|_2^2\,d\rho_t\right)^{1/2}\leq C (h+h^2) + Ch\sqrt{2\tilde{B}}.
        \end{align*}
        This implies $
        \sup_{0\leq t\leq T}|F_h(t) - F(t)|\rightarrow 0.$
        }

                \HZnew{We next consider the terms on the RHS. First, since $\nabla G$ is globally $L_G$-Lipschitz, we have
        \begin{align*}
        \|\nabla G_h(v)-\nabla G(v)\|_2
        &=\left\|\int \eta_h(z)\left(\nabla G(v-z)-\nabla G(v)\right)\,dz\right\|_2\\
        &\leq \int \eta_h(z)\|\nabla G(v-z)-\nabla G(v)\|_2\,dz\\
        &\leq L_G\int \eta_h(z)\|z\|_2\,dz
        \leq L_G h,
        \end{align*}
        where the last inequality follows from $\mathrm{supp}\,\eta_h\subseteq B_h(0)$. Therefore, $
        \sup_{v\in\R^d}\|\nabla G_h(v)-\nabla G(v)\|_2\rightarrow 0$
        as $h\rightarrow 0+$. Hence, for the first term on the RHS, one has
        \begin{align*}
        &\left|\int_s^r\left\langle (v-\con)\cdot\left(\nabla G_h-\nabla G\right),\rho_t\right\rangle\,dt\right|\\
        &\qquad\leq L_Gh\int_s^r\int \|v-\con\|_2\,d\rho_t(v)\,dt\\
        &\qquad\leq L_Gh\int_s^r\left(\int\|v-v^*\|_2\,d\rho_t(v)+\|\con-v^*\|_2\right)\,dt\\
        &\qquad\leq L_Gh(r-s)\left(\sqrt{2\tilde B}+B\right)\rightarrow 0.
        \end{align*}
        }

        \HZnew{For the second term, using the same uniform convergence and the linear growth of $\nabla G$, we obtain
        \begin{align*}
        &\left|\int_s^r\left\langle \nabla G\cdot\left(\nabla G_h-\nabla G\right),\rho_t\right\rangle\,dt\right|\\
        &\qquad\leq L_Gh\int_s^r\int\|\nabla G(v)\|_2\,d\rho_t(v)\,dt\\
        &\qquad\leq L_Gh\int_s^r\left(M_G+L_G\int\|v-v^*\|_2\,d\rho_t(v)\right)\,dt\\
        &\qquad\leq L_Gh(r-s)\left(M_G+L_G\sqrt{2\tilde B}\right)\rightarrow 0.
        \end{align*}
        Therefore, $
        \int_s^r\left\langle \nabla G\cdot\nabla G_h,\rho_t\right\rangle\,dt
        \rightarrow
        \int_s^r\left\langle \|\nabla G\|_2^2,\rho_t\right\rangle\,dt.$
        }

        \HZnew{For the diffusion term, since $
        |\partial_{kk}G_h(v)|\leq L_G,
        $ for $ k=1,\ldots,d,$
        we have $
        |\Delta G_h(v)|
        \leq dL_G.$
        Therefore,
        \begin{align*}
        \left\langle
        \|v-\con\|_2^2\Delta G_h,\rho_t
        \right\rangle
        &\leq
        dL_G\int\|v-\con\|_2^2\,d\rho_t(v)\\
        &\leq
        2dL_G\int\|v-v^*\|_2^2\,d\rho_t(v)
        +
        2dL_G\|\con-v^*\|_2^2\\
        &\leq
        dL_G\left(4\tilde B+2B^2\right).
        \end{align*}
        Notice that we only need an upper bound for this term and therefore do
        not need to pass the second derivatives of $G_h$ to a second derivative
        of $G$.}

        \HZnew{Combining the above estimates, using the uniform convergence
$F_h\rightarrow F$, and taking $\limsup_{h\rightarrow0+}$, we obtain
        \begin{align*}
        F(r)-F(s)
        &\leq -\lambda\int_s^r\left\langle (v-\con)\cdot\nabla G,\rho_t\right\rangle\,dt
        -\frac{1}{\epsilon}\int_s^r\left\langle\|\nabla G\|_2^2,\rho_t\right\rangle\,dt\\
        &\quad+\frac{d\sigma^2L_G}{2}\left(4\tilde B+2B^2\right)(r-s).
        \end{align*}
        We decompose the first term on the RHS as
        \begin{align*}
        -\lambda\left\langle (v-\con)\cdot\nabla G,\rho_t\right\rangle=
        -\lambda\left\langle (v-v^*)\cdot\nabla G,\rho_t\right\rangle
        -\lambda\left\langle (v^*-\con)\cdot\nabla G,\rho_t\right\rangle.
        \end{align*}
        The first term above is non-positive due to Assumption~\ref{wellbehave} (B1). For the second term, we have
        \begin{align*}
        -\lambda\left\langle (v^*-\con)\cdot\nabla G,\rho_t\right\rangle
        &\leq \lambda\|\con-v^*\|_2\int\|\nabla G(v)\|_2\,d\rho_t(v)\leq \lambda B\left(M_G+L_G\sqrt{2\tilde B}\right).
        \end{align*}
        }

        \HZnew{Furthermore, by Assumption~\ref{wellbehave} (B2), there exists a constant $C_G>0$ such that $
        G(v)\leq C_G\|\nabla G(v)\|_2^2.$ Therefore, $
        -\frac{1}{\epsilon}\left\langle\|\nabla G\|_2^2,\rho_t\right\rangle
        \leq -\frac{1}{C_G\epsilon}F(t).$ Hence, by defining
        \begin{gather*}
        D:=\lambda B\left(M_G+L_G\sqrt{2\tilde B}\right)+\frac{d\sigma^2L_G}{2}\left(4\tilde B+2B^2\right),
        \end{gather*}
        we obtain for all $0\leq s<r\leq T$,
        \begin{gather*}
        F(r)-F(s)\leq \int_s^r\left(D-\frac{1}{C_G\epsilon}F(t)\right)\,dt.
        \end{gather*}
        }

        \HZnew{Now consider the function $f$ satisfying
        \begin{gather*}
        \frac{d}{dt}f(t)=D-\frac{1}{C_G\epsilon}f(t),
        \qquad
        f(0)=F(0).
        \end{gather*}
        By the comparison theorem, one has $F(t)\leq f(t)$ for all $t\in[0,T]$. A direct calculation gives
        \begin{gather*}
        f(t)
        =C_G\epsilon D+
        \left(F(0)-C_G\epsilon D\right)
        \exp\left(-\frac{t}{C_G\epsilon}\right).
        \end{gather*}
        Consequently,
        \begin{gather*}
        F(t)\leq
        \exp\left(-\frac{t}{C_G\epsilon}\right)F(0)
        +C_G\epsilon D
        \left(1-\exp\left(-\frac{t}{C_G\epsilon}\right)\right),
        \qquad t\in[0,T].
        \end{gather*}
        In particular, $
	F(t)\leq F(0)+C_G\epsilon D,
	$ for $ t\in[0,T].$
If $D=0$, the desired estimate holds for every $\epsilon>0$.
Otherwise, set
\begin{gather}\label{epsilon}
	\epsilon_0=\frac{\xi}{C_GD}.
\end{gather}
Then, for every $0<\epsilon\leq\epsilon_0$, $
	F(t)\leq F(0)+\xi,
	$ for $t\in[0,T].$
\HZnew{Since $D$ depends only on the prescribed bounds
$B,\tilde B$ and the fixed model parameter, $\epsilon_0$ is uniform
over all weak solutions satisfying these bounds.}}
\end{proof}
\HZnew{\section{Proof of Proposition~\ref{prop:implicit_inverse}}\label{appendix:invert}}
\begin{proof}
\HZnew{Since $
    G(v)=\sum_{i=1}^{m}g_i(v)^2,$ one has
\begin{gather}\label{eq:hessian_decomp}
    \nb^{2}G(v)
    =
    2\sum_{i=1}^{m}
    \left[
    \nb g_i(v)\nb g_i(v)^{\top}
    +
    g_i(v)\nb^{2}g_i(v)
    \right].
\end{gather}
For any $z\in\mathbb{R}^{d}$, it follows that
\begin{align*}
    z^{\top}\nb^{2}G(v)z
    &=
    2\sum_{i=1}^{m}
    |\nb g_i(v)^{\top}z|^{2}
    +
    2\sum_{i=1}^{m}
    g_i(v)z^{\top}\nb^{2}g_i(v)z\\
    &\geq
    -2\sum_{i=1}^{m}
    |g_i(v)|
    \left|z^{\top}\nb^{2}g_i(v)z\right|\geq
    -2\sum_{i=1}^{m}
    |g_i(v)|M_i\|z\|_2^{2}\geq
    -2M\|g(v)\|_2\|z\|_2^{2},
\end{align*}
where the second inequality follows from \eqref{eq:hessian_g_bound}, and
the last inequality follows from the Cauchy inequality. Since the above
inequality holds for every $z\in\mathbb{R}^{d}$, we obtain $
    \nb^{2}G(v)\succeq -2M\|g(v)\|_2I,$
which proves \eqref{eq:hessian_lower}.}

\HZnew{It then follows that $
    I+\frac{\g}{\e}\nb^{2}G(v)
    \succeq
    \left(
    1-\frac{2\g}{\e}M\|g(v)\|_2
    \right)I.$ Under condition \eqref{eq:inverse_condition}, the quantity on the
right-hand side is positive. Therefore,
$I+\frac{\g}{\e}\nb^{2}G(v)$ is positive definite and hence invertible.
Furthermore, since this matrix is symmetric, one has
\begin{align*}
    \left\|
    \left[I+\frac{\g}{\e}\nb^{2}G(v)\right]^{-1}
    \right\|_{\mathrm{op}}
    =
    \frac{1}{
    \lambda_{\min}\left(
    I+\frac{\g}{\e}\nb^{2}G(v)
    \right)}\leq
    \frac{1}{
    1-\frac{2\g}{\e}M\|g(v)\|_2},
\end{align*}
which proves \eqref{eq:inverse_bound}.}

\HZnew{Finally, if $\|g(v)\|_2\leq c\e$, then $
    \frac{2\g}{\e}M\|g(v)\|_2
    \leq 2\g Mc.$ Thus, as long as $2\g Mc<1$, condition
\eqref{eq:inverse_condition} is satisfied and $
    \left\|
    \left[I+\frac{\g}{\e}\nb^{2}G(v)\right]^{-1}
    \right\|_{\mathrm{op}}
    \leq
    \frac{1}{1-2\g Mc}.$ The right-hand side does not depend on $\e$, which completes the proof.}
\end{proof}
\HZnew{\section{Proof of Proposition~\ref{prop:consistency}}\label{appendix:consist}}
\begin{proof}
\HZnew{Recall from the proof of Theorem~\ref{wellposedmicro} that the
finite-particle system \eqref{constraincon} can be equivalently written as
\begin{gather*}
    dV_t
    =
    -\lambda F_N(V_t)\,dt
    -\frac{1}{\e}L_N(V_t)\,dt
    +\sigma M_N(V_t)\,dB_t^{(N)},
\end{gather*}
where $V_t\in\R^{Nd}$ and $F_N$, $L_N$, and $M_N$ are defined in
Appendix~\ref{appenmicro}. Assumption~\ref{assump1} (2) implies that
$\CE$ is locally Lipschitz. Therefore, by Lemma~2.1 in \cite{CCTT},
$F_N$ is locally Lipschitz and has at most linear growth. By the
definition of $M_N$, the same properties hold for $M_N$. Moreover,
$L_N$ is globally Lipschitz and has at most linear growth by
Assumption~\ref{assump1} (3).}

\HZnew{We first estimate the additional error introduced by the
linearly implicit scheme. Since $\nabla G$ is globally Lipschitz by
Assumption~\ref{assump1} (3) and $G\in\mathcal{C}^{2}(\R^d)$, there
exists $L_G>0$ such that $
    \|\nabla^2G(v)\|_{\mathrm{op}}\leq L_G$ for all $v\in\R^d$. Define the block diagonal matrix
\begin{gather*}
    H_N(V_0)
    :=
    \operatorname{diag}\left(
    \nabla^2G(V_0^{1,N}),\ldots,\nabla^2G(V_0^{N,N})
    \right)
\end{gather*}
and
\begin{gather*}
    R_{\g}
    :=
    \left(I+\frac{\g}{\e}H_N(V_0)\right)^{-1}.
\end{gather*}
Since $\|H_N(V_0)\|_{\mathrm{op}}\leq L_G$, for $\g$ small enough such that
$\g L_G/\e\leq1/2$, $R_{\g}$ is well defined and $
    \|R_{\g}\|_{\mathrm{op}}\leq2.$ Moreover, by the definition of $R_\g$, $
    R_{\g}-I
    =
    -\frac{\g}{\e}R_{\g}H_N(V_0),$
and hence
\begin{gather}\label{eq:Rgamma}
    \|R_{\g}-I\|_{\mathrm{op}}
    \leq\frac{2L_G}{\e}\g.
\end{gather}}
\HZnew{Let $
    \Delta B_{\g}^{(N)}
    :=
    B_{\g}^{(N)}-B_0^{(N)}.$
We couple the Gaussian random variables used in the numerical scheme
with the Brownian increments by taking
$z^j=-\Delta B_{\g}^{j,N}/\sqrt{\g}$, $j=1,\ldots,N$.
Since $\{z^j\}_{j=1}^N$ are independent standard Gaussian random
variables, this has exactly the distribution required by the numerical
scheme. We introduce the one-step Euler--Maruyama increment
\begin{align*}
    \widetilde{V}_{\g}
    :=
    V_0
    -\lambda\g F_N(V_0)
    -\frac{\g}{\e}L_N(V_0)
    +\sigma M_N(V_0)\Delta B_{\g}^{(N)}.
\end{align*}
And one step of \eqref{eq:implicit1} can be written as
\begin{align*}
    V_{\g}^{\mathrm{num}}-V_0
    =
    R_{\g}
    \left(
    -\lambda\g F_N(V_0)
    -\frac{\g}{\e}L_N(V_0)
    +\sigma M_N(V_0)\Delta B_{\g}^{(N)}
    \right).
\end{align*}
Therefore, using \eqref{eq:Rgamma},
$\E[\Delta B_{\g}^{(N)}]=0$, and
$\E[\Delta B_{\g}^{(N)}(\Delta B_{\g}^{(N)})^T]=\g I$, we obtain
\begin{align*}
    &\E\left[
    \|V_{\g}^{\mathrm{num}}-\widetilde{V}_{\g}\|_2^2
    \right]\\
    &\qquad\leq
    \|R_{\g}-I\|_{\mathrm{op}}^2
    \E\left[
    \left\|
    -\lambda\g F_N(V_0)
    -\frac{\g}{\e}L_N(V_0)
    +\sigma M_N(V_0)\Delta B_{\g}^{(N)}
    \right\|_2^2
    \right]\\
    &\qquad=
    \|R_{\g}-I\|_{\mathrm{op}}^2
    \left[
    \g^2
    \left\|
    \lambda F_N(V_0)+\frac{1}{\e}L_N(V_0)
    \right\|_2^2
    +
    \sigma^2\g\|M_N(V_0)\|_{\mathrm{F}}^2
    \right]\\
    &\qquad\leq C_{\e,N,V_0}\g^3.
\end{align*}}

\HZnew{This implies that the one-step update of~\eqref{eq:implicit1}, namely, $V^{\mathrm{num}}_\g$, approximates the one-step update of Euler--Maruyama scheme, namely, $\widetilde{V}_\ga$. We next compare $\widetilde{V}_{\g}$ with the exact continuous-time solution.
For this purpose, we first establish an estimate for the coefficients.
For $V=(V^1,\ldots,V^N)\in\R^{Nd}$, let $
    S(V):=\sum_{j=1}^N\omega_{\alpha}(V^j),
    $ and $
    Q(V):=\sum_{j=1}^NV^j\omega_{\alpha}(V^j).$
Then $
    F_N^i(V)=V^i-\frac{Q(V)}{S(V)}.$ Since $\CE$ is bounded, there exist constants
$c_{\omega},C_{\omega}>0$ such that $
    c_{\omega}\leq\omega_{\alpha}(v)\leq C_{\omega}$ for all $v\in\R^d$. Moreover, by Assumption~\ref{assump1} (2) and the mean-value theorem
applied to the exponential function, $
|\omega_\alpha(u)-\omega_\alpha(v)|
\leq
C(1+\|u\|_2+\|v\|_2)\|u-v\|_2.$ It follows from Cauchy's inequality that
\begin{align*}
|S(V)-S(W)|
&\leq
C_N(1+\|V\|_2+\|W\|_2)\|V-W\|_2,\\
\|Q(V)-Q(W)\|_2
&\leq
C_N(1+\|V\|_2^2+\|W\|_2^2)\|V-W\|_2.
\end{align*}
Indeed, for the second inequality,
\begin{align*}
    \|Q(V)-Q(W)\|_2
    &\leq
    \sum_{j=1}^N
    \|V^j-W^j\|_2\omega_{\alpha}(V^j)+
    \sum_{j=1}^N
    \|W^j\|_2
    |\omega_{\alpha}(V^j)-\omega_{\alpha}(W^j)|,
\end{align*}
and the desired estimate follows again from Cauchy's inequality.
Since $S(V),S(W)\geq Nc_{\omega}$ and
$\|Q(W)\|_2\leq C_N\|W\|_2$, we obtain
\begin{gather}\label{eq:FNlip}
    \|F_N(V)-F_N(W)\|_2
    \leq
    C_N(1+\|V\|_2^2+\|W\|_2^2)\|V-W\|_2.
\end{gather}
By the definition of $M_N$, the same estimate gives
\begin{gather}\label{eq:MNlip}
    \|M_N(V)-M_N(W)\|_{\mathrm{F}}
    \leq
    C_N(1+\|V\|_2^2+\|W\|_2^2)\|V-W\|_2.
\end{gather}
Finally, Assumption~\ref{assump1} (3) gives
\begin{gather}\label{eq:LNlip}
    \|L_N(V)-L_N(W)\|_2
    \leq L_G\|V-W\|_2.
\end{gather}}

\HZnew{Since $F_N$, $L_N$, and $M_N$ have at most linear growth (as mentioned at the beginning of the proof) and
$V_0$ is fixed, the standard moment estimates for stochastic
differential equations give
\begin{gather}\label{eq:8moment}
    \sup_{0\leq s\leq1}
    \E\left[\|V_s\|_2^8\right]
    \leq C_{\e,N,V_0}.
\end{gather}
This estimate follows in the standard way by applying It\^o's formula
to $(1+\|V_s\|_2^2)^4$, using the linear growth of the drift and
diffusion coefficients, and then applying Gronwall's inequality.
Here we use that $V_0$ is deterministic. The same argument also applies
to a random initial condition $V_0$ provided that
$\E[(1+\|V_0\|_2^2)^4]<\infty$.}

\HZnew{We further prove that
\begin{gather}\label{eq:increment4}
    \E\left[\|V_s-V_0\|_2^4\right]
    \leq C_{\e,N,V_0}s^2,
    \qquad 0\leq s\leq1.
\end{gather}
\HZnew{Using the integral form of \eqref{constraincon}, we have
\begin{align*}
    V_s-V_0
    ={}&
    -\int_0^s
    \left(
    \lambda F_N(V_r)+\frac{1}{\e}L_N(V_r)
    \right)\,dr
    +\sigma\int_0^sM_N(V_r)\,dB_r^{(N)}.
\end{align*}
First, using the inequality
$\|x+y\|_2^4\leq C(\|x\|_2^4+\|y\|_2^4)$, we obtain
\begin{align*}
    \E\left[\|V_s-V_0\|_2^4\right]
    &\leq
    C\E\left[
    \left\|
    \int_0^s
    \left(
    \lambda F_N(V_r)+\frac{1}{\e}L_N(V_r)
    \right)\,dr
    \right\|_2^4
    \right]+
    C\sigma^4\E\left[
    \left\|
    \int_0^sM_N(V_r)\,dB_r^{(N)}
    \right\|_2^4
    \right].
\end{align*}
For the first term, H\"older's inequality gives
\begin{align*}
    \left\|
    \int_0^s
    \left(
    \lambda F_N(V_r)+\frac{1}{\e}L_N(V_r)
    \right)\,dr
    \right\|_2^4
    &\leq
    \left(
    \int_0^s
    \left\|
    \lambda F_N(V_r)+\frac{1}{\e}L_N(V_r)
    \right\|_2\,dr
    \right)^4\\
    &\leq
    s^3\int_0^s
    \left\|
    \lambda F_N(V_r)+\frac{1}{\e}L_N(V_r)
    \right\|_2^4\,dr.
\end{align*}
For the second term, the Burkholder--Davis--Gundy inequality yields
\begin{align*}
    \E\left[
    \left\|
    \int_0^sM_N(V_r)\,dB_r^{(N)}
    \right\|_2^4
    \right]
    &\leq
    C\E\left[
    \left(
    \int_0^s\|M_N(V_r)\|_{\mathrm{F}}^2\,dr
    \right)^2
    \right].
\end{align*}
Combining the above estimates, we obtain
\begin{align*}
    \E\left[\|V_s-V_0\|_2^4\right]
    &\leq
    Cs^3\int_0^s
    \E\left[
    \left\|
    \lambda F_N(V_r)+\frac{1}{\e}L_N(V_r)
    \right\|_2^4
    \right]\,dr+
    C\sigma^4
    \E\left[
    \left(
    \int_0^s\|M_N(V_r)\|_{\mathrm{F}}^2\,dr
    \right)^2
    \right].
\end{align*}}
The last term satisfies, again by H\"older's inequality,
\begin{align*}
    \E\left[
    \left(
    \int_0^s\|M_N(V_r)\|_{\mathrm{F}}^2\,dr
    \right)^2
    \right]
    &\leq
    s\int_0^s
    \E\left[\|M_N(V_r)\|_{\mathrm{F}}^4\right]\,dr.
\end{align*}
Therefore, using the linear growth of $F_N$, $L_N$, and $M_N$,
together with \eqref{eq:8moment}, we obtain
\begin{align*}
    \E\left[\|V_s-V_0\|_2^4\right]
    &\leq
    C_{\e,N,V_0}s^4
    +
    C_{\e,N,V_0}s^2\leq C_{\e,N,V_0}s^2,
\end{align*}
where we used $s\leq1$. This proves \eqref{eq:increment4}.}

\HZnew{Combining \eqref{eq:FNlip} with Cauchy's inequality,
\eqref{eq:8moment}, and \eqref{eq:increment4}, we obtain
\begin{align*}
    &\E\left[
    \|F_N(V_s)-F_N(V_0)\|_2^2
    \right]\\
    &\qquad\leq
    C_N
    \left(
    \E\left[
    (1+\|V_s\|_2^2+\|V_0\|_2^2)^4
    \right]
    \right)^{1/2}
    \left(
    \E\left[
    \|V_s-V_0\|_2^4
    \right]
    \right)^{1/2}\leq C_{\e,N,V_0}\cdot s.
\end{align*}
In the same way, \eqref{eq:MNlip} gives $
    \E\left[
    \|M_N(V_s)-M_N(V_0)\|_{\mathrm{F}}^2
    \right]
    \leq C_{\e,N,V_0}\cdot s.$ Moreover, by \eqref{eq:LNlip} and \eqref{eq:increment4},
\begin{align*}
    \E\left[
    \|L_N(V_s)-L_N(V_0)\|_2^2
    \right]
    &\leq
    L_G^2\E\left[\|V_s-V_0\|_2^2\right]\\
    &\leq
    L_G^2
    \left(
    \E\left[\|V_s-V_0\|_2^4\right]
    \right)^{1/2}\leq C_{\e,N,V_0}\cdot s.
\end{align*}}
\HZnew{Finally, from the integral form of \eqref{constraincon}, one has
\begin{align*}
    V_{\g}-\widetilde{V}_{\g}
    ={}&
    -\lambda\int_0^{\g}
    \left(F_N(V_s)-F_N(V_0)\right)\,ds\\
    &-\frac{1}{\e}\int_0^{\g}
    \left(L_N(V_s)-L_N(V_0)\right)\,ds+\sigma\int_0^{\g}
    \left(M_N(V_s)-M_N(V_0)\right)\,dB_s^{(N)}.
\end{align*}
Therefore, Cauchy's inequality and It\^o's isometry yield
\begin{align*}
    \E\left[
    \|V_{\g}-\widetilde{V}_{\g}\|_2^2
    \right]
    &\leq
    C\g\int_0^{\g}
    \E\left[
    \|F_N(V_s)-F_N(V_0)\|_2^2
    \right]\,ds\\
    &\quad+
    \frac{C\g}{\e^2}\int_0^{\g}
    \E\left[
    \|L_N(V_s)-L_N(V_0)\|_2^2
    \right]\,ds\\
    &\quad+
    C\sigma^2\int_0^{\g}
    \E\left[
    \|M_N(V_s)-M_N(V_0)\|_{\mathrm{F}}^2
    \right]\,ds\\
    &\leq
    C_{\e,N,V_0}\left(\g^3+\g^2\right)\leq C_{\e,N,V_0}\g^2,
\end{align*}
where we used $\g\leq1$. Combining this estimate with $
    \E\left[
    \|V_{\g}^{\mathrm{num}}-\widetilde{V}_{\g}\|_2^2
    \right]
    \leq C_{\e,N,V_0}\g^3$
and applying the triangle inequality in $L^2$ gives $
    \left(
    \E\left[
    \|V_{\g}^{\mathrm{num}}-V_{\g}\|_2^2
    \right]
    \right)^{1/2}
    \leq C_{\e,N,V_0}\cdot \g.$ This completes the proof.}
\end{proof}
\section{Details of the Numerical Experiments}

\subsection{Figures \ref{fig:cbo_obj} and Figure \ref{fig:converge}}\label{appendix:fig0}

\HZnew{The objective function $\mE(v)$ is similar to \eqref{eq:ackley},}
\begin{equation*}
\begin{aligned}
	&\min_v \qd -A\exp\l(-a\sqrt{\frac{b^2}d\ll v - \hv\rl_2^2}\r)- \exp\l(\frac1d\sum_{i=1}^d\cos(2\pi b(v - \hv)_i)\r)+e^1+A,
\end{aligned}
\end{equation*}
with $b = 3, A = 20, a = 0.2$. The circular constraint reads, $
    g_{\mathrm{circ}}(v)=\|v\|_2^2-1,$
and the parabolic constraint reads, $
    g_{\mathrm{para}}(v)=v_1^2-v_2.$

The first case is a circular constraint, and the unconstrained minimizer is the same as the constrained minimizer, $
\hv = v^* = \frac1{\sqrt{2}}(1,-1).$
\HZnew{The second case is a circular constraint, and the unconstrained minimizer is different from the constrained minimizer. In this case,} \HZnew{$
\hv = (1/2,1/3),$ and $
v^* = (0.7817183879,0.6236315916).$}
\HZnew{The third case is a parabolic constraint, and the unconstrained minimizer is different from the constrained minimizer. In this case, $\hv = (1/2,1/3), $ and $
v^* = (0.5427013393,0.2945247437).$} \HZnew{The constrained minimizers in the second and third cases are computed numerically by reducing the constrained problems to one-dimensional optimization problems. More precisely, for the circular constraint we parameterize $v=(\cos\theta,\sin\theta)$, while for the parabolic constraint we set $v=(x,x^2)$, and solve the corresponding first-order optimality conditions to high accuracy.}

We use Algorithm 1 with
\begin{equation}\label{f0 parameter}
N = 50, \ \a = 30, \ \epsilon = 0.01,\  \lam = 1,\ \s = 1,\ \g = 0.01,\  \epsilon_{\text{stop}} = 0.
\end{equation}
We set $\epsilon_{\text{stop}}$ to be $0$ to see the iteration evolves until it reaches $300$ steps.
\HZnew{The particles in the proposed method are initially sampled from
$\operatorname{Unif}([-3,3]^2)$. For each repetition, we generate a common base ensemble
$U^j\sim\operatorname{Unif}([-3,3]^2)$, $j=1,\ldots,N$.
The proposed method, Penalized CBO, and CB2O use these samples
directly. For the projected CBO method, we normalize the same
samples before the first iteration: $
    V_0^j=\frac{U^j}{\|U^j\|_2}.$ All methods are run for $300$ steps, and projected CBO is considered only
for the two circular constraints. We consider the algorithm successful in finding the constrained minimizer $v^{*}$ if the distance between the output consensus point $v_{\alpha}$ and $v^{*}$ satisfies $
\|v^{*}-v_{\alpha}\|_{\infty}\leq 0.01.$ The distance reported in Figures \ref{fig:cbo_obj} and \ref{fig:converge} is measured in terms of \eqref{def of norm}. The success rate is computed over $100$ simulations, and the averaged distance in Figure \ref{fig:cbo_obj} is averaged over all $100$ simulations.}

We use Algorithm 1 in \cite{FHPS1} for the projected CBO method. For the penalized CBO method, we set the penalty as $\frac1\e G(v)$, and then apply the CBO algorithm to the following unconstrained optimization problem, $
\mE^\e(v) = \mE(v) + \frac1\e G(v).$
\HZnew{For the CB2O method \cite{CB2O}, we take $G(v)$ as the lower-level objective and $\mE(v)$ as the upper-level objective. The quantile parameter $\beta$ follows the reference choice used for the constrained optimization examples in \cite{CB2O}. We set $\beta=1/20$ in all our experiments. We use the anisotropic version of CB2O with the same number of particles and the same parameters $\alpha,\lambda,\sigma,$ and $\gamma$ as in \eqref{f0 parameter}.}

\HZnew{For all the comparison methods, the particles follow the same initial distribution Unif$[-3,3]^2$ and the algorithms are run for $300$ steps. The projected CBO method is only considered for the two circular constraints.}

\subsection{\HZnew{The Thomson problem}}
\label{appendix:thomson_diagnostics}

\subsubsection{\HZnew{Blockwise implementation}}
\HZnew{For the Thomson problem, when we fix the first electron, we have $
G(v)=\sum_{i=2}^{k}\left(\|v_i\|_2^2-1\right)^2.$ For each free electron, $
\nabla_{v_i}^2G(v)
=
8v_iv_i^T+
4\left(\|v_i\|_2^2-1\right)I_3.$
Hence, $
I+\frac{\gamma}{\epsilon}\nabla^2G(v)$ is block diagonal with $k-1$ blocks of size $3\times3$. As a result,
the linearly implicit step can be done by solving $k-1$
independent $3\times3$ systems for each particle, rather than one dense
$3(k-1)\times3(k-1)$ system. For example, for $k=470$ and $N=50$, this corresponds to $469$ small systems per particle and $23450$ $3\times3$ systems per iteration.}

\subsubsection{\HZnew{Conditioning of the implicit systems}}
\HZnew{
Let $q=\gamma/\epsilon=10$. For the $i$-th free electron, the implicit $3\times3$
block is $
A_i=I_3+q\nabla_{v_i}^2G(v).$ Let $r_i=\|v_i\|_2$. $A_i$ has the three eigenvalues $
1+4q(r_i^2-1)$, $1+4q(r_i^2-1),$ and $ 1+q(12r_i^2-4).$ In particular, if the electron is on the constraint, namely $r_i=1$, we have
$\kappa_2(A_i)=81$. If all free electrons are on the sphere, the corresponding
$1407\times1407$ block-diagonal system therefore also has condition number $81$.
Away from the constraint, these eigenvalues may approach zero, and thus uniform
conditioning cannot be guaranteed.}

\HZnew{We then examine the conditioning for three representative random initializations.
We use the same algorithm configuration as in Section~\ref{thompson} with
$50$ particles and $2000$ iterations. In total, there are
$3\times50\times2000=3\times10^5$ block-diagonal linear systems of dimension
$1407\times1407$. We report the condition number statistics in
Table~\ref{table:thomson_conditioning}.}

\begin{table}[h]
\centering
\caption{\HZnew{Condition number statistics for three representative
$k=470$ Thomson simulations.}}
\label{table:thomson_conditioning}
\begin{tabular}{|c|c|}
\hline
Quantity & Value\\
\hline
Median & $81.74$\\
$95\%$ quantile & $8.86\times10^2$\\
$99\%$ quantile & $3.64\times10^4$\\
Number of systems with $\sigma_{\min}<10^{-4}$ & $7$\\
Number of systems with $\sigma_{\min}<10^{-5}$ & $0$\\
Worst condition number & $3.08\times10^{11}$\\
\hline
\end{tabular}
\end{table}

\HZnew{These results show that the implicit systems are mostly moderately
conditioned, with the median close to the on-constraint value $81$.
However, rare off-constraint cases can produce much larger condition numbers. Thus, the numerical results do not imply uniform conditioning of the linearly implicit update.}

\subsubsection{\HZnew{Flop Counts}}
\HZnew{Although the conditioning analysis above considers the corresponding
block-diagonal system, the implementation does not form or solve a dense
$3(k-1)\times3(k-1)$ linear system. Instead, the blockwise implicit update
requires $O(N(k-1))$ independent $3\times3$ linear solves per iteration.
We also need to count the evaluation of the Thomson objective function,
which involves pairwise interactions and requires $O(Nk^2)$ interaction
evaluations.}
\subsection{\HZnew{Sensitivity to $\alpha$ and $\epsilon$}}
\label{appendix:sensitivity}

\HZnew{We provide additional details for the sensitivity experiment in
Section~\ref{subsec:sensitivity}. We consider the $20$-dimensional Ackley
function with the spherical constraint in Case~3 and use
\[
N=100,\qquad
\lambda=\sigma=1,\qquad
\gamma=0.1,\qquad
\sigma_{\mathrm{indep}}=0.3,\qquad\e_{\mathrm{indep}}=10^{-5},\qquad \e_{\min}=0.01
\]
The maximum number of iterations is $5000$. All particles are initialized independently from
$\operatorname{Unif}[-3,3]^{20}$. We vary $
\alpha\in\{10,20,30,50,80\},$ $
\epsilon\in\{0.005,0.01,0.02,0.05,0.1\},$ and perform $100$ simulations for each parameter pair. The reference
solution is $
v^*=\frac{1}{\sqrt{20}}(1,\ldots,1).$ As in Section~\ref{ackley}, a simulation is considered successful if $
\|\bxs-v^*\|_\infty\leq0.1,
$
and the reported distance is
$
D(\bxs,v^*)=\frac{1}{\sqrt{20}}\|\bxs-v^*\|_2.
$}
\begin{table}[h]
\centering
\caption{\HZnew{Additional results for the $(\alpha,\epsilon)$ sensitivity
experiment. All quantities are averaged over all $100$ simulations for each
parameter pair.}}
\label{table:sensitivity_details}
\small
\setlength{\tabcolsep}{4.5pt}

\HZnew{\text{(a) Mean distance $D(\bxs,v^*)$}}\\[-0.5ex]
\begin{tabular}{|c|ccccc|}
\hline
$\alpha\backslash\epsilon$
& $0.005$ & $0.01$ & $0.02$ & $0.05$ & $0.1$\\
\hline
$10$ & $0.03376$ & $0.03105$ & $0.03316$ & $0.03276$ & $0.03278$\\
$20$ & $0.02099$ & $0.02128$ & $0.02012$ & $0.02073$ & $0.02185$\\
$30$ & $0.01799$ & $0.01797$ & $0.01761$ & $0.01831$ & $0.01841$\\
$50$ & $0.01538$ & $0.01565$ & $0.01576$ & $0.01657$ & $0.01662$\\
$80$ & $0.01467$ & $0.01453$ & $0.01463$ & $0.01528$ & $0.01587$\\
\hline
\end{tabular}

\vspace{1ex}

\HZnew{\text{(b) Mean constraint violation $\sqrt{G(\bxs)}$}}\\[-0.5ex]
\begin{tabular}{|c|ccccc|}
\hline
$\alpha\backslash\epsilon$
& $0.005$ & $0.01$ & $0.02$ & $0.05$ & $0.1$\\
\hline
$10$ & $1.22\!\times\!10^{-4}$ & $1.12\!\times\!10^{-4}$
     & $1.25\!\times\!10^{-4}$ & $1.09\!\times\!10^{-4}$
     & $2.11\!\times\!10^{-4}$\\
$20$ & $1.03\!\times\!10^{-4}$ & $1.08\!\times\!10^{-4}$
     & $1.04\!\times\!10^{-4}$ & $1.43\!\times\!10^{-4}$
     & $1.15\!\times\!10^{-4}$\\
$30$ & $9.42\!\times\!10^{-5}$ & $9.21\!\times\!10^{-5}$
     & $1.52\!\times\!10^{-4}$ & $1.02\!\times\!10^{-4}$
     & $1.17\!\times\!10^{-4}$\\
$50$ & $9.75\!\times\!10^{-5}$ & $1.45\!\times\!10^{-4}$
     & $1.01\!\times\!10^{-4}$ & $2.01\!\times\!10^{-4}$
     & $1.09\!\times\!10^{-4}$\\
$80$ & $1.15\!\times\!10^{-4}$ & $9.22\!\times\!10^{-5}$
     & $1.23\!\times\!10^{-4}$ & $1.66\!\times\!10^{-4}$
     & $1.50\!\times\!10^{-4}$\\
\hline
\end{tabular}
\end{table}

\HZnew{The distances show the same behavior as the success rates in
Table~\ref{table:sensitivity}: the main improvement occurs when $\alpha$
is increased from $10$ to $20$, while the dependence on $\epsilon$ becomes
much weaker once $\alpha$ is sufficiently large. The terminal constraint
violation remains small over all tested parameter pairs.}

\HZnew{\subsection{Computational cost comparison}
\label{appendix:runtime}}

\HZnew{We provide additional details in Table~\ref{table:runtime_details} for the computational cost comparison in Section~\ref{subsec:runtime}.}
\begin{table}[h]
\centering
\caption{\HZnew{Per-iteration computational cost. The three
$T_{\mathrm{iter}}$ columns report the median time per iteration over all
$100$ simulations. The last two columns report the medians of the $100$
paired per-run overhead ratios.}}
\label{table:runtime_details}
\scriptsize
\setlength{\tabcolsep}{2.7pt}
\begin{tabular}{|c|ccc|cc|}
\hline
& \multicolumn{3}{c|}{Median $T_{\mathrm{iter}}$ (s)}
& \multicolumn{2}{c|}{$T_{\mathrm{iter}}$ Ratio}\\
\cline{2-6}
Problem
& Proposed & Penalized CBO & CB2O
& Proposed/Penalized & Proposed/CB2O\\
\hline
Section~\ref{sec:eq1} Segment
& $1.48\!\times10^{-4}$ & $3.76\!\times10^{-5}$ & $4.12\!\times10^{-5}$
& $3.96$ & $3.57$\\
Section~\ref{sec:eq1} Ellipse
& $2.01\!\times10^{-4}$ & $2.82\!\times10^{-5}$ & $3.42\!\times10^{-5}$
& $7.07$ & $5.88$\\
Section~\ref{sec:eq1} Line
& $1.97\!\times10^{-4}$ & $2.84\!\times10^{-5}$ & $3.43\!\times10^{-5}$
& $6.87$ & $5.73$\\
\hline
Section~\ref{ackley} Fig.~\ref{fig: eq2_er}(a)
& $3.11\!\times10^{-4}$ & $8.22\!\times10^{-5}$ & $6.49\!\times10^{-5}$
& $3.80$ & $4.80$\\
Section~\ref{ackley} Fig.~\ref{fig: eq2_er}(b)
& $4.15\!\times10^{-4}$ & $7.19\!\times10^{-5}$ & $5.84\!\times10^{-5}$
& $5.80$ & $7.15$\\
Section~\ref{ackley} Fig.~\ref{fig: eq2_er}(c)
& $4.26\!\times10^{-4}$ & $7.65\!\times10^{-5}$ & $6.08\!\times10^{-5}$
& $5.56$ & $7.04$\\
Section~\ref{ackley} Fig.~\ref{fig: eq2_er}(d)
& $3.04\!\times10^{-4}$ & $8.20\!\times10^{-5}$ & $6.36\!\times10^{-5}$
& $3.70$ & $4.78$\\
Section~\ref{ackley} Fig.~\ref{fig: eq2_er}(e)
& $9.45\!\times10^{-4}$ & $1.45\!\times10^{-4}$ & $1.09\!\times10^{-4}$
& $6.50$ & $8.69$\\
Section~\ref{ackley} Fig.~\ref{fig: eq2_er}(f)
& $1.30\!\times10^{-3}$ & $1.03\!\times10^{-4}$ & $8.52\!\times10^{-5}$
& $12.76$ & $15.30$\\
Section~\ref{ackley} Fig.~\ref{fig: eq2_er}(g)
& $1.35\!\times10^{-3}$ & $1.15\!\times10^{-4}$ & $1.01\!\times10^{-4}$
& $11.87$ & $14.08$\\
\hline
\end{tabular}
\end{table}
\section*{Acknowledgements}
JAC was supported by the Advanced Grant Nonlocal-CPD (Nonlocal PDEs for Complex Particle Dynamics: Phase Transitions, Patterns and Synchronization) of the European Research Council Executive Agency (ERC) under the European Union’s Horizon 2020 research and innovation programme (grant agreement No. 883363).
JAC was also partially supported by the Engineering and Physical Sciences Research Council (EPSRC) under grants EP/T022132/1 and EP/V051121/1.

SJ was  partially supported by funding XDA25010404, the NSFC grants Nos.  12350710181 and  92270001, the Shanghai Municipal Science and Technology Major Projects (Nos. 2021SHZDZX0102 and 23JC1402300), the Shanghai
Municipal Science and Technology Key Project (No. 22JC1401500),  and the Fundamental Research
Funds for the Central Universities.

YZ was supported by the NSF grants No 2529107.

This work was started while JAC and YZ were attending a workshop at the American Institute of Mathematics in 2018 and then developed
further while JAC and YZ were visiting the Simons Institute to participate in the program “Geometric Methods in Optimization and Sampling”
during the Fall of 2021.

\bibliographystyle{spmpsci}      %
\bibliography{references}   %

@InProceedings{S,
	author={Sznitman, A.},
	
    title={Topics in propagation of chaos},
	booktitle={Ecole d'Et{\'e} de Probabilit{\'e}s de Saint-Flour XIX --- 1989},
	year={1991},
	publisher={Springer Berlin, Heidelberg},
	pages={165--251},
	isbn={978-3-540-46319-1}
}

@article{GHV,
  title={Mean-field limits for Consensus-Based Optimization and Sampling},
  author={Gerber, N. J. and Hoffmann, F. and Vaes, U.},
  journal={arXiv preprint arXiv:2312.07373},
  year={2023}
}

@book{A,
	title={Stochastic Differential Equations: Theory and Applications},
	author={Arnold, L.},
	isbn={9780471033592},
	lccn={73022256},
	year={1974},
	publisher={Wiley, New York}
}

@article{CCTT,
	title={An analytical framework for consensus-based global optimization method},
	author={Carrillo, J. A. and Choi, Y.-P. and Totzeck, C. and Tse, O.},
	journal={Mathematical Models and Methods in Applied Sciences},
	volume={28},
	number={06},
	pages={1037--1066},
	year={2018},
	publisher={World Scientific}
}

@article{TPBS,
 author = {Totzeck, C. and Pinnau, R. and Blauth, S. and 
 Schotth{\"o}fer, S.},
 year = {2018},
 month = {12},
 pages = {1-2},
 title = {A Numerical Comparison of Consensus‐Based Global 
 Optimization to other Particle‐based Global Optimization Schemes},
 volume = {18},
 journal = {Proceedings in Applied Mathematics and Mechanics},
 doi = {10.1002/pamm.201800291},

}

@article{PTTM,
	title={A consensus-based model for global optimization and its mean-field limit},
	author={Pinnau, R. and Totzeck, C. and Tse, O. and Martin, S.},
	journal={Mathematical Models and Methods in Applied Sciences},
	volume={27},
	number={01},
	pages={183--204},
	year={2017},
	publisher={World Scientific}
}

@article{FHPS1,
	title={Consensus-based optimization on the sphere: Convergence to global minimizers and machine learning},
	author={Fornasier, M. and Huang, H. and Pareschi, L. and S{\"u}nnen, P.},
	journal={The Journal of Machine Learning Research},
	volume={22},
	number={1},
	pages={10722--10776},
	year={2021},
	publisher={JMLRORG}
}

@article{BHP,
	title={Constrained consensus-based optimization},
	author={Borghi, G. and Herty, M. and Pareschi, L.},
	journal={SIAM Journal on Optimization},
	volume={33},
	number={1},
	pages={211--236},
	year={2023},
	publisher={SIAM}
}

@book{D,
	title={Stochastic Calculus: A Practical Introduction},
	author={Durrett, R.},
	isbn={9780849380716},
	lccn={96024642},
	series={Probability and Stochastics Series},
	year={1996},
	publisher={CRC Press}
}

@article{FHPS2,
	title={Consensus-based optimization on hypersurfaces: Well-posedness and mean-field limit},
	author={Fornasier, M. and Huang, H. and Pareschi, L. and S{\"u}nnen, P.},
	journal={Mathematical Models and Methods in Applied Sciences},
	volume={30},
	number={14},
	pages={2725--2751},
	year={2020},
	publisher={World Scientific}
}

@article{FHPS3,
	title={Anisotropic diffusion in consensus-based optimization on the sphere},
	author={Fornasier, M. and Huang, H. and Pareschi, L. and S{\"u}nnen, P.},
	journal={SIAM Journal on Optimization},
	volume={32},
	number={3},
	pages={1984--2012},
	year={2022},
	publisher={SIAM}
}

@article{HQ,
	title={On the mean-field limit for the consensus-based optimization},
	author={Huang, H. and Qiu, J.},
	journal={Mathematical Methods in the Applied Sciences},
	volume={45},
	number={12},
	pages={7814--7831},
	year={2022},
	publisher={Wiley Online Library}
}

@article{FKR,
  title={Consensus-based optimization methods converge globally},
  author={Fornasier, M. and Klock, T. and Riedl, K.},
  journal={SIAM Journal on Optimization},
  year={2024}
}

@article{bailo2024cbx, doi = {10.21105/joss.06611}, url = {https://doi.org/10.21105/joss.06611}, year = {2024}, publisher = {The Open Journal}, volume = {9}, number = {98}, pages = {6611}, author = {Rafael Bailo and Alethea Barbaro and Susana N. Gomes and Konstantin Riedl and Tim Roith and Claudia Totzeck and Urbain Vaes}, title = {CBX: Python and Julia Packages for Consensus-Based Interacting Particle Methods}, journal = {Journal of Open Source Software} }

@article{CJLZ,
	title={A consensus-based global optimization method for high dimensional machine learning problems},
	author={Carrillo, J. A. and Jin, S. and Li, L. and Zhu, Y.},
	journal={ESAIM: Control, Optimisation and Calculus of Variations},
	volume={27},
	pages={S5},
	year={2021},
	publisher={EDP Sciences}
}

@book{BO,
	title={Advanced mathematical methods for scientists and engineers I: Asymptotic methods and perturbation theory},
	author={Bender, C. M. and Orszag, S. A.},
	volume={1},
	year={1999},
	publisher={Springer, New York}
}

@book{M,
	title={Applied asymptotic analysis},
	author={Miller, P. D.},
	volume={75},
	year={2006},
	publisher={American Mathematical Soc.}
}

@book{DZ,
	title={Large deviations techniques and applications},
	author={Dembo, A. and Zeitouni, O.},
	volume={38},
	year={2009},
	publisher={Springer Berlin, Heidelberg}
}

@Article{CJL,
author = {Chen, J. and Jin, S. and Lyu, L.},
title = {A Consensus-Based Global Optimization Method with Adaptive Momentum Estimation},
journal = {Communications in Computational Physics},
year = {2022},
volume = {31},
number = {4},
pages = {1296--1316},
issn = {1991-7120},
doi = {https://doi.org/10.4208/cicp.OA-2021-0144}
}

@Article{TW,
title = {Consensus-based global optimization with personal best},
journal = {Mathematical Biosciences and Engineering},
volume = {17},
number = {5},
pages = {6026-6044},
year = {2020},
issn = {1551-0018},
doi = {10.3934/mbe.2020320},
url = {https://www.aimspress.com/article/doi/10.3934/mbe.2020320},
author = {C. Totzeck and  M.-T. Wolfram},
}

@article{HJK2,
	title={Convergence and error estimates for time-discrete consensus-based optimization algorithms},
	author={Ha, S.-Y. and Jin, S. and Kim, D.},
	journal={Numerische Mathematik},
	volume={147},
	pages={255--282},
	year={2021},
	publisher={Springer}
}

@incollection{CTV,
	title={Consensus-based optimization and ensemble Kalman inversion for global optimization problems with constraints},
	author={Carrillo, J. A. and Totzeck, C. and Vaes, U.},
	booktitle={Modeling and Simulation for Collective Dynamics},
	pages={195--230},
	year={2023},
	publisher={World Scientific}
}

@article{GP,
	title={From particle swarm optimization to consensus based optimization: stochastic modeling and mean-field limit},
	author={Grassi, S. and Pareschi, L.},
	journal={Mathematical Models and Methods in Applied Sciences},
	volume={31},
	number={08},
	pages={1625--1657},
	year={2021},
	publisher={World Scientific}
}

@book{O,
	title={Stochastic differential equations: an introduction with applications},
	author={Oksendal, B.},
	year={2013},
	publisher={Springer Berlin, Heidelberg}
}

@incollection{GHPQ,
	title={Mean-field particle swarm optimization},
	author={Grassi, S. and Huang, H. and Pareschi, L. and Qiu, J.},
	booktitle={Modeling and Simulation for Collective Dynamics},
	pages={127--193},
	year={2023},
	publisher={World Scientific}
}

@incollection{T,
	title={Trends in consensus-based optimization},
	author={Totzeck, C.},
	booktitle={Active Particles, Volume 3: Advances in Theory, Models, and Applications},
	pages={201--226},
	year={2021},
	publisher={Springer}
}

@article{HQR,
author = {Huang, Hui and Qiu, Jinniao and Riedl, Konstantin},
title = {Consensus-Based Optimization for Saddle Point Problems},
journal = {SIAM Journal on Control and Optimization},
volume = {62},
number = {2},
pages = {1093-1121},
year = {2024},
doi = {10.1137/22M1543367},

URL = { 
    
        https://doi.org/10.1137/22M1543367
    
    

},
eprint = { 
    
        https://doi.org/10.1137/22M1543367
    
    

}
}

@article{RKGF,
	title={Gradient is All You Need?},
	author={Riedl, K. and Klock, T. and Geldhauser, C. and Fornasier, M.},
	journal={arXiv preprint arXiv:2306.09778},
	year={2023}
}

@article{RK, 
title={Leveraging memory effects and gradient information in consensus-based optimisation: On global convergence in mean-field law}, DOI={10.1017/S0956792523000293}, journal={European Journal of Applied Mathematics}, author={Riedl, K.}, year={2023}, pages={1–32}
}

@article{H,
	title={Which neural net architectures give rise to exploding and vanishing gradients?},
	author={Hanin, B.},
	journal={Advances in neural information processing systems},
	volume={31},
	year={2018}
}

@article{FRRS, title={Consensus-based optimisation with truncated noise}, volume={36}, DOI={10.1017/S095679252400007X}, number={2}, journal={European Journal of Applied Mathematics}, author={Fornasier, Massimo and Richtárik, Peter and Riedl, Konstantin and Sun, Lukang}, year={2025}, pages={292–315}}

@InProceedings{FKR1,
author="Fornasier, M.
and Klock, T.
and Riedl, K.",
editor="Jim{\'e}nez Laredo, Juan Luis
and Hidalgo, J. Ignacio
and Babaagba, Kehinde Oluwatoyin",
title="Convergence of Anisotropic Consensus-Based Optimization in Mean-Field Law",
booktitle="Applications of Evolutionary Computation",
year="2022",
publisher="Springer International Publishing",
pages="738--754",
isbn="978-3-031-02462-7"
}

@article{MNS,
title = {Facility location and supply chain management – A review},
journal = {European Journal of Operational Research},
volume = {196},
number = {2},
pages = {401-412},
year = {2009},
issn = {0377-2217},
doi = {https://doi.org/10.1016/j.ejor.2008.05.007},
url = {https://www.sciencedirect.com/science/article/pii/S0377221708004104},
author = {M.T. Melo and S. Nickel and F. Saldanha-da-Gama},

}

@article{CSTCX,
title = {A review of optimization techniques in spacecraft flight trajectory design},
journal = {Progress in Aerospace Sciences},
volume = {109},
pages = {100543},
year = {2019},
issn = {0376-0421},
doi = {https://doi.org/10.1016/j.paerosci.2019.05.003},
url = {https://www.sciencedirect.com/science/article/pii/S037604211830191X},
author = {R. Chai and A. Savvaris, A. Tsourdos, S. Chai and Y. Xia},
}

@article{YCLG,
title = {Review of uncertainty-based multidisciplinary design optimization methods for aerospace vehicles},
journal = {Progress in Aerospace Sciences},
volume = {47},
number = {6},
pages = {450-479},
year = {2011},
issn = {0376-0421},
doi = {https://doi.org/10.1016/j.paerosci.2011.05.001},
url = {https://www.sciencedirect.com/science/article/pii/S0376042111000340},
author = {Yao, W. and Chen, X. and Luo, W. and {van Tooren}, M. and Guo, J.},
}

@article{GV,
author = {Grandhi, R. V. and Venkayya, V. B.},
title = {Structural optimization with frequency constraints},
journal = {AIAA Journal},
volume = {26},
number = {7},
pages = {858-866},
year = {1988},
doi = {10.2514/3.9979},

URL = { 
    
        https://doi.org/10.2514/3.9979
    
    

},
eprint = { 
    
        https://doi.org/10.2514/3.9979
    
    

}

}

@article{RTy,
author = {Rockafellar, R. T.},
title = {Lagrange Multipliers and Optimality},
journal = {SIAM Review},
volume = {35},
number = {2},
pages = {183-238},
year = {1993},
doi = {10.1137/1035044},

URL = { 
    
        https://doi.org/10.1137/1035044
    
    

},
eprint = { 
    
        https://doi.org/10.1137/1035044
    
    

}
,
}

@INPROCEEDINGS{WO,
  author={Wei, E. and Ozdaglar, A.},
  booktitle={2012 IEEE 51st IEEE Conference on Decision and Control (CDC)}, 
  title={Distributed Alternating Direction Method of Multipliers}, 
  year={2012},
  volume={},
  number={},
  pages={5445-5450},
  doi={10.1109/CDC.2012.6425904}}

@article{BC3,
author = {Bostan, Miha\"{\i} and Carrillo, Jos\'{e} Antonio},
title = {Fluid models with phase transition for kinetic equations in swarming},
journal = {Mathematical Models and Methods in Applied Sciences},
volume = {30},
number = {10},
pages = {2023-2065},
year = {2020},
doi = {10.1142/S0218202520400163},

URL = { 
    
        https://doi.org/10.1142/S0218202520400163
    
    

},
eprint = { 
    
        https://doi.org/10.1142/S0218202520400163
    
    

}
}

@article{BC1,
author = {Bostan, Mihai and Carrillo, Jose Antonio},
title = {ASYMPTOTIC FIXED-SPEED REDUCED DYNAMICS FOR KINETIC EQUATIONS IN SWARMING},
journal = {Mathematical Models and Methods in Applied Sciences},
volume = {23},
number = {13},
pages = {2353-2393},
year = {2013},
doi = {10.1142/S0218202513500346},

URL = { 
    
        https://doi.org/10.1142/S0218202513500346
    
    

},
eprint = { 
    
        https://doi.org/10.1142/S0218202513500346
    
    

}
}

@article{BC2,
author = {Bostan, Mihai and Carrillo, Jose Antonio},
title = {Reduced fluid models for self-propelled particles interacting through alignment},
journal = {Mathematical Models and Methods in Applied Sciences},
volume = {27},
number = {07},
pages = {1255-1299},
year = {2017},
doi = {10.1142/S0218202517400152},

URL = { 
    
        https://doi.org/10.1142/S0218202517400152
    
    

},
eprint = { 
    
        https://doi.org/10.1142/S0218202517400152
    
    

}
}

@Article{ABCD,
title = {Hydrodynamic limits for kinetic flocking models of Cucker-Smale type},
journal = {Mathematical Biosciences and Engineering},
volume = {16},
number = {6},
pages = {7883-7910},
year = {2019},
issn = {1551-0018},
doi = {10.3934/mbe.2019396},
url = {https://www.aimspress.com/article/doi/10.3934/mbe.2019396},
author = {Pedro Aceves-Sánchez and  Mihai Bostan and  Jose-Antonio Carrillo and  Pierre Degond},
}

@article{albers2019ensemble,
  title={Ensemble Kalman methods with constraints},
  author={Albers, David J and Blancquart, Paul-Adrien and Levine, Matthew E and Seylabi, Elnaz Esmaeilzadeh and Stuart, Andrew},
  journal={Inverse Problems},
  volume={35},
  number={9},
  pages={095007},
  year={2019},
  publisher={IOP Publishing}
}

@book{bauschke2017,
  title={Convex Analysis and Monotone Operator Theory in Hilbert Spaces. 2nd Edition},
  author={Bauschke, Heinz H and Combettes, Patrick L},
  year={2017},
  publisher={Springer}
}

@inproceedings{lotz2020concentration,
  title={Concentration of the intrinsic volumes of a convex body},
  author={Lotz, Martin and McCoy, Michael B and Nourdin, Ivan and Peccati, Giovanni and Tropp, Joel A},
  booktitle={Geometric Aspects of Functional Analysis: Israel Seminar (GAFA) 2017-2019 Volume II},
  pages={139--167},
  year={2020},
  organization={Springer}
}

@article{rataj2010volume,
  title={On volume and surface area of parallel sets},
  author={Rataj, Jan and Winter, Steffen},
  journal={Indiana University mathematics journal},
  pages={1661--1685},
  year={2010},
  publisher={JSTOR}
}

@book{portfolio_book, place={Cambridge}, title={Portfolio Optimization: Theory and Application}, publisher={Cambridge University Press}, author={Palomar, Daniel P.}, year={2025}}

@article{Markowitz,
 ISSN = {00221082, 15406261},
 URL = {http://www.jstor.org/stable/2975974},
 author = {Harry Markowitz},
 journal = {The Journal of Finance},
 number = {1},
 pages = {77--91},
 publisher = {[American Finance Association, Wiley]},
 title = {Portfolio Selection},
 urldate = {2026-03-09},
 volume = {7},
 year = {1952}
}

@inproceedings{Kast,
  title={VaR and optimization},
  author={Kast, R and Luciano, E and Peccati, L},
  booktitle={2nd International Workshop on Preferences and Decisions, Trento, July},
  volume={1},
  number={3},
  pages={1998},
  year={1998}
}

@inproceedings{mausser1999beyond,
  title={Beyond VaR: from measuring risk to managing risk},
  author={Mausser, Helmut and Rosen, Dan},
  booktitle={Proceedings of the IEEE/IAFE 1999 Conference on Computational Intelligence for Financial Engineering (CIFEr)(IEEE Cat. No. 99TH8408)},
  pages={163--178},
  year={1999},
  organization={IEEE}
}

@article{Rockafellar,
  title={Optimization of conditional value-at-risk},
  author={Rockafellar, R Tyrrell and Uryasev, Stanislav and others},
  journal={Journal of risk},
  volume={2},
  pages={21--42},
  year={2000}
}

@article{CB2O,
  title={CB$^2$O: Consensus-Based Bi-Level Optimization},
  author={Garc{\'i}a Trillos, Nicol{\'a}s and Li, Sixu and Riedl, Konstantin and Zhu, Yuhua},
  journal={arXiv preprint arXiv:2411.13394},
  year={2024}
}

@article{chao2026convergence,
  title={Convergence of Consensus-Based Particle Methods for Nonconvex Bi-Level Optimization},
  author={Chao, Yutong and Sun, Xudong and Riedl, Konstantin and Khadiv, Majid and Etesami, Jalal},
  journal={arXiv preprint arXiv:2605.19667},
  year={2026}
}

@article{trevisan2016well,
  title={Well-posedness of multidimensional diffusion processes with weakly differentiable coefficients},
  author={Trevisan, Dario},
  year={2016}
}

@book{kallenberg2002foundations,
  title     = {Foundations of Modern Probability},
  author    = {Kallenberg, Olav},
  edition   = {2nd},
  year      = {2002},
  publisher = {Springer-Verlag},
  address   = {New York},
  isbn      = {9978-1-4757-4015-8},
  pagetotal = {638},
  series    = {Probability and Its Applications}
}

@book{EK,
  author    = {Ethier, Stewart N. and Kurtz, Thomas G.},
  title     = {Markov Processes: Characterization and Convergence},
  publisher = {John Wiley \& Sons},
  year      = {1986}
}

@article{HKS,
  title={Faithful global convergence for the rescaled consensus-based optimization},
  author={Huang, Hui and Kouhkouh, Hicham and Sun, Lukang},
  journal={arXiv preprint arXiv:2503.08578},
  year={2025}
}

\end{document}